\documentclass[english,reqno]{amsart}

\usepackage{amsmath, appendix}
\usepackage{amssymb, color}
\usepackage[margin=1.05in, marginparwidth=0.8in]{geometry}
\usepackage{mathrsfs}
\usepackage{bbm}
\usepackage{graphicx}
\usepackage{float}
\usepackage{epsf}
\usepackage{titletoc}
\usepackage{subcaption}
\usepackage{enumitem}

\usepackage{subeqnarray}
\usepackage{cases}

\usepackage{mdframed}
\usepackage{xcolor}

\usepackage{hyperref}
\hypersetup{
	final,
    colorlinks=true,
    linkcolor=blue,
    filecolor=magenta,
    urlcolor=cyan,
}

\usepackage{graphicx} 
\usepackage{tikz}
\usetikzlibrary{decorations.pathreplacing}
\usetikzlibrary{fadings}
\usepackage{tikz-cd}
	\usetikzlibrary{arrows,intersections,calc,patterns,angles,quotes}
	\usetikzlibrary{shapes.geometric}
	\definecolor{Blau}{RGB}{0,101,189} % Pantone 300
	\definecolor{BlauDunkel}{RGB}{0,82,147} % Pantone 301
	\definecolor{BlauHell}{RGB}{152,198,234} % Pantone 283
	\definecolor{BlauMittel}{RGB}{100,160,200} % Pantone 542
	\definecolor{Elfenbein}{RGB}{218,215,203} % Pantone 7527 -Elfenbein
	\definecolor{Gruen}{RGB}{162,173,0} % Pantone 383 - Gruen
	\definecolor{Orange}{RGB}{227,114,34} % Pantone 158 - Orange
	\definecolor{Grau}{gray}{0.6} % Grau 60%

\usepackage{algorithm}
\usepackage{algpseudocode}

\AtBeginDocument{%
   \def\MR#1{}
   \def\note#1{}
}

\numberwithin{equation}{section}

\newtheorem{lem}{Lemma}[section]
\newtheorem{thm}{Theorem}[section]

\newtheorem{remark}[thm]{Remark}
\newtheorem{assum}[thm]{Assumption}

\theoremstyle{remark}

\newcommand{\Id}{{\mathrm{Id}}}

\renewcommand{\tilde}{\widetilde}

\newcommand{\CH}{{\mathcal{H}}}

\newcommand{\mc}[1]{\mathcal{#1}}

\newcommand{\EE}{\mathbb{E}}
\newcommand{\RR}{\mathbb{R}}
\newcommand{\NN}{\mathbb{N}}
\newcommand{\PP}{\mathbb{P}}

\newcommand{\OV}{\overline{V}}
\newcommand{\OX}{\overline{X}}

\newcommand{\OY}{\overline{Y}}

\newcommand{\TE}{\mathcal{E}}

\newcommand{\CP}{\mathcal{P}}
\newcommand{\CB}{\mathcal{B}}
\newcommand{\CC}{\mathcal{C}}
\newcommand{\CT}{\mathcal{T}}

\DeclareMathOperator*{\relu}{ReLU}
\providecommand{\Id}{\mathrm{Id}}

\newcommand{\abs}[1]{\left\lvert#1\right\rvert}
\newcommand{\absnormal}[1]{\lvert#1\rvert}
\newcommand{\absbig}[1]{\big\lvert#1\big\rvert}
\newcommand{\norm}[1]{\left\lVert#1 \right\rVert}
\newcommand{\normnormal}[1]{\lVert#1 \rVert}

\DeclareMathOperator\diag{diag}

\providecommand{\argmin}{\operatorname*{\arg\min}}
\providecommand{\supp}{\operatorname*{\mathrm{supp}}}

\newif\ifunsure

\unsuretrue
\newif\ifrevised

\revisedtrue
\usepackage[foot]{amsaddr}
\author{Hui Huang}
\email{hui.huang1@ucalgary.ca}
\author{Jinniao Qiu}
\email{jinniao.qiu@ucalgary.ca}
\address[Hui Huang, Jinniao Qiu]{Department of Mathematics and Statistics, University of Calgary, Calgary, Canada}

\author{Konstantin Riedl}
\address[Konstantin Riedl]{Department of Mathematics, Technical University of Munich, Munich, Germany}
\email{konstantin.riedl@ma.tum.de}

\begin{document}
\title[Global Convergence of Particle Swarm Optimization]{On the Global Convergence of \\Particle Swarm Optimization Methods}
% On the Global Convergence of \\Particle Swarm Optimization Methods
% Convergence Analysis of Particle Swarm Optimization Methods
\maketitle
\begin{abstract}
	\noindent
%	Particle Swarm Optimization is a renowned multi-agent metaheuristic optimization algorithm capable of minimizing nonsmooth and nonconvex functions in high dimensions without the use of gradient information. %27 words
	In this paper we provide a rigorous convergence analysis for the renowned particle swarm optimization method by using tools from stochastic calculus and the analysis of partial differential equations. %30 words
	Based on a time-continuous formulation of the particle dynamics as a system of stochastic differential equations, we establish convergence to a global minimizer of a possibly nonconvex and nonsmooth objective function in two steps. %34 words
	First, we prove consensus formation of an associated mean-field dynamics by analyzing the time-evolution of the variance of the particle distribution. %23 words
	We then show that this consensus is close to a global minimizer by employing the asymptotic Laplace principle and a tractability condition on the energy landscape of the objective function. %30 words
	These results allow for the usage of memory mechanisms, and hold for a rich class of objectives provided certain conditions of well-preparation of the hyperparameters and the initial datum. %29 words
	In a second step, at least for the case without memory effects, we provide a quantitative result about the mean-field approximation  of particle swarm optimization, which specifies the convergence of the interacting particle system to the associated mean-field limit. %41 words
	Combining these two results allows for global convergence guarantees of the numerical particle swarm optimization method with provable polynomial complexity. %20 words
	To demonstrate the applicability of the method we propose an efficient and parallelizable implementation, which is tested in particular on a competitive and well-understood high-dimensional benchmark problem in machine learning. %31 words
\end{abstract}\vspace{0.4cm}

\vspace{0.2cm}

\vspace{0.4cm}

%%%%%%%%%%%%%%%%%%%%%%%%%%%%%%%%%%%%%%%%
%%%%%%%%%%%%%%%%%%%%%%%%%%%%%%%%%%%%%%%%
%%%%%%%%%%%%%%%%%%%%%%%%%%%%%%%%%%%%%%%%
\section{Introduction}
\noindent
%In nature, collective behavior and self-organization allow complicated global patterns to emerge from simple local interaction rules and random fluctuations.
%Such phenomena are observed in all kinds of biological systems ranging from bacteria~\cite{koch1998social} through ant colonies~\cite{couzin2003self} to schools of fish~\cite{niwa1994self} and flocks of birds~\cite{cucker2007emergent}.
%Even on the cell level~\cite{camazine2020self} and in human societies there are plenty of examples to be found, like, for instance, with opinion formation on social media~\cite{moussaid2013social}, the direction of the stock market in a financial or economic crisis~\cite{Dal_Maso_Peron_2011}, and in traffic when pedestrians are crossing a road~\cite{cristiani2010modeling} or when a traffic jam is forming~\cite{cristiani2012can}.
%Over the last decades, large systems of interacting particles have been used to model the behavior of such social and biological systems, trying to recreated the observed patterns by assuming reasonable interaction principles as well as learning these very rules from real-world data~\cite{bongini2017inferring}.
%
%Inspired by the fascinating capabilities of swarm intelligence in nature, such as the ability of a school of fish to defend from superior predators or the ability of a flock of birds to energy-efficiently travel far distances, large multi-agent systems following the discovered concepts were also introduced as a tool for solving challenging problems in applied mathematics.

\noindent
In nature, collective behavior and self-organization allow complicated global patterns to emerge from simple interaction rules and random fluctuations.
Inspired by the fascinating capabilities of swarm intelligence, large multi-agent systems are employed as a tool for solving challenging problems in applied mathematics.
One classical task arising throughout science is concerned with the global optimization of a problem-dependent possibly nonconvex and nonsmooth objective function $\TE:\RR^d\to \RR$, i.e., the search for a global optimizer
\begin{equation} \label{optimization}
	x^*\in\argmin_{x\in \RR^d} \TE(x).
\end{equation}
A popular class of methods with a long history of achieving state-of-the-art performance on such problems are metaheuristics~\cite{dreo2006metaheuristics}.
They orchestrate an interplay between local and global improvement procedures, consider memory mechanisms and selection strategies, and combine random and deterministic decisions, to create a process capable of escaping local optima and performing a robust search of the solution space in order to find a global optimizer.
Initiated by seminal works on stochastic approximation~\cite{randomsearch} and random search~\cite{rastrigin63}, a big variety of such mechanisms has been introduced, analyzed and applied to numerous real-world problems.
A non-exclusive list of representatives includes evolutionary programming~\cite{fogel2006evolutionary}, genetic algorithms~\cite{holland1992adaptation}, simulated annealing~\cite{aarts1989simulated}, and particle swarm optimization~\cite{kennedy1995particle}.
Despite their tremendous empirical success, it is very difficult to provide a theoretical convergence analysis to global minimizers, mostly due to their stochastic nature and the appearance of memory effects.

In this paper we study particle swarm optimization~(PSO), which was initially introduced by Kennedy and Eberhart in the 90s~\cite{kennedy1995particle,kennedy1997particle} and is now widely recognized as an efficient method for tackling complex optimization problems~\cite{poli2007particle,lin2008particle}.
Originally, PSO solves~\eqref{optimization} by considering a group of finitely many particles, which explore the energy landscape of~$\TE$.
Each agent experiences a force towards its own personal (historical) best position as well as towards the global best position communicated in the swarm.
Although these interaction rules are seemingly simple, a complete numerical analysis of PSO is still lacking; see, e.g., \cite{witt2011theory,panigrahi2011handbook,zhang2015comprehensive} and references therein.
Recently, however, by introducing a continuous description of PSO based on a system of stochastic differential equations~(SDEs), the authors of~\cite{grassi2020particle} have paved the way for a rigorous mathematical analysis using tools from stochastic calculus and the analysis of partial differential equations~(PDEs).

To explore the domain and to form a global consensus about the minimizer $x^*$ as time passes, the formulation of PSO proposed by the authors of~\cite{grassi2020particle} uses $N$ particles, described by triplets $\big((X_t^i,Y_t^i,V_t^i)_{t\geq0}\big)_{i=1\dots,N}$, with $X_t^i$ and $V_t^i$ denoting the position and velocity, and $Y_t^i$ being a regularized version of the local (historical) best position of the $i$th agent at time~$t$.
The particles, formally stochastic processes, are initialized independently according to some common distribution~$f_0\in\mc{P}(\RR^{3d})$.
In the most general form the PSO dynamics is given by the system of SDEs, expressed in It\^o's form as
\begin{subequations} \label{eq:PSO_with_memory}
\begin{align}
	dX_t^i   &= V_t^i \,dt, \label{eq:PSO_with_memory_X}\\
	dY_{t}^i &= \kappa \left(X_{t}^i-Y_{t}^i\right) S^{\beta,\theta}\!\left(X_{t}^i, Y_{t}^i\right) dt, \label{eq:PSO_with_memory_Y}\\
%	dV_{t}^i &= \begin{aligned}[t]
%		 &\!-\frac{\gamma}{m} V_{t}^i \,dt
%		    + \frac{\lambda_{1}}{m}\left(Y_{t}^i-X_{t}^i\right) dt
%		    +\frac{\lambda_{2}}{m}\left(y_{\alpha}(\widehat{\rho}_{Y,t}^N)-X_{t}^i\right) dt \\
%		 &\!+\frac{\sigma_{1}}{m} D\!\left(Y_{t}^i-X_{t}^i\right) d B_{t}^{1,i}
%		    +\frac{\sigma_{2}}{m} D\!\left(y_{\alpha}(\widehat{\rho}_{Y,t}^N)-X_{t}^i\right) d B_{t}^{2,i},
%	\end{aligned}
	m\,dV_{t}^i &= \begin{aligned}[t] \label{eq:PSO_with_memory_V}
		 &\!-\gamma V_{t}^i \,dt
		    + \lambda_{1}\!\left(Y_{t}^i-X_{t}^i\right) dt
		    +\lambda_{2}\!\left(y_{\alpha}(\widehat{\rho}_{Y,t}^N)-X_{t}^i\right) dt \\
		 &\!+\sigma_{1} D\!\left(Y_{t}^i-X_{t}^i\right) d B_{t}^{1,i}
		    +\sigma_{2} D\!\left(y_{\alpha}(\widehat{\rho}_{Y,t}^N)-X_{t}^i\right) d B_{t}^{2,i},
	\end{aligned}
\end{align}
\end{subequations}
where $\alpha,\beta,\theta,\kappa, \gamma,m,\lambda_1,\lambda_2,\sigma_1,\sigma_2\geq 0$ are user-specified parameters.
The change of the velocity in \eqref{eq:PSO_with_memory_V} is subject to five forces.
The first term on the right-hand side models friction with a coefficient commonly chosen as $\gamma=1-m\geq 0$, where $m>0$ denotes the inertia weight.
The subsequent term can be regarded as the drift towards the local best position of the $i$th particle.
A time-continuous approximation of its evolution is given by $Y^i_t$ and described in Equation~\eqref{eq:PSO_with_memory_Y}.
It involves the operator $S^{\beta,\theta}$, given by $S^{\beta,\theta}(x,y)=1+\theta+\tanh(\beta(\TE(y)-\TE(x)))$ for $0\leq\theta\ll1$ and $\beta\gg1$, which converges to the Heaviside function as $\theta\to0$ and $\beta\to\infty$.
The concept behind Equation~\eqref{eq:PSO_with_memory_Y} can then be seen when being discretized, see Remark~\ref{rem:local_best}.
For an alternative implementation of the local best position we refer to~\cite{TW}.
\begin{remark} \label{rem:local_best}
	A time-discretization of~\eqref{eq:PSO_with_memory_Y} with $\kappa=1/(2\Delta t)$, $\theta=0$ and $\beta=\infty$ yields the update rule
	\begin{equation} \label{eq:local_best update}
		Y_{(k+1)\Delta t}^i = 
		\begin{cases}
			Y_{k\Delta t}^i, &\textrm{if } \TE(X_{(k+1)\Delta t}^i) \geq \TE(Y_{k\Delta t}^i), \\
			X_{(k+1)\Delta t}^i, &\textrm{if } \TE(X_{(k+1)\Delta t}^i) < \TE(Y_{k\Delta t}^i),
		\end{cases}
	\end{equation}
	meaning that the $i$th particle stores in $Y_{k\Delta t}^i$ the best position which it has seen up to the $k$th iteration.
	This explains the name local (historical) best position and restores the original definition from the work~\cite{kennedy1995particle}.
\end{remark}
\noindent
The last deterministic term imposes a drift towards the momentaneous consensus point $y_{\alpha}(\widehat{\rho}_{Y,t}^N)$, given by
\begin{align} \label{eq:momentaneous_consensus}
	y_{\alpha}(\widehat{\rho}_{Y,t}^N)
	:= \int_{\RR^d} y \frac{\omega_\alpha^\TE(y)}{\norm{\omega_\alpha^\TE}_{L_1(\widehat{\rho}_{Y,t}^N)}}\,d\widehat{\rho}_{Y,t}^N(y),
	\quad \textrm{ with }\quad
	\omega_\alpha^\TE(y) := \exp(-\alpha \TE(y)),
\end{align}
where $\widehat{\rho}_{Y,t}^N$ denotes the empirical measure $\widehat{\rho}_{Y,t}^N:=\frac{1}{N}\sum_{i=1}^{N}\delta_{Y_t^{i}}$ of the particles' local best positions.
The choice of the weight $\omega_\alpha^\TE$ in~\eqref{eq:momentaneous_consensus} comes from the well-known Laplace principle~\cite{miller2006applied,Dembo2010}, a classical asymptotic argument for integrals stating that for any probability measure $\varrho\in\mc{P}(\RR^d)$ it holds
\begin{equation}\label{eq:laplace_principle}
	\lim_{\alpha\to\infty}\left(-\frac{1}{\alpha}\log\left(\int_{\RR^d}\omega_\alpha^\TE(y) \, d\varrho(y)\right)\right)
	=\inf_{y\in\supp(\varrho)} \TE(y).
\end{equation}
Based thereon, $y_{\alpha}(\widehat{\rho}_{Y,t}^N)$ is expected to be a rough estimate for a global minimizer~$x^*$, which improves as $\alpha\to\infty$ and as larger regions of the domain are explored.
To feature the latter, the two remaining terms in~\eqref{eq:PSO_with_memory_V}, each associated with a drift term, are diffusion terms injecting randomness into the dynamics through independent standard Brownian motions~$\big((B_t^{1,i})_{t\geq0}\big)_{i=1,\dots,N}$ and~$\big((B_t^{2,i})_{t\geq0}\big)_{i=1,\dots,N}$.
The two commonly studied diffusion types for similar methods are isotropic~\cite{pinnau2017consensus,carrillo2018analytical,fornasier2021consensus} and anisotropic~\cite{carrillo2019consensus,fornasier2021convergence} diffusion with
\begin{equation} \label{eq:diffustion_types}
	D\!\left(y-x\right) =
	\begin{cases}
		\norm{y-x}_2 \Id, & \textrm{ for isotropic diffusion,}\\
		\diag\left(y-x\right)\!, & \textrm{ for anisotropic diffusion},
	\end{cases}
\end{equation}
where $\Id\in\RR^{d\times d}$ is the identity matrix and $\diag:\RR^d\rightarrow\RR^{d\times d}$ the operator mapping a vector onto a diagonal matrix with the vector as its diagonal.
Intuitively, the term's scaling encourages agents far from its own local best position or the globally computed consensus point to explore larger regions, whereas agents already close try to enhance their position only locally.
As the coordinate-dependent scaling of anisotropic diffusion has been proven to be highly beneficial for high-dimensional problems~\cite{carrillo2019consensus,fornasier2021anisotropic}, in what follows, we limit our analysis to this case.
An illustration of the formerly described PSO dynamics~\eqref{eq:PSO_with_memory} is given in Figure~\ref{figure:PSO_dynamics}.
\begin{figure}[!ht]
	\centering
	\includegraphics[width=0.75\textwidth, trim=0 327 0 318,clip]{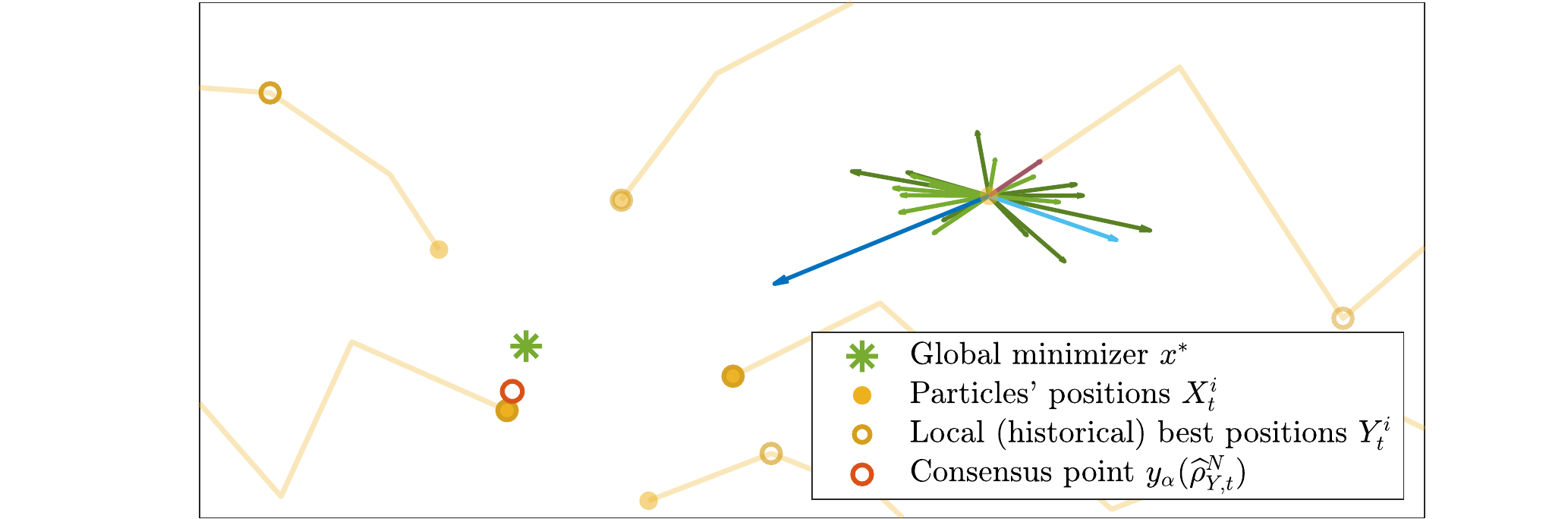}
	\caption{An illustration of the PSO dynamics.
		Agents with positions~$X^1,\dots,X^N$ (yellow dots with their trajectories) explore the energy landscape of~$\TE$ in search of the global minimizer~$x^*$ (green star).
		The dynamics of each particle is governed by five terms.
		A local drift term (light blue arrow) imposes a force towards its local best position~$Y^i_t$ (indicated by a circle).
		A global drift term (dark blue arrow) drags the agent towards a momentaneous consensus point~$y_\alpha(\widehat{\rho}_{Y,t}^N)$ (orange circle) computed as a weighted (visualized through color opacity) average of the particles' local best positions.
		Friction (purple arrow) counteracts inertia.
		The two remaining terms are diffusion terms (light and dark green arrows) associated with a respective drift term.}
	\label{figure:PSO_dynamics}
\end{figure}

A theoretical convergence analysis of PSO is possible either on the microscopic level~\eqref{eq:PSO_with_memory} or by analyzing the macroscopic behavior of the particle density through a mean-field limit, what usually admits more powerful analysis tools.
In the large particle limit an individual particle is not influenced any more by individual particles but only by the average behavior of all particles.
As shown in~\cite[Section~3.2]{grassi2021mean}, the empirical particle measure~$\widehat f^N :=\frac{1}{N}\sum_{i=1}^{N}\delta_{(X^{i},Y^{i},V^{i})}$ converges in law to the deterministic agent distribution $f\in\mc{C}([0,T],\mc{P}(\RR^{3d}))$, which weakly satisfies the nonlinear Vlasov-Fokker-Planck equation
\begin{equation} \label{eq:PSO_with_memory_weak}
\begin{split}
	\partial_{t} f_t
	+ v \cdot \nabla_{x} f_t
	& +\nabla_{y} \cdot \left(\kappa(x-y) S^{\beta,\theta}(x,y) f_t\right)
	= \nabla_{v} \cdot\Bigg(
		\frac{\gamma}{m} v f_t
		+ \frac{\lambda_{1}}{m}\left(x-y\right) f_t
		+ \frac{\lambda_{2}}{m}\left(x-y_{\alpha}(\rho_{Y,t})\right) f_t\\
		& + \left(\frac{\sigma_{1}^{2}}{2 m^{2}} \big(D\!\left(x-y\right)\!\big)^{2} + \frac{\sigma_{2}^{2}}{2m^2} \big(D\!\left(x-y_{\alpha}(\rho_{Y,t})\right)\!\big)^{2}\right) \nabla_{v} f_t
	\Bigg)
\end{split}	
\end{equation}
with initial datum $f_0$.
The mean-field limit results~\cite{bolley2011stochastic,huang2020mean,sznitman1991topics,jabin2017mean,huang2021mean} ensure that the particle system~\eqref{eq:PSO_with_memory} is well approximated by the following self-consistent mean-field McKean process
\begin{subequations} \label{eq:PSO_with_memory_mf}
\begin{align}
	d\OX_t   &= \OV_t \,dt, \label{eq:PSO_with_memory_X_mf}\\
	d\OY_{t} &= \kappa \left(\OX_{t}-\OY_{t}\right) S^{\beta,\theta}\!\left(\OX_{t}, \OY_{t}\right) dt, \label{eq:PSO_with_memory_Y_mf}\\
	m\,d\OV_{t} &= \begin{aligned}[t] \label{eq:PSO_with_memory_V_mf}
		 &\!-\gamma \OV_{t} \,dt
		    + \lambda_{1}\!\left(\OY_{t}-\OX_{t}\right) dt
		    +\lambda_{2}\!\left(y_{\alpha}(\rho_{Y,t})-\OX_{t}\right) dt \\
		 &\!+\sigma_{1} D\!\left(\OY_{t}-\OX_{t}\right) d B_{t}^1
		    +\sigma_{2} D\!\left(y_{\alpha}(\rho_{Y,t})-\OX_{t}\right) d B_{t}^2,
	\end{aligned}
\end{align}
\end{subequations}
with initial datum $(\OX_0,\OY_0,\OV_0)\sim f_0$ and the marginal law~$\rho_{Y,t}$ of $\OY_t$ given by $\rho_Y(t,\,\cdot\,)=\iint_{\RR^{d}\times\RR^{d}} df_t(x,\,\cdot\,,v)$.
%\begin{equation*}
%	\rho_Y(t,y)=\iint_{\RR^{d}\times\RR^{d}} f(t,dx, y,d v),
%\end{equation*}
Here, $f_t$ denotes the distribution of $(\OX_t,\OY_t,\OV_t)$.
This makes \eqref{eq:PSO_with_memory_weak} and \eqref{eq:PSO_with_memory_mf} nonlinear.

%%%%%%%%%%%%%%%%%%%%%%%%%%%%%%%%%%%%%%%%
%%%%%%%%%%%%%%%%%%%%%%%%%%%%%%%%%%%%%%%%
\subsection{Contribution}
In view of the versatility, efficiency, and wide applicability of PSO combined with its long historical tradition, a mathematical analysis of the finite particle system~\eqref{eq:PSO_with_memory} is of considerable interest.

In this work we advance the theoretical understanding of the method and contribute to the completion of a full numerical analysis of PSO by proving rigorously the convergence of PSO with memory effects to global minimizers using mean-field techniques.
More precisely, under mild regularity assumptions on the objective~$\TE$ and a well-preparation condition about the initialization~$f_0$, we analyze the behavior of the particle distribution~$f$ solving the mean-field dynamics~\eqref{eq:PSO_with_memory_mf}.
At first, it is shown that concentration is achieved at some $\tilde x$ in the sense that the marginal law w.r.t.\@~the local best position, $\rho_{Y,t}$, converges narrowly to a Dirac delta $\delta_{\tilde x}$ as $t\to\infty$.
%This is obtained by showing that the quantities $\EE\absnormal{\OX_t-\EE\OX_t}^2$, $\EE\absnormal{\OV_t}^2$, and $\EE\absnormal{\OX_t-\OY_t}^2$ decrease exponentially in time and thus yielding some $\tilde x$ depending on $\alpha$ such that $\EE[\OX_t]\to\tilde x$ as well as $y_\alpha(\rho_{Y,t})\to \tilde x$ as $t\to\infty$.
Consecutively, we argue that, for an appropriate choice of the parameters, in particular~$\alpha\gg1$, $\TE(\tilde x)$ can be made arbitrarily close to the minimal value $\underline{\TE} := \inf_{x\in \RR^d} \TE(x)$.
A suitable tractability condition on the objective~$\TE$ eventually ensures that $\tilde x$ is close to a global minimizer.
Similar mean-field convergence results are obtained for the case without memory effects.
In this setting we are moreover able to establish the convergence of the interacting $N$-particle dynamics to its mean-field limit with a dimension-independent rate, which allows to obtain a so far unique wholistic and quantitative convergence statement of PSO.
As the mean-field approximation result does not suffer from the curse of dimensionality, we in particular prove that the numerical PSO method has polynomial complexity.
With these new results we solve the theoretical open problem about the convergence of PSO posed in~\cite{grassi2020particle}.

Furthermore, we propose an efficient and parallelizable implementation, which is particularly suited for machine learning problems by integrating modern machine learning techniques such as random mini-batch ideas as well as traditional metaheuristic-inspired techniques from genetic programming and simulated annealing.

%%%%%%%%%%%%%%%%%%%%%%%%%%%%%%%%%%%%%%%%
%%%%%%%%%%%%%%%%%%%%%%%%%%%%%%%%%%%%%%%%
\subsection{Prior Arts}
The convergence of PSO algorithms has been investigated by many scholars since its introduction, which has lead to several variations allowing to establish desirable properties such as consensus formation or convergence to optimal solutions.
While the matter of consensus is well-studied, see, e.g., \cite{Clerc2002Theparticleswarm,zcan1998AnalysisOA}, only few general theoretical statements regarding the properties of the found consensus are available.
Both the existence of a large number of variations of the algorithm and the lack of a rigoros global convergence analysis are attributed amongst other things, such as the stochasticity and the usage of memory mechanisms,
to the phenomenon of premature convergence of basic PSO~\cite{kennedy1995particle}, which was observed in~\cite{van2010convergence,van2007analysis} and remedied by proposing a modified version, called guaranteed convergence PSO.
Nevertheless, this adaptation only allows to prove the convergence to local optima.
In order to obtain therefrom a stochastic global search algorithm, the authors suggest to add purely stochastic particles to the swarm, which trivially makes the method capable of detecting a global optimizer, but entails a computational time which  coincides with the time required to examine every location in the search space.
Other works consider certain notions of weak convergence~\cite{bruned2018weak} or provide probabilistic guarantees of finding locally optimal solutions, meaning that eventually all particles are located almost surely at a local optimum of the objective function~\cite{schmitt2015particle}.
In \cite{poli2009mean}, similarly to our work, the expected behavior of the particles is investigated.

All of the formerly mentioned results though are obtained through the analysis of the particles' trajectories generated by a time-discretized algorithm as in~\cite[Equation~(6.3)]{grassi2021mean}.
The present paper takes a different point of view by studying the time-continuous description of the PSO model~\eqref{eq:PSO_with_memory} through the lens of the mean-field approximation~\eqref{eq:PSO_with_memory_weak}.
Analyzing the macroscopic behavior of a system through a mean-field limit instead of investigating the microscopic particle dynamics has its origins in statistical mechanics~\cite{kadanoff2009Moreisthesame}, where interactions between particles are approximated by an averaged influence.
By eliminating the correlation between the particles, a many-body problem can be reduced to a one-body problem, which is usually much easier to solve while still giving an understanding of the mechanisms at play by describing the average behavior of the particles.
These ideas, for instance, are also used to study the collective behavior of animals when forming large-scale patterns through self-organization by analyzing an associated kinetic PDE~\cite{bolley2011stochastic}.
In very recent works, this perspective of analysis has also been taken to demystify the training process of neural networks, see, e.g., \cite{montanare2018Ameanfieldview,ding2021global}, where a mean-field approximation is utilized to formulate risk minimization by stochastic gradient descent~(SGD) in terms of a gradient-flow PDE, which allows for a rigorous mathematical analysis.

The analysis technique we use follows the line of work of self-organization.
It is inspired by~\cite{carrillo2018analytical,carrillo2019consensus}, where a variance-based analysis approach has been developed for consensus-based optimization~(CBO), which follows the guiding principles of metaheuristics and in particular resembles PSO but is of much simpler nature and therefore easier to analyze.
In comparison to Equation~\eqref{eq:PSO_with_memory}, CBO methods are described by a system of first-order SDEs~\cite[Equation~(1.1)]{carrillo2018analytical} and do not contain memory mechanisms, which are responsible for both a significantly more challenging mathematical modeling and convergence analysis.

%%%%%%%%%%%%%%%%%%%%%%%%%%%%%%%%%%%%%%%%
%%%%%%%%%%%%%%%%%%%%%%%%%%%%%%%%%%%%%%%%
\subsection{Organization}
Sections~\ref{sec:PSO_without_memory} and~\ref{sec:PSO_with_memory} are dedicated to the analysis of PSO without and with memory mechanisms, respectively.
After providing details about the well-posedness of the mean-field dynamics, we present and discuss the main result about the convergence of the mean-field dynamics to a global minimizer of the objective function.
In Section~\ref{sec:wholisticconvergence} we then state a quantitative result about the mean-field approximation for PSO without memory effects, which enables us to obtain a wholistic convergence statement of the numerical PSO method.
Eventually, a computationally efficient implementation of PSO is proposed in Section~\ref{sec:numerics}, before Section~\ref{sec:conclusion} concludes the paper.
For the sake of reproducible research, in the GitHub repository \url{https://github.com/KonstantinRiedl/PSOAnalysis} we provide the Matlab code implementing the PSO algorithm analyzed in this work.

%%%%%%%%%%%%%%%%%%%%%%%%%%%%%%%%%%%%%%%%
%%%%%%%%%%%%%%%%%%%%%%%%%%%%%%%%%%%%%%%%
%%%%%%%%%%%%%%%%%%%%%%%%%%%%%%%%%%%%%%%%
\section{Mean-Field Analysis of PSO without Memory Effects} \label{sec:PSO_without_memory}

\noindent
Before providing a theoretical analysis of the mean-field PSO dynamics~\eqref{eq:PSO_with_memory_weak} and~\eqref{eq:PSO_with_memory_mf}, in this section we investigate a reduced version, which does not involve memory mechanisms.
Its multi-particle formulation was proposed in~\cite[Section~3.1]{grassi2020particle} and reads
\begin{subequations} \label{eq:PSO_without_memory}
\begin{align}
	dX_t^i   &= V_t^i \,dt, \label{eq:PSO_without_memory_X}\\
	m\,dV_{t}^i &= -\gamma V_{t}^i \,dt
		    +\lambda\left(x_{\alpha}(\widehat{\rho}_{X,t}^N)-X_{t}^i\right) dt 
		    +\sigma D\!\left(x_{\alpha}(\widehat{\rho}_{X,t}^N)-X_{t}^i\right) dB_{t}^{i}.
		    \label{eq:PSO_without_memory_V}
\end{align}
\end{subequations}
Compared to the full model, each particle is characterized only by its position~$X^i$ and velocity~$V^i$.
The forces acting on a particle, i.e., influencing its velocity in Equation~\eqref{eq:PSO_without_memory_V}, are friction, acceleration through the consensus drift and diffusion as in~\eqref{eq:diffustion_types} with independent standard Brownian motions $\big((B_t^i)_{t\geq0}\big)_{i=1,\dots,N}$.
The consensus point~$x_{\alpha}(\widehat{\rho}_{X,t}^N)$ is directly computed from the current positions of the particles according to
\begin{align}
	x_{\alpha}(\widehat{\rho}_{X,t}^N)
	:= \int_{\RR^d} x \frac{\omega_\alpha^\TE(x)}{\norm{\omega_\alpha^\TE}_{L_1(\widehat{\rho}_{X,t}^N)}}\,d\widehat{\rho}_{X,t}^N(x),
\end{align}
where $\widehat{\rho}_{X,t}^N$ denotes the empirical measure $\widehat{\rho}_{X,t}^N:=\frac{1}{N}\sum_{i=1}^{N}\delta_{X_t^{i}}$ of the particles' positions.
Independent and identically distributed initial data~$\big((X_0^i,V_0^i)\sim f_0\big)_{i=1,\dots,N}$ with~$f_0\in\mc{P}(\RR^{2d})$ complement~\eqref{eq:PSO_without_memory}.

Similar to the particle system \eqref{eq:PSO_with_memory}, as $N\to\infty$, the mean-field dynamics of~\eqref{eq:PSO_without_memory} is described by the nonlinear self-consistent McKean process
\begin{subequations} \label{eq:PSO_without_memory_mf}
\begin{align}
	d\OX_t   &= \OV_t \,dt, \label{eq:PSO_without_memory_X_mf}\\
	m\,d\OV_{t} &= -\gamma \OV_{t} \,dt  
		+\lambda\left(x_{\alpha}(\rho_{X,t})-\OX_{t}\right) dt
		+\sigma D\!\left(x_{\alpha}(\rho_{X,t})-\OX_{t}\right) d B_{t}, \label{eq:PSO_without_memory_V_mf}
\end{align}
\end{subequations}
with initial datum $(\OX_0,\OV_0)\sim f_0$ and the marginal law $\rho_{X,t}$ of $\OX_t$ given by $\rho_X(t,\,\cdot\,)=\int_{\RR^d} df(t,\,\cdot\,,v)$.
%\begin{equation*}
%	\rho_X(t,x)=\int_{\RR^d} f(t, x, dv).
%\end{equation*}
A direct application of the It\^{o}-Doeblin formula shows that the law~$f\in\mc{C}([0,T],\mc{P}(\RR^{2d}))$ is a weak solution to the nonlinear Vlasov-Fokker-Planck equation
\begin{equation} \label{eq:PSO_without_memory_weak}
\begin{split}
	\partial_{t} f_t
	+ v \cdot \nabla_{x} f_t
	&= \nabla_{v} \cdot\left(
		\frac{\gamma}{m} v f_t
		+ \frac{\lambda}{m}\left(x-x_{\alpha}({\rho_{X,t}})\right) f_t
		+ \frac{\sigma^{2}}{2m^2} \big(D\!\left(x-x_{\alpha}({\rho_{X,t}})\right)\!\big)^{2} \,\nabla_{v} f_t
	\right)
\end{split}	
\end{equation}
with initial datum $f_0$.

\begin{remark}
	A separate theoretical analysis of the dynamics~\eqref{eq:PSO_without_memory} is necessary as it cannot be derived from~\eqref{eq:PSO_with_memory} in a way that also the proof technique can be adopted in a straightforward manner.
	
	It is also worth noting that Equation~\eqref{eq:PSO_without_memory} bears a certain resemblance to CBO~\cite{pinnau2017consensus,carrillo2018analytical,carrillo2019consensus,fornasier2021consensus,fornasier2021convergence}.
	Indeed, as made rigorous in~\cite{cipriani2022zero}, CBO methods can be derived from PSO in the small inertia limit~$m\to0$.
\end{remark}

Before turning towards the well-posedness of the mean-field dynamics~\eqref{eq:PSO_without_memory_mf} and presenting the main result of this section about the convergence to the global minimizer~$x^*$,
let us introduce the class of objective function~$\TE$ considered in the theoretical part of this work.
We remark that the assumptions made in what follows coincide with the ones of~\cite{carrillo2018analytical,carrillo2019consensus} as well as several subsequent works in this direction.

\begin{assum} \label{ass:assumptions}
Throughout the paper we are interested in objective functions $\TE:\RR^d\rightarrow\RR$, for which
\begin{enumerate}[label=A\arabic*,labelsep=10pt,leftmargin=35pt]
	\item\label{asm:zero_global}
		there exists $x^*\in\RR^d$ such that $\TE(x^*)=\inf_{x\in\RR^d} \TE(x)=:\underline{\TE}$,
	\item\label{asm:Lipschitz}
		there exists some constant $L_\TE>0$ such that
		\begin{equation*}
			\abs{\TE(x)-\TE(x')} \leq L_\TE\left(\abs{x}+\abs{x'}\right)\abs{x-x'}, \quad \text{for all } x,x'\in\RR^d,
		\end{equation*}
	\item\label{asm:growth}
		either $\overline\TE:=\sup_{x\in \RR^d}\TE(x)<\infty$ or there exist constants $c_\TE,R>0$ such that
		\begin{equation*}
			\TE(x)-\underline\TE \geq c_\TE\abs{x}^2, \quad \text{for all } x\in\RR^d \text{ with } \abs{x}\geq R,
		\end{equation*}
	\item\label{asm:Hessian}
		$\TE\in\mc{C}^2(\RR^d)$ with $\norm{\nabla^2\TE}_\infty\leq C_\TE$ for some constant $C_\TE>0$,
	\item\label{asm:icp}
		there exist $\eta>0$ and $\nu\in(0,\infty)$ such that for any $x\in\RR^d$ there exists a global minimizer $x^*$ of $\TE$ (which may depend on $x$) such that
		\begin{equation*}
			\abs{x-x^*} \leq (\TE(x)-\underline\TE)^{\nu}/\eta.
		\end{equation*}
\end{enumerate}
\end{assum}

Assumption~\ref{asm:zero_global} just states that the objective function~$\TE$ attains its infimum~$\underline\TE$ at some~$x^*\in\RR^d$, which may not necessarily be unique.
Assumption~\ref{asm:Lipschitz} describes the local Lipschitz-continuity of~$\TE$, entailing in particular that the objective has at most quadratic growth at infinity.
Assumption~\ref{asm:growth}, on the other hand, requires~$\TE$ to be either bounded or of at least quadratic growth in the farfield.
Together, \ref{asm:Lipschitz} and \ref{asm:growth} allow to obtain the well-posedness of the PSO model.
Assumption~\ref{asm:Hessian} is a regularity assumption about~$\TE$, which is required only for the theoretical analysis. The PSO method is a zero-order method where we do not need the gradient information of the objective function in the numerical application.
Assumption~\ref{asm:icp} should be interpreted as a tractability condition of the landscape of~$\TE$, which ensures that achieving an objective value of approximately~$\underline\TE$ guarantees closeness to a global minimizer~$x^*$ and thus eliminates cases of almost-optimal valleys in the energy landscape far away from any globally minimizing argument.
Such assumption is therefore also referred to as an inverse continuity property.

It shall be emphasized that objectives with multiple global minima of identical quality are not excluded.

%%%%%%%%%%%%%%%%%%%%%%%%%%%%%%%%%%%%%%%%
%%%%%%%%%%%%%%%%%%%%%%%%%%%%%%%%%%%%%%%%
\subsection{Well-Posedness of PSO without Memory Effects} \label{sec:wellposedness_without_memory}

Let us recite a well-posedness result about the mean-field PSO dynamics~\eqref{eq:PSO_without_memory_mf} and the associated Vlasov-Fokker-Planck equation~\eqref{eq:PSO_without_memory_weak}.
Its proof is analogous to the one provided for Theorem~\ref{thm:wellposedness_memory} for the full dynamics~\eqref{eq:PSO_with_memory_mf} based on the Leray-Schauder fixed point theorem. 

\begin{thm} \label{thm:wellposedness}
	Let $\TE$ satisfy Assumptions~\ref{asm:zero_global}--\ref{asm:growth} and let $m,\gamma,\lambda,\sigma,\alpha,T>0$.
	If $(\OX_0,\OV_0)$ is distributed according to $f_0\in \mc{P}_4(\RR^{2d})$, then the nonlinear SDE~\eqref{eq:PSO_without_memory_mf} admits a unique strong solution up to time~$T$ with the paths of process~$(\OX,\OV)$ valued in $ \mc{C}([0,T],\RR^{d})\times \mc{C}([0,T],\RR^{d})$.
	The associated law~$f$ has regularity~$\mc{C}([0,T],\mc{P}(\RR^{2d}))$ and is a weak solution to the Vlasov-Fokker-Planck equation~\eqref{eq:PSO_without_memory_weak}.
	In particular,
	\begin{equation} \label{eq:wellposedness_regularity}
		\sup_{t\in[0,T]} \EE[\absnormal{\OX_t}^4+\absnormal{\OV_t}^4]\leq \left(1+2\EE[\absnormal{\OX_0}^4+\absnormal{\OV_0}^4]\right)e^{CT}
	\end{equation}
	for some constant $C>0$ depending only on $m,\gamma,\lambda,\sigma,\alpha, c_\TE, R,$ and $L_\TE$.
\end{thm}

%%%%%%%%%%%%%%%%%%%%%%%%%%%%%%%%%%%%%%%%
%%%%%%%%%%%%%%%%%%%%%%%%%%%%%%%%%%%%%%%%
\subsection{Convergence of PSO without Memory Effects to a Global Minimizer} \label{sec:convergence_without_memory}

A successful application of the PSO dynamics underlies the premise that the particles form consensus about a certain position~$\tilde x$.
In particular, in the mean-field limit one expects that the distribution of a particle's position~$\rho_{X,t}$ converges to a Dirac delta~$\delta_{\tilde x}$.
This entails that the variance in the position~$\EE[\absnormal{\OX_t-\EE[\OX_t]}^2]$ and the second-order moment of the velocity~$\EE[\absnormal{\OV_t}^2]$ of the averaged particle vanish.
As we show in what follows, both functionals indeed decay exponentially fast in time.
Motivated by these expectations we define the functional 
\begin{equation} \label{eq:H}
	\CH(t) := \left(\frac{\gamma}{2m}\right)^2 \absnormal{\OX_t-\EE[\OX_t]}^2+\absnormal{\OV_t}^2+\frac{\gamma}{2m}\left\langle\OX_t-\EE[\OX_t]), \OV_t\right\rangle,
\end{equation}
which we analyze in the remainder of this section.
Its last term is required from a technical perspective.
However, by proving the decay of $\EE[\CH(t)]$, one immediately obtains the same for $\EE[\absnormal{\OX_t-\EE[\OX_t]}^2+\absnormal{\OV_t}^2]$ as a consequence of the equivalence established in Lemma~\ref{lem:equivalence_H}, which follows from Young's inequality.

\begin{lem} \label{lem:equivalence_H}
	The functional $\CH(t)$ is equivalent to $\absnormal{\OX_t-\EE[\OX_t]}^2+\absnormal{\OV_t}^2$ in the sense that
	\begin{equation} \label{eq:equivalence_H}
	\begin{split}
		\frac{1}{2}\left(\frac{\gamma}{2m}\right)^2\absnormal{\OX_t-\EE[\OX_t]}^2+\frac{1}{2}\absnormal{\OV_t}^2
		\leq\CH(t)
		%&\leq \frac{3}{2}\left(\frac{\gamma}{2m}\right)^2\absnormal{\OX_t-\EE[\OX_t]}^2+\frac{3}{2}\absnormal{\OV_t}^2\\
		&\leq \frac{3}{2}\left(\left(\frac{\gamma}{2m}\right)^2+1\right)\left(\absnormal{\OX_t-\EE[\OX_t]}^2+\absnormal{\OV_t}^2\right).
	\end{split}
	\end{equation}
\end{lem}
%\begin{proof}
%	Both inequalities follow from Young's inequality.
%\end{proof}

We now derive an evolution inequality of the quantity $\EE[\CH(t)]$.
\begin{lem} \label{lem:evolution_H}
	Let $\TE$ satisfy Assumptions~\ref{asm:zero_global}--\ref{asm:growth} and let $(\OX_t,\OV_t)_{t\geq 0}$ be a solution to the nonlinear SDE~\eqref{eq:PSO_without_memory_mf}.
	Then $\EE[\CH(t)]$ with $\CH$ as defined in~\eqref{eq:H} satisfies
	\begin{equation} \label{eq:evolution_H}
		\frac{d}{dt}\EE[\CH(t)] \leq 
		-\frac{\gamma}{m}\EE[\absnormal{\OV_t}^2]
		-\left(\frac{\lambda\gamma}{2m^2}-\left(\frac{2\lambda^2}{\gamma m}+\frac{\sigma^2}{m^2}\right)\frac{2e^{-\alpha\underline\TE}}{\EE[\exp(-\alpha\TE(\OX_t))]}\right)\EE[\absnormal{\OX_t-\EE[\OX_t]}^2].
	\end{equation}
\end{lem}

\begin{proof}%[Proof of Lemma~\ref{lem:evolution_H}]
	Let us write $\delta\OX_t:=\OX_t-\EE[\OX_t]$ for short and note that the integration by parts formula gives
	\begin{equation} \label{eq:time_evolution_variance}
		\frac{d}{dt} \EE[\absnormal{\delta\OX_t}^2]
			=2\EE[\left\langle\delta\OX_t, \OV_t\right\rangle].
			%\leq \EE[\absnormal{\delta\OX_t}^2]+\EE[\absnormal{\OV_t}^2]
	\end{equation}
	Observe that the stochastic integrals have vanishing expectations as a consequence of the regularity obtained in Theorem~\ref{thm:wellposedness}. Applying the It\^{o}-Doeblin formula and Young's inequality yields 
	\begin{equation} \label{term2}
	\begin{split}
		\frac{d}{dt} \EE[\absnormal{\OV_t}^2]
		&= -\frac{2\gamma}{m}\EE[\absnormal{\OV_t}^2]
			+\frac{2\lambda}{m}\EE[\left\langle\OV_t,x_\alpha(\rho_{X,t})-\OX_t\right\rangle]
			+\frac{\sigma^2}{m^2}\EE[\absnormal{x_\alpha(\rho_{X,t})-\OX_t}^2] \\
		&\leq -\left(\frac{2\gamma}{m}-\frac{\lambda}{\varepsilon m}\right)\EE[\absnormal{\OV_t}^2]+\left(\frac{\varepsilon\lambda}{m}+\frac{\sigma^2}{m^2}\right)\EE[\absnormal{x_\alpha(\rho_{X,t})-\OX_t}^2],\quad \forall\, \varepsilon>0.
	\end{split}
	\end{equation}
	Again by employing the It\^{o}-Doeblin formula we obtain
	\begin{equation} \label{eq:time_evolution_scalarproductXV}
	\begin{split}
		\frac{d}{dt}\EE[\left\langle\delta\OX_t, \OV_t\right\rangle]
		&=\EE[\absnormal{\OV_t}^2]-\left(\EE[\OV_t]\right)^2-\frac{\gamma}{m}\EE[\left\langle\delta\OX_t,\OV_t\right\rangle]+\frac{\lambda}{m}\EE[\left\langle\delta\OX_t,x_\alpha(\rho_{X,t})-\OX_t\right\rangle]\\
		&\leq\EE[\absnormal{\OV_t}^2]-\frac{\gamma}{2m}\frac{d}{dt} \EE[\absnormal{\delta\OX_t}^2]+\frac{\lambda}{m}\EE[\left\langle\delta\OX_t,x_\alpha(\rho_{X,t})-\EE[\OX_t]\right\rangle]-\frac{\lambda}{m}\EE[\absnormal{\delta\OX_t}^2] \\
		&= \EE[\absnormal{\OV_t}^2]-\frac{\gamma}{2m}\frac{d}{dt} \EE[\absnormal{\delta\OX_t}^2]-\frac{\lambda}{m}\EE[\absnormal{\delta\OX_t}^2],
	\end{split}
	\end{equation}
	where we used the identity~\eqref{eq:time_evolution_variance} and the fact that $\EE[\left\langle\delta\OX_t,x_\alpha(\rho_{X,t})-\EE[\OX_t]\right\rangle]=0$ in the last two steps.
	We now rearrange inequality~\eqref{eq:time_evolution_scalarproductXV} to get
	\begin{equation*}
		\left(\frac{\gamma}{2m}\right)^2\frac{d}{dt} \EE[\absnormal{\delta\OX_t}^2]+\frac{\gamma}{2m}\frac{d}{dt}\EE[\left\langle\delta\OX_t, \OV_t\right\rangle]
		\leq \frac{\gamma}{2m}\EE[\absnormal{\OV_t}^2]-\frac{\lambda\gamma}{2m^2}\EE[\absnormal{\delta\OX_t}^2],
	\end{equation*}
	which, in combination with~\eqref{term2}, allows to show
	\begin{equation} \label{eq:evolution_H_aux}
		\frac{d}{dt}\EE[\CH(t)]
		\leq -\left(\frac{3\gamma}{2m}-\frac{\lambda}{\varepsilon m}\right)\EE[\absnormal{\OV_t}^2]
			-\frac{\lambda\gamma}{2m^2}\EE[\absnormal{\delta\OX_t}^2]
			+\left(\frac{\varepsilon\lambda}{m}+\frac{\sigma^2}{m^2}\right)\EE[\absnormal{\OX_t-x_\alpha(\rho_{X,t})}^2].
	\end{equation}
	In order to upper bound $\EE[\absnormal{\OX_t-x_\alpha(\rho_{X,t})}^2]$, an application of Jensen's inequality yields
	\begin{equation} \label{eq:boundX-consensuspoint}
		\EE[\absnormal{\OX_t-x_\alpha(\rho_{X,t})}^2]
		\leq 
		\frac{\iint_{\RR^{2d}}\abs{x-x'}^2\omega_\alpha^\TE(x')\,d\rho_{X,t}(x')d\rho_{X,t}(x)}{\int_{\RR^d} \omega_\alpha^\TE(x')\,d\rho_{X,t}(x')}
		\leq 
		2e^{-\alpha\underline \TE}\frac{\EE[\absnormal{\delta\OX_t}^2]}{\EE[\exp(-\alpha\TE(\OX_t))]}.
	\end{equation}
	By choosing $\varepsilon=(2\lambda)/\gamma$ in~\eqref{eq:evolution_H_aux} and utilizing the estimate~\eqref{eq:boundX-consensuspoint},
%we obtain
%\begin{align}
%	\frac{d}{dt}\EE[\CH(t)]
%	\leq-\frac{\gamma}{m}\EE[\absnormal{\OV_t}^2]-\left(\frac{\lambda\gamma}{2m^2}-\left(\frac{2\lambda^2}{\gamma m}+\frac{\sigma^2}{m^2}\right)\frac{2e^{-\alpha\underline \TE}}{\EE[\exp(-\alpha\TE(\OX_t))]}\right)\EE[\absnormal{\delta\OX_t}^2],
%\end{align}
%which completes the proof.
	we obtain~\eqref{eq:evolution_H} as desired.
\end{proof}

\begin{remark} \label{remark:parameter_choice}
	To obtain exponential decay of $\EE[\CH(t)]$ it is necessary to ensure the negativity of the prefactor of~$\EE[\absnormal{\OX_t-\EE[\OX_t]}^2]$ in Inequality~\eqref{eq:evolution_H} by choosing the parameters of the PSO method in a suitable manner.
	This may be achieved by choosing for any fixed time~$t$, given~$\alpha$ and arbitrary $\sigma,\gamma>0$,
	\begin{equation} \label{eq:remark:parameter_choice}
		\lambda > 4D_t^X\sigma^2/\gamma
		\quad\text{and subsequently}\quad
		m < \gamma^2/(8D_t^X\lambda),
	\end{equation}
	where we abbreviate $D_t^X=2e^{-\alpha\underline\TE}/\EE[\exp(-\alpha\TE(\OX_t))]$.
\end{remark}

In order to be able to choose the parameters in Remark~\ref{remark:parameter_choice} once at the beginning of the algorithm instead of at every time step~$t$, we need to be able to control the time-evolution of~$\EE[\exp(-\alpha\TE(\OX_t))]$.
We therefore study its time-derivative in the following lemma.

\begin{lem} \label{lem:evolution_expE2}
	Let $\TE$ satisfy Assumptions~\ref{asm:zero_global}--\ref{asm:Hessian} and let $(\OX_t,\OV_t)_{t\geq 0}$ be the solution to the nonlinear SDE~\eqref{eq:PSO_without_memory_mf}.
	Then it holds that
	\begin{equation}
		\frac{d^2}{dt^2}\left(\EE[\exp(-\alpha\TE(\OX_t))]\right)^2
		\geq -\frac{\gamma}{m} \frac{d}{dt} \left(\EE[\exp(-\alpha\TE(\OX_t))]\right)^2 
		 	-4\alpha e^{-2\alpha \underline{\TE}} C_\TE \left(1+2\frac{\lambda}{m}\left(\frac{2m}{\gamma}\right)^{2} \right)\EE[\CH(t)].
			\label{eq-inequality_lemE2}
	\end{equation}
\end{lem}

\begin{proof}
We first note that
\begin{equation} \label{eq:evolution_expE2}
\begin{split}
	\frac{1}{2}\frac{d^2}{dt^2}\left(\EE[\exp(-\alpha\TE(\OX_t))]\right)^2
	&= \frac{d}{dt}\left(\EE[\exp(-\alpha\TE(\OX_t))] \, \frac{d}{dt}\EE[\exp(-\alpha\TE(\OX_t))]\right) \\
	&= \left(\frac{d}{dt}\EE[\exp(-\alpha\TE(\OX_t))]\right)^2
		+\EE[\exp(-\alpha\TE(\OX_t))]\frac{d^2}{dt^2}\EE[\exp(-\alpha\TE(\OX_t))] \\
	&\geq \EE[\exp(-\alpha\TE(\OX_t))]\frac{d^2}{dt^2}\EE[\exp(-\alpha\TE(\OX_t))],
\end{split}
\end{equation}	
leaving the second time-derivative of $\EE[\exp(-\alpha\TE(\OX_t))]$ to be lower bounded.
To do so, we start with its first derivative.
Applying the It\^{o}-Doeblin formula twice and noting that stochastic integrals have vanishing expectations as a consequence of the regularity obtained in Theorem~\ref{thm:wellposedness}, we have
\begin{equation} \label{eq:evolution_expE_firstderivative}
\begin{split}
	&\frac{d}{dt}\EE[\exp(-\alpha\TE(\OX_t))]
	=-\alpha\EE[\exp(-\alpha\TE(\OX_t)) \langle\nabla\TE(\OX_t),\OV_t\rangle] \\
	&\qquad=-\alpha\EE[\exp(-\alpha\TE(\OX_0))\langle\nabla\TE(\OX_0),\OV_0 \rangle]
		- \alpha \EE\left[ \int_0^t d\left\langle \exp(-\alpha\TE(\OX_s))\nabla \TE(\OX_s), \OV_s \right\rangle \right] \\
	&\qquad=-\alpha\EE[\exp(-\alpha\TE(\OX_0))\langle\nabla\TE(\OX_0),\OV_0 \rangle]
		- \alpha \EE\left[\int_0^t \left\langle \exp(-\alpha\TE(\OX_s)) \OV_s, \nabla^2 \TE(\OX_s) \OV_s \right\rangle ds\right] \\
		&\qquad\quad\, + \alpha^2 \EE\left[\int_0^t \exp(-\alpha\TE(\OX_s)) \abs{\left\langle\nabla\TE(\OX_s), \OV_s \right\rangle}^2 ds\right]
		-\alpha \EE\left[\int_0^t \exp(-\alpha\TE(\OX_s)) \left\langle \nabla \TE(\OX_s), -\frac{\gamma}{m}\OV_s \right\rangle ds\right] \\
		&\qquad\quad\, -\alpha \EE\left[\int_0^t \exp(-\alpha\TE(\OX_s)) \left\langle \nabla \TE(\OX_s), \frac{\lambda}{m}(x_\alpha(\rho_{X,s})-\OX_s) \right\rangle ds \right].
\end{split}
\end{equation}
Differentiating  both sides of \eqref{eq:evolution_expE_firstderivative} with respect to the time~$t$ yields
\begin{equation} \label{eq:evolution_expE_secondderivative}
\begin{split}
	\frac{d^2}{dt^2}\EE[\exp(-\alpha\TE(\OX_t))]
	&=-\alpha \EE[\langle  \exp(-\alpha\TE(\OX_t))\OV_t,\nabla^2\TE(\OX_t)\OV_t \rangle]
		+ \alpha^2 \EE[\exp(-\alpha\TE(\OX_t)) \absnormal{\langle \nabla\TE(\OX_t), \OV_t \rangle}^2]\\
	&\quad\, + \frac{\alpha\gamma}{m} \EE[\exp(-\alpha\TE(\OX_t)) \langle \nabla\TE(\OX_t), \OV_t \rangle ]
		- \frac{\alpha\lambda}{m} \EE[ \exp(-\alpha\TE(\OX_t)) \langle \nabla \TE(\OX_t), x_\alpha(\rho_{X,t})-\OX_t \rangle]\\
	& \geq -\frac{\gamma}{m} \frac{d}{dt}\EE[\exp(-\alpha\TE(\OX_t))]
		-\alpha\underbrace{\EE[\langle \exp(-\alpha\TE(\OX_t))\OV_t,\nabla^2\TE(\OX_t)\OV_t \rangle]}_{T_1}\\
	&\quad\, -\frac{\alpha\lambda}{m}\underbrace{\EE[\exp(-\alpha\TE(\OX_t)) \langle \nabla\TE(\OX_t), x_\alpha(\rho_{X,t})-\OX_t \rangle]}_{T_2},
\end{split}
\end{equation}
where we employed the first line of~\eqref{eq:evolution_expE_firstderivative} in the last step.
It remains to upper bound the terms~$T_1$ and~$T_2$.
Making use of Assumptions~\ref{asm:zero_global} and~\ref{asm:Hessian}, we immediately obtain
\begin{equation} \label{eq:auxiliary_T1}
	T_1
	\leq \EE[\exp(-\alpha\TE(\OX_t)) \normnormal{\nabla^2 \TE}_\infty \absnormal{\OV_t}^2]
	\leq e^{-\alpha \underline{\TE}} C_\TE \EE[\absnormal{\OV_t}^2].
\end{equation}
For $T_2$, again under Assumptions~\ref{asm:zero_global} and~\ref{asm:Hessian}, we first note that 
\begin{equation} \label{eq:auxiliary_T2_prel}
\begin{split}
	T_2
	= -\EE\!\left[\exp(-\alpha\TE(\OX_t))\left\langle \nabla \TE(\OX_t)-\nabla \TE(x_\alpha(\rho_{X,t})), \OX_t-x_\alpha(\rho_{X,t}) \right\rangle\right] 
	\leq e^{-\alpha\underline\TE} C_\TE\EE[\absnormal{\OX_t-x_\alpha(\rho_{X,t})}^2],
\end{split}
\end{equation}
where the equality is a consequence of $\EE[\exp(-\alpha\TE(\OX_t))\langle\nabla\TE(x_\alpha(\rho_{X,t})), \OX_t-x_\alpha(\rho_{X,t})\rangle]=0$, which follows from the definition of~$x_\alpha(\rho_{X,t})$.
Bounding $\EE[\absnormal{\OX_t-x_\alpha(\rho_{X,t})}^2]$ as in~\eqref{eq:boundX-consensuspoint} we can further bound \eqref{eq:auxiliary_T2_prel} as
\begin{equation} \label{eq:auxiliary_T2}
	T_2
	\leq e^{-
		\alpha\underline \TE}C_\TE\EE[\absnormal{\OX_t-x_\alpha(\rho_{X,t})}^2]
	\leq 2e^{-2\alpha\underline \TE}C_\TE\frac{\EE[\absnormal{\OX_t-\EE[\OX_t]}^2]}{\EE[\exp(-\alpha\TE(\OX_t))]}.
\end{equation}
Collecting the estimates~\eqref{eq:auxiliary_T1} and~\eqref{eq:auxiliary_T2} within~\eqref{eq:evolution_expE_secondderivative} and inserting the result into~\eqref{eq:evolution_expE2} give
\begin{equation*}
\begin{split}
	\frac{1}{2}\frac{d^2}{dt^2}\left(\EE[\exp(-\alpha\TE(\OX_t))]\right)^2
	& \geq -\frac{\gamma}{m}\EE[\exp(-\alpha\TE(\OX_t))] \frac{d}{dt}\EE[\exp(-\alpha\TE(\OX_t))] \\
	&\quad\, -\EE[\exp(-\alpha\TE(\OX_t))]\alpha C_\TE e^{-\alpha \underline{\TE}} \EE [\absnormal{\OV_t}^2]
	 -\frac{2\alpha\lambda}{m} e^{-2\alpha\underline \TE}C_\TE\EE[\absnormal{\OX_t-\EE[\OX_t]}^2] \\
	 %& \geq -\frac{\gamma}{2m} \frac{d}{dt} \left(\EE[\exp(-\alpha\TE(\OX_t))]\right)^2 
	 	%-\alpha e^{-2\alpha \underline{\TE}} C_\TE \EE [\absnormal{\OV_t}^2]
	 	%-\frac{2\alpha\lambda}{m} e^{-2\alpha\underline \TE}C_\TE\EE[\absnormal{\OX_t-\EE[\OX_t]}^2]\\
	 & \geq -\frac{\gamma}{2m} \frac{d}{dt} \left(\EE[\exp(-\alpha\TE(\OX_t))]\right)^2 
	 	-\alpha e^{-2\alpha \underline{\TE}} C_\TE \left(\EE [\absnormal{\OV_t}^2]+\frac{2\lambda}{m} \EE[\absnormal{\OX_t-\EE[\OX_t]}^2]\right),
\end{split}
\end{equation*}
which yields the statement after employing the lower bound of~\eqref{eq:equivalence_H} as in Lemma~\ref{lem:equivalence_H}.
\end{proof}

We are now ready to state and prove the main result about the convergence of the mean-field PSO dynamics~\eqref{eq:PSO_without_memory_mf} without memory mechanisms to the global minimizer~$x^*$.

\begin{thm} \label{thm:convergence}
	Let $\TE$ satisfy Assumptions~\ref{asm:zero_global}--\ref{asm:Hessian} and let $(\OX_t,\OV_t)_{t\geq 0}$ be a solution to the nonlinear SDE~\eqref{eq:PSO_without_memory_mf}.
	Moreover, let us assume the well-preparation of the initial datum~$\OX_0$ and~$\OV_0$ in the sense that
	\begin{enumerate}[label=P\arabic*,labelsep=10pt,leftmargin=35pt]
		\item\label{asm:well_preparedness_parameters}
			$\mu>0$ with
			\begin{equation*}
				\mu:=\frac{\lambda\gamma}{2m^2}-\left(\frac{2\lambda^2}{\gamma m}+\frac{\sigma^2}{m^2}\right)\frac{4e^{-\alpha\underline\TE}}{\EE[\exp(-\alpha\TE(\OX_0))]},
			\end{equation*}
		\item\label{asm:well_preparedness_initial}
			it holds
%			\begin{equation*}
%				\frac{m\alpha}{\gamma} \frac{\left(\EE[\exp(-\alpha\TE(\OX_0))\left\langle\nabla\TE(\OX_0), \OV_0\right\rangle]\right)_+}{\EE[\exp(-\alpha\TE(\OX_0))]}
%				+\frac{\alpha e^{-2\alpha \underline{\TE}} C_\TE}{\chi(\frac{\gamma}{m}-\chi)} \left(1+2\frac{\lambda}{m}\left(\frac{2m}{\gamma}\right)^{2} \right) \frac{\EE[\CH(0)]}{\left(\EE[\exp(-\alpha\TE(\OX_0))]\right)^2}
%				< \frac{3}{16},
%			\end{equation*}
			\begin{equation*}
				\qquad\quad
				\frac{m\alpha}{2\gamma} \frac{\left(\EE[\left\langle\exp(-\alpha\TE(\OX_0))\nabla\TE(\OX_0), \OV_0\right\rangle]\right)_+}{\EE[\exp(-\alpha\TE(\OX_0))]}
				+\frac{\alpha C_\TE}{\chi(\frac{\gamma}{m}-\chi)} \left(1+\frac{8m\lambda}{\gamma^2}\right) \frac{\EE[\CH(0)]}{\left(\EE[\exp(-\alpha(\TE(\OX_0)-\underline{\TE}))]\right)^2}
				< \frac{3}{16},
			\end{equation*}
%			\begin{equation}
%			\begin{split}
%				&\frac{m\alpha}{\gamma}\EE[\exp(-\alpha\TE(\OX_0))]\left(\EE[\exp(-\alpha\TE(\OX_0))\left\langle\nabla\TE(\OX_0), \OV_0\right\rangle]\right)_+
%				+\frac{\alpha e^{-2\alpha \underline{\TE}} C_\TE}{\chi(\frac{\gamma}{m}-\chi)} \left(1+2\frac{\lambda}{m}\left(\frac{2m}{\gamma}\right)^{2} \right) \EE[\CH(0)] \\
%				&\qquad\qquad\qquad\qquad<\frac{3}{16} \left(\EE[\exp(-\alpha\TE(\OX_0))]\right)^2,
%			\end{split}
%			\end{equation}
			with $x_+=\max\{x,0\}$ for $x\in\RR$ denoting the positive part and where
			\begin{equation*}
				\chi := \frac{2}{3}\frac{\min\{\gamma/m,\mu\}}{\big(\!\left(\gamma/(2m)\right)^2+1\big)}.
			\end{equation*}
	\end{enumerate}
	Then $\EE[\CH(t)]$ with $\CH$ as defined in Equation~\eqref{eq:H} converges exponentially fast with rate~$\chi$ to $0$ as $t\to \infty$.
	Moreover, there exists  some $\tilde x$, which may depend on $\alpha$ and $f_0$, such that $\EE[\OX_t]\to \tilde x$ and $x_\alpha(\rho_{X,t})\to \tilde x$ exponentially fast with rate~$\chi/2$ as $t\to \infty$.
	Eventually, for any given accuracy~$\varepsilon>0$, there exists $\alpha_0>0$, such that for all $\alpha>\alpha_0$, $\tilde x$ satisfies
	\begin{equation*}
		\TE(\tilde x)-\underline{\TE} \leq \varepsilon.
	\end{equation*}
	If $\TE$ additionally satisfies Assumption~\ref{asm:icp}, we have $\abs{\tilde x-x^*}\leq\varepsilon^\nu/\eta$.
\end{thm}

\begin{remark} \label{remark:parameter_choice2}
	As suggested in Remark~\ref{remark:parameter_choice}, Theorem~\ref{thm:convergence} traces back the evolution of $\EE[\exp(-\alpha\TE(\OX_t))]$ to its initial state by employing Lemma~\ref{lem:evolution_expE2}.
	This allows to fixate all parameters of PSO at initialization time.
	By replacing $D_t^X$ with $2D_0^X$ in~\eqref{eq:remark:parameter_choice}, the well-preparation of the parameters as in Condition~\ref{asm:well_preparedness_parameters} can be ensured.
	
	Condition~\ref{asm:well_preparedness_initial} requires the well-preparation of the initialization in the sense that the initial datum~$f_0$ is both well-concentrated and to a certain extent not too far from an optimal value.
	While this might have a locality flavor, the condition is generally fulfilled in practical applications.
	Moreover, for CBO methods there is recent work where such assumption about the initial datum is reduced to the absolute minimum~\cite{fornasier2021consensus,fornasier2021convergence}.
\end{remark}

\begin{proof}[Proof of Theorem~\ref{thm:convergence}]
Let us define the time horizon 
\begin{equation*}
	T := \inf\left\{t\geq 0:\EE[\exp(-\alpha\TE(\OX_t))] < \frac{1}{2} \EE[\exp(-\alpha\TE(\OX_0))]     \right\} \quad\text{with }\inf \emptyset=\infty.
\end{equation*}
Obviously, by continuity, $T>0$.
We claim that $T=\infty$, which we prove by contradiction in the following.
Therefore, assume $T<\infty$. Then, for $t\in[0,T]$, we have
\begin{equation*}
	\frac{\lambda\gamma}{2m^2}-\left(\frac{2\lambda^2}{\gamma m}+\frac{\sigma^2}{m^2}\right)\frac{2e^{-\alpha\underline\TE}}{\EE[\exp(-\alpha\TE(\OX_t))]}
	\geq
	\frac{\lambda\gamma}{2m^2}-\left(\frac{2\lambda^2}{\gamma m}+\frac{\sigma^2}{m^2}\right)\frac{4e^{-\alpha\underline\TE}}{\EE[\exp(-\alpha\TE(\OX_0))]}
	= \mu > 0,
\end{equation*}
where the positivity of $\mu$ is due to the well-preparation condition~\ref{asm:well_preparedness_parameters} of the initialization.
Lemma~\ref{lem:evolution_H} then provides an upper bound for the time derivative of the functional~$\EE[\CH(t)]$, 
\begin{equation} \label{eq:evolution_H_2}
\begin{split}
	\frac{d}{dt}\EE[\CH(t)]
	&\leq -\frac{\gamma}{m}\EE[\absnormal{\OV_t}^2]-\mu\EE[\absnormal{\OX_t-\EE[\OX_t]}^2]
	\leq -\min\left\{\frac{\gamma}{m},\mu\right\} \left(\EE[\absnormal{\OX_t-\EE[\OX_t]}^2] + \EE[\absnormal{\OV_t}^2]\right)\\
	&\leq -\frac{2}{3}\frac{\min\{\gamma/m,\mu\}}{\big(\!\left(\gamma/(2m)\right)^2+1\big)}\EE[\CH(t)]
	=: -\chi\EE[\CH(t)],
\end{split}
\end{equation}
where we made use of the upper bound of~\eqref{eq:equivalence_H} as in Lemma~\ref{lem:equivalence_H} in the last inequality.
The rate~$\chi$ is defined implicitly and it is straightforward to check that $\chi<\gamma/m$.
Gr\"onwall's inequality implies
\begin{equation} \label{eq:decay_H}
	\EE[\CH(t)] \leq
	%\EE[\CH(0)]\exp\left(\frac{2\min\{\mu,\gamma/m\}}{3\big(\!\left(\gamma/(2m)\right)^2+1\big)}t\right) = 
	\EE[\CH(0)]\exp(-\chi t).
\end{equation}
Let us now investigate the evolution of the functional~$\mc{X}(t) := \left(\EE[\exp(-\alpha\TE(\OX_t))]\right)^2$.
First note that
\begin{equation*}
	\dot{\mc{X}}(0)
	:= \frac{d}{dt} \mc{X}(t)\bigr|_{t=0}
	= -2\alpha\EE[\exp(-\alpha\TE(\OX_0))] \EE[\exp(-\alpha\TE(\OX_0))\left\langle\nabla\TE(\OX_0), \OV_0\right\rangle].
\end{equation*}
Then, an application of Gr\"{o}nwall's inequality to Equation~\eqref{eq-inequality_lemE2} from Lemma~\ref{lem:evolution_expE2} and using the explicit bound of $\EE[\CH(t)]$ from~\eqref{eq:decay_H} yields
\begin{equation*}
\begin{split}
	\frac{d}{dt} \mc{X}(t)
	&\geq \dot{\mc{X}}(0) \exp\left(-\frac{\gamma}{m}t\right)
		- 4\alpha e^{-2\alpha \underline{\TE}} C_\TE \left(1+2\frac{\lambda}{m}\left(\frac{2m}{\gamma}\right)^{2} \right) \int_0^t \EE[\CH(s)] \exp\left(-\frac{\gamma}{m}(t-s)\right)ds \\
	%&\geq \dot{\mc{X}}(0) \exp\left(-\frac{\gamma}{m}t\right)
	%	- 4\alpha e^{-2\alpha \underline{\TE}} C_\TE \left(1+2\frac{\lambda}{m}\left(\frac{2m}{\gamma}\right)^{2} \right) \EE[\CH(0)] \int_0^t \exp\left(-\chi s-\frac{\gamma}{m}(t-s)\right)\,ds \\
	&\geq \dot{\mc{X}}(0) \exp\left(-\frac{\gamma}{m}t\right)
		- 4\alpha e^{-2\alpha \underline{\TE}} C_\TE \left(1+2\frac{\lambda}{m}\left(\frac{2m}{\gamma}\right)^{2} \right) \EE[\CH(0)] \frac{1}{\gamma/m-\chi} \left(\exp\left(-\chi t\right)-\exp\left(-\frac{\gamma}{m}t\right)\right) \\
	&\geq \dot{\mc{X}}(0) \exp\left(-\frac{\gamma}{m}t\right)
		- 4\alpha e^{-2\alpha \underline{\TE}} C_\TE \left(1+2\frac{\lambda}{m}\left(\frac{2m}{\gamma}\right)^{2} \right) \EE[\CH(0)] \frac{1}{\gamma/m-\chi} \exp\left(-\chi t\right),  
\end{split}
\end{equation*}
which, in turn, implies
\begin{equation*}
	\mc{X}(t)
	\geq \mc{X}(0)
		-\frac{m}{\gamma} \big(-\dot{\mc{X}}(0)\big)_{+}
		- \frac{4\alpha e^{-2\alpha \underline{\TE}} C_\TE}{\chi(\gamma/m-\chi)} \left(1+2\frac{\lambda}{m}\left(\frac{2m}{\gamma}\right)^{2} \right) \EE[\CH(0)]
\end{equation*}
after discarding the positive parts.
Recalling the definition of~$\mc{X}$ and employing the second well-preparation condition~\ref{asm:well_preparedness_initial}, we can deduce that for all $t\in[0,T]$ it holds
\begin{equation*}
\begin{split}
	\left(\EE[\exp(-\alpha\TE(\OX_t))]\right)^2
	&\geq \left(\EE[\exp(-\alpha\TE(\OX_0))]\right)^2
		-\frac{2m\alpha}{\gamma}\EE[\exp(-\alpha\TE(\OX_0))] \left(\EE[\exp(-\alpha\TE(\OX_0))\left\langle\nabla\TE(\OX_0), \OV_0\right\rangle]\right)_+ \\
	&\quad\, -\frac{4\alpha e^{-2\alpha \underline{\TE}} C_\TE}{\chi(\gamma/m-\chi)} \left(1+2\frac{\lambda}{m}\left(\frac{2m}{\gamma}\right)^{2} \right) \EE[\CH(0)]
	> \frac{1}{4} \left(\EE[\exp(-\alpha\TE(\OX_0))]\right)^2,
\end{split}
\end{equation*}
which entails that there exists $\delta>0$ such that $\EE[\exp(-\alpha\TE(\OX_t))]\geq \EE[\exp(-\alpha\TE(\OX_0))]/2$ in $[T,T+\delta]$ as well, contradicting the definition of $T$ and therefore showing the claim $T=\infty$.

\noindent
As a consequence of~\eqref{eq:decay_H} we have
\begin{equation} \label{eq:decay_H2}
	\EE[\CH(t)] \leq \EE[\CH(0)]\exp(-\chi t)
	\quad\text{and}\quad
	\EE[\exp(-\alpha\TE(\OX_t))]\geq\frac{1}{2}\EE[\exp(-\alpha\TE(\OX_0))]
\end{equation}
for all $t\geq0$.
In particular, by means of Lemma~\ref{lem:equivalence_H}, for a suitable generic constant $C>0$, we infer
\begin{equation} \label{eq:decay_rest}
	\EE[\absnormal{\OX_t-\EE[\OX_t]}^2] \leq C\exp(-\chi t),
	\quad
	\EE[\absnormal{\OV_t}^2] \leq C\exp(-\chi t),
	\quad\text{and}\quad
	\EE[\absnormal{\OX_t-x_\alpha(\rho_{X,t})}^2] \leq C\exp(-\chi t),
\end{equation}
where the last inequality uses the fact~\eqref{eq:boundX-consensuspoint}.
%$C_1 = 2(2m/\gamma)^2\EE[\CH(0)]$, $C_2 = 2\EE[\CH(0)]$, and $C_3 = 4e^{-\alpha\underline\TE}C_1/\EE[\exp(-\alpha\OX_0)]$
Moreover, with Jensen's inequality,
\begin{equation*}
	\abs{\frac{d}{dt}\EE[\OX_t]}
	\leq \EE[\absnormal{\OV_t}]
	\leq C\exp\left(-\chi t/2\right) \to 0 \quad\text{as } t\to\infty,
\end{equation*}
showing that $\EE[\OX_t]\to \tilde x$ for some $\tilde x\in\RR^d$, which may depend on $\alpha$ and $f_0$.
According to \eqref{eq:decay_rest}, $\OX_t\to \tilde x$ in mean-square and $x_\alpha(\rho_{X,t})\to \tilde x$, since
\begin{equation*}
	\absnormal{x_\alpha(\rho_{X,t})-\tilde x}^2
	\leq 3\EE[\absnormal{x_\alpha(\rho_{X,t})-\OX_t}^2]+3\EE[\absnormal{\OX_t-\EE\OX_t}^2]+3\absnormal{\EE\OX_t-\tilde x}^2 \to 0 \quad\text{as } t\to\infty.
\end{equation*}
Eventually, by continuity of the objective function~$\TE$ and  by the dominated convergence theorem, we conclude that $\EE[\exp(-\alpha\TE(\OX_t))]\to e^{-\alpha\TE(\tilde x)}$ as $t\to\infty$.
Using this when taking the limit~$t\to\infty$ in the second bound of~\eqref{eq:decay_H2} after applying the logarithm and multiplying both sides with $-1/\alpha$, we obtain
\begin{equation} \label{eq:E(tildex)}
	\TE(\tilde x) 
	= \lim_{t\to\infty}\left(-\frac{1}{\alpha}\log\EE[\exp(-\alpha\TE(\OX_t))]\right)
	\leq-\frac{1}{\alpha}\log\EE[\exp(-\alpha\TE(\OX_0))] + \frac{1}{\alpha}\log 2.
\end{equation}
The Laplace principle~\eqref{eq:laplace_principle} on the other hand allows to choose $\tilde\alpha\gg1$ large enough such that for given $\varepsilon>0$ it holds $-\frac{1}{\alpha}\log\EE[\exp(-\alpha\TE(\OX_0))]-\underline{\TE} < \varepsilon/2$
%\begin{equation*}
%	-\frac{1}{\alpha}\log\EE[\exp(-\alpha\TE(\OX_0))]-\underline{\TE} < \frac{\varepsilon}{2}
%\end{equation*}
for any $\alpha\geq\tilde\alpha$.
Together with~\eqref{eq:E(tildex)}, this establishes $0 \leq \TE(\tilde x)-\underline{\TE}
	\leq \varepsilon/2 + (\log 2)/\alpha \leq \varepsilon$
%\begin{equation*}
%	0 \leq \TE(\tilde x)-\underline{\TE}
%	\leq \frac{\varepsilon}{2} + \frac{1}{\alpha}\log 2 \leq \varepsilon
%\end{equation*}
for $\alpha\geq\max\{\tilde\alpha,(2\log 2)/\varepsilon\}$.
Finally, under the inverse continuity property~\ref{asm:icp} we additionally have 
$\abs{\tilde x-x^*} \leq (\TE(\tilde x)-\underline\TE)^{\nu}/\eta \leq \varepsilon^\nu/\eta$,
%\begin{equation*}
%	\abs{\tilde x-x^*}
%	\leq (\TE(\tilde x)-\underline\TE)^{\nu}/\eta
%	\leq \varepsilon^\nu/\eta,
%\end{equation*}
concluding the proof.
\end{proof}

%%%%%%%%%%%%%%%%%%%%%%%%%%%%%%%%%%%%%%%%
%%%%%%%%%%%%%%%%%%%%%%%%%%%%%%%%%%%%%%%%
%%%%%%%%%%%%%%%%%%%%%%%%%%%%%%%%%%%%%%%%
\section{Mean-Field Analysis of PSO with Memory Effects} \label{sec:PSO_with_memory}

\noindent
Let us now turn back to the PSO dynamics~\eqref{eq:PSO_with_memory} described in the introduction.
The fundamental difference to what was analyzed in the preceding section is the presence of a personal memory of each particle, encoded through the additional state variable~$Y^i_t$.
It can be thought of as an approximation to the in-time best position~$\argmin_{\tau\leq t} \TE(X^i_\tau)$ seen by the respective particle.
Its dynamics is encoded in Equation~\eqref{eq:PSO_with_memory_Y}.

In this section we analyze~\eqref{eq:PSO_with_memory} in the large particle limit, i.e., through its mean-field limit~\eqref{eq:PSO_with_memory_mf}.

%%%%%%%%%%%%%%%%%%%%%%%%%%%%%%%%%%%%%%%%
%%%%%%%%%%%%%%%%%%%%%%%%%%%%%%%%%%%%%%%%
\subsection{Well-Posedness of PSO with Memory Effects} \label{sec:wellposedness_with_memory}

Ensured by a sufficiently regularized implementation of the local best position~$\OY$, we can show the well-posedness of the mean-field PSO dynamics~\eqref{eq:PSO_with_memory_mf}, respectively, the associated Vlasov-Fokker-Planck equation~\eqref{eq:PSO_with_memory_weak}.

\begin{thm} \label{thm:wellposedness_memory}
	Let $\TE$ satisfy Assumptions~\ref{asm:zero_global}--\ref{asm:growth} and let $m,\gamma,\lambda_1,\lambda_2,\sigma_1,\sigma_2,\alpha,\beta,\theta,\kappa,T>0$.
	If $(\OX_0,\OY_0,\OV_0)$ is distributed according to $f_0\in \mc{P}_4(\RR^{3d})$,
	then the nonlinear SDE~\eqref{eq:PSO_with_memory_mf} admits a unique strong solution up to time~$T$ with $\mc{C}([0,T],\RR^{d})\times\mc{C}([0,T],\RR^{d})\times\mc{C}([0,T],\RR^{d})$-valued paths.
	The associated law~$f$ has regularity~$\mc{C}([0,T],\mc{P}(\RR^{3d}))$ and is a weak solution to the Vlasov-Fokker-Planck equation~\eqref{eq:PSO_with_memory_weak}.
	In particular,
	\begin{equation} \label{eq:wellposedness_regularity_memory}
		\sup_{t\in[0,T]} \EE[\absnormal{\OX_t}^4+\absnormal{\OY_t}^4+\absnormal{\OV_t}^4]\leq \left(1+3\EE[\absnormal{\OX_0}^4+\absnormal{\OY_0}^4+\absnormal{\OV_0}^4]\right)e^{CT}
	\end{equation}
	for some constant $C>0$ depending only on $m,\gamma,\lambda_1,\lambda_2,\sigma_1,\sigma_2,\alpha,\beta,\theta,\kappa,c_\TE, R$ and $L_\TE$.
\end{thm}

\begin{proof}[Proof sketch]
	The proof follows the steps taken in~\cite[Theorems~3.1, 3.2]{carrillo2018analytical}.
	
	\noindent\textit{Step 1:} For a given function $u\in\mathcal{C}([0,T],\mathbb{R}^d)$ and an initial measure~$f_0\in\CP_4(\RR^{3d})$, according to standard SDE theory~\cite[Chapter~7]{arnold1974stochasticdifferentialequations}, we can uniquely solve the auxiliary SDE
		\begin{align*}
			d\widetilde{X}_t   &= \widetilde{V}_t \,dt,\\
			d\widetilde{Y}_{t} &= \kappa \big(\widetilde{X}_{t}-\widetilde{Y}_{t}\big)\, S^{\beta,\theta}\big(\widetilde{X}_{t}, \widetilde{Y}_{t}\big)\,dt,\\
			m\,d\widetilde{V}_{t} &=
			\begin{aligned}[t]
				 &\!-\gamma \widetilde{V}_{t} \,dt
				    + \lambda_{1}\big(\widetilde{Y}_{t}-\widetilde{X}_{t}\big)\, dt
				    +\lambda_{2}\big(u_t-\widetilde{X}_{t}\big)\, dt
				    +\sigma_{1} D\big(\widetilde{Y}_{t}-\widetilde{X}_{t}\big)\, dB_{t}^1
				    +\sigma_{2} D\big(u_t-\widetilde{X}_{t}\big)\, dB_{t}^2,
			\end{aligned}
		\end{align*}
		with initial condition~$\big(\widetilde{X}_0,\widetilde{Y}_0,\widetilde{V}_0\big) \sim f_0$
		as, due to the smoothness of~$S^{\beta,\theta}$ and Assumptions~\ref{asm:Lipschitz} and~\ref{asm:growth}, the coefficients are locally Lipschitz and have at most linear growth.
		This induces $\widetilde f_t=\mathrm{Law}\big(\widetilde{X}_t,\widetilde{Y}_t,\widetilde{V}_t\big)$.
		Moreover, the regularity of $f_0 \in \CP_4(\RR^{3d})$ allows for a moment estimate
		%of the form~$\mathbb{E} [\absnormal{\widetilde{X}_t}_2^4+\absnormal{\widetilde{Y}_t}_2^4+\absnormal{\widetilde{V}_t}_2^4] \leq \big(1+\mathbb{E}[\absnormal{\widetilde{X}_0}_2^4+\absnormal{\widetilde{Y}_0}_2^4+\absnormal{\widetilde{V}_0}_2^4]\big)e^{ct}$.
		of the form~\eqref{eq:wellposedness_regularity_memory} and thus $\widetilde f\in\CC([0,T],\CP_4(\RR^{3d}))$.
		In what follows, $\widetilde\rho_Y$ denotes the spatial local best marginal of $\widetilde f$, i.e., $\widetilde\rho_Y(t,\,\cdot\,)=\iint_{\RR^{2d}} d\widetilde{f}(t,x, \,\cdot\,,v)$.
%		\begin{equation*}
%			\widetilde\rho_Y(t,y)=\iint_{\RR^{2d}} \widetilde{f}(t,dx, y,dv).
%		\end{equation*}
	
	\noindent\textit{Step 2:} Let us now define, for some constant $C>0$, the test function space
		\begin{equation} \label{eq:function_space_memory}
		\begin{split}
			\CC^2_{*}(\RR^{3d}) := \big\{\phi\in\CC^2(\RR^{3d}): \absnormal{\nabla_v\phi}\leq C\left(1+\absnormal{x}+\absnormal{y}+\absnormal{v}\right) \,\text{ and }\, \sup_{k=1,\dots,d}\norm{\partial^2_{v_kv_k} \phi}_\infty < \infty\big\}.
		\end{split}
		\end{equation}
		For some $\phi\in\CC^2_{*}(\RR^{3d})$, by the It\^o-Doeblin formula, we derive
		\begin{equation*}
		\begin{split}
			d\phi
			&= \nabla_x\phi\cdot\widetilde{V}_t\,dt 
				+ \kappa\nabla_y\phi\cdot \big(\widetilde{X}_{t}-\widetilde{Y}_{t}\big)\, S^{\beta,\theta}\big(\widetilde{X}_{t}, \widetilde{Y}_{t}\big)\,dt 
				+ \nabla_v\phi\cdot \left(-\frac{\gamma}{m} \widetilde{V}_{t}+ \frac{\lambda_{1}}{m}\big(\widetilde{Y}_{t}-\widetilde{X}_{t}\big)+\frac{\lambda_{2}}{m}\big(u_t-\widetilde{X}_{t}\big)\right)dt\\
			&\,\,+\frac{1}{2} \sum_{k=1}^d \partial^2_{v_kv_k}\phi \left(\frac{\sigma_1^2}{m^2}\big(\widetilde{Y}_{t}-\widetilde{X}_{t}\big)_k^2 + \frac{\sigma_2^2}{m^2}\big(u_t-\widetilde{X}_{t}\big)_k^2\right)dt
				+\nabla_v\phi \cdot \left(\frac{\sigma_{1}}{m} D\big(\widetilde{Y}_{t}-\widetilde{X}_{t}\big)\, dB_{t}^1 + \frac{\sigma_{2}}{m} D\big(u_t-\widetilde{X}_{t}\big)\, dB_{t}^2\right),
		\end{split}
		\end{equation*}
		where we mean $\phi\big(\widetilde{X}_t,\widetilde{Y}_t,\widetilde{V}_t\big)$ whenever we write $\phi$.
		After taking the expectation, applying Fubini's theorem and observing that the stochastic integrals vanish due to the definition of the test function space~$\CC^2_{*}(\RR^{3d})$ and the regularity~\eqref{eq:wellposedness_regularity_memory}, we observe that~$\widetilde f\in\CC([0,T],\CP_4(\RR^{3d}))$ satisfies the Vlasov-Fokker-Planck equation
		\begin{align} \label{eq:PSO_with_memory_weak_wellposedness_aux}
		\begin{aligned}
			\frac{d}{dt}\iiint_{\RR^{3d}} \phi\,d\widetilde f_t =
			& \iiint_{\RR^{3d}} v\cdot\nabla_x\phi\,d\widetilde f_t + \iiint_{\RR^{3d}} \kappa(x-y) S^{\beta,\theta}(x,y) \cdot\nabla_y\phi\,d\widetilde f_t \\
			& -\iiint_{\RR^{3d}} \left(\frac{\gamma}{m} v + \frac{\lambda_{1}}{m}\left(x-y\right) + \frac{\lambda_{2}}{m}\left(x-u_t\right)\right) \cdot\nabla_v \phi\,d\widetilde f_t \\
			&+ \iiint_{\RR^{3d}} \sum_{k=1}^d \left(\frac{\sigma_{1}^{2}}{2 m^{2}} \left(x-y\right)_k^{2} + \frac{\sigma_{2}^{2}}{2m^2} \left(x-u_t\right)_k^2\right) \cdot \partial^2_{v_kv_k}\phi\,d\widetilde f_t.
		\end{aligned}
		\end{align}
	
	\noindent\textit{Step 3:} Setting $\CT u:=y_\alpha(\widetilde\rho_Y)\in\mathcal{C}([0,T],\mathbb{R}^d)$ provides the self-mapping property of the map
		\begin{align*}
			\CT:\CC([0,T],\mathbb{R}^d)\rightarrow\CC([0,T],\RR^d), \quad u\mapsto\mathcal{T}u=y_\alpha(\widetilde\rho_Y),
		\end{align*}
		which is compact as a consequence of the stability estimate $\absnormal{y_\alpha(\widetilde\rho_{Y,t})-y_\alpha(\widetilde\rho_{Y,s})}_2 \lesssim W_2(\widetilde\rho_{Y,t},\widetilde\rho_{Y,s})$ for $\widetilde\rho_{Y,t},\widetilde\rho_{Y,s}\in\CP_4(\RR^d)$, see, e.g., \cite[Lemma~3.2]{carrillo2018analytical}, and the \mbox{H\"{o}lder-$1/2$} continuity of the Wasserstein-$2$ distance~$W_2(\widetilde\rho_{Y,t},\widetilde\rho_{Y,s})$.

	\noindent\textit{Step 4:} Then, for $u=\vartheta\CT u$ with $\vartheta\in[0,1]$, there exists $\widetilde f\in\CC([0,T],\CP_4(\RR^{3d}))$ satisfying~\eqref{eq:PSO_with_memory_weak_wellposedness_aux} with marginal~$\widetilde\rho_Y$ such that $u_t=\vartheta y_\alpha(\widetilde\rho_{Y,t})$.
		For such $u$, a uniform bound can be obtained as of Assumption~\ref{asm:growth}.
		An application of the Leray-Schauder fixed point theorem provides a solution to~\eqref{eq:PSO_with_memory_mf}.
		
	\noindent\textit{Step 5:} As for uniqueness, we assume the existence of two distinct fixed points~$u^1$ and~$u^2$ with associated processes~$\big(\widetilde{X}^1,\widetilde{Y}^1,\widetilde{V}^1\big)$ and $\big(\widetilde{X}^2,\widetilde{Y}^2,\widetilde{V}^2\big)$, respectively.
		Under the premise of the same initialization and identical Brownian motion paths, Gr\"onwall's inequality ensures that they coincide, cf.\@~\cite[Theorem 3.1]{carrillo2018analytical}.
\end{proof}

%%%%%%%%%%%%%%%%%%%%%%%%%%%%%%%%%%%%%%%%
%%%%%%%%%%%%%%%%%%%%%%%%%%%%%%%%%%%%%%%%
\subsection{Convergence of PSO with Memory Effects to a Global Minimizer} \label{sec:convergence_with_memory}

Analogously to Section~\ref{sec:convergence_without_memory} we define a functional~$\CH(t)$, which is analyzed in this section to eventually prove its exponential decay and thereby consensus formation at some~$\tilde x$ close to the global minimizer~$x^*$.
In addition to the requirements that the variance~$\EE[\absnormal{\OX_t-\EE[\OX_t]}^2]$ in the position and the second-order moment of the velocity~$\EE[\absnormal{\OV_t}^2]$ of the averaged particle vanish, we also expect that the particle's position~$\OX_t$ aligns with its personal best position~$\OY_t$ over time, meaning that $\EE[\absnormal{\OX_t-\OY_t}^2]$ decays to zero.
This motivates the definition
\begin{equation} \label{eq:H_memory}
\begin{split}
	\CH(t) := &\left(\frac{\gamma}{2m}\right)^2 \absnormal{\OX_t-\EE[\OX_t]}^2
		+ \frac{3}{2} \absnormal{\OV_t}^2
		+ \frac{1}{2}\left(\frac{3\lambda_1}{m}+\frac{\gamma^2}{m^2}\right)\absnormal{\OX_t-\OY_t}^2\\
		&\qquad\qquad + \frac{\gamma}{2m}\left\langle\OX_t-\EE[\OX_t], \OV_t\right\rangle
		+\frac{\gamma}{m}\left\langle\OX_t-\OY_t,\OV_t\right\rangle,
\end{split}
\end{equation}
whose last two terms are required for technical reasons.
Again, by the equivalence established in Lemma~\ref{lem:equivalence_H_memory}, proving the decay of $\EE[\CH(t)]$ directly entails the decay of $\EE[\absnormal{\OX_t-\EE[\OX_t]}^2+\absnormal{\OV_t}^2+\absnormal{\OX_t-\OY_t}^2]$ with the same rate.
\begin{lem} \label{lem:equivalence_H_memory}
	The functional $\CH(t)$ is equivalent to $\absnormal{\OX_t-\EE[\OX_t]}^2+\absnormal{\OV_t}^2+\absnormal{\OX_t-\OY_t}^2$ in the sense that
	\begin{equation} \label{eq:equivalence_H_memory}
	\begin{split}
		&\frac{1}{2}\left(\frac{\gamma}{2m}\right)^2\absnormal{\OX_t-\EE[\OX_t]}^2+\frac{1}{2}\absnormal{\OV_t}^2+\frac{3\lambda_1}{2m}\absnormal{\OX_t-\OY_t}^2
		\leq\CH(t)\\
		%&\qquad\qquad\qquad\qquad\leq \frac{3}{2}\left(\frac{\gamma}{2m}\right)^2\absnormal{\OX_t-\EE[\OX_t]}^2+\frac{5}{2}\absnormal{\OV_t}^2+\frac{1}{2}\left(\frac{3\lambda_1}{m}+\frac{2\gamma^2}{m^2}\right)\absnormal{\OX_t-\OY_t}^2 \\
		&\qquad\qquad\qquad\qquad\leq \frac{5}{2}\left(\left(\frac{\gamma}{2m}\right)^2+1+\frac{3\lambda_1}{m}+\frac{2\gamma^2}{m^2}\right)\left(\absnormal{\OX_t-\EE[\OX_t]}^2+\absnormal{\OV_t}^2+\absnormal{\OX_t-\OY_t}^2\right).
	\end{split}
	\end{equation}
\end{lem}
%\begin{proof}
%	Both inequalities follow from Young's inequality.
%\end{proof}

We now derive an evolution inequality of the quantity $\EE[\CH(t)]$.
\begin{lem} \label{lem:evolution_H_memory}
	Let $\TE$ satisfy Assumptions~\ref{asm:zero_global}--\ref{asm:growth} and let $(\OX_t,\OY_t,\OV_t)_{t\geq 0}$ be a solution to the nonlinear SDE~\eqref{eq:PSO_with_memory_mf}.
	Then $\EE[\CH(t)]$ with $\CH$ as defined in~\eqref{eq:H_memory} satisfies
	\begin{equation} \label{eq:evolution_H_memory}
	\begin{split}
		\frac{d}{dt}\EE[\CH(t)] \leq\,
		&-\frac{\gamma}{2m}\EE[\absnormal{\OV_t}^2]
		-\left(\frac{(\lambda_1+2\lambda_2)\gamma}{(2m)^2}-\left(\frac{9\lambda_2^2}{\gamma m}+\frac{3\sigma_2^2}{m^2}+\frac{3\lambda_1\gamma}{(2m)^2}\right)\frac{6e^{-\alpha\underline\TE}}{\EE[\exp(-\alpha\TE(\OY_t))]}\right)\EE[\absnormal{\OX_t-\EE[\OX_t]}^2]\\
		&-\bigg(\frac{(\lambda_1+\lambda_2)\gamma}{m^2}+\kappa\theta\left(\frac{3\lambda_1}{m}+\frac{\gamma^2}{m^2}\right)-\frac{8\kappa^2\gamma}{m}-\frac{\lambda_2^2\gamma}{2m^2\lambda_1}-\frac{3\sigma_1^2}{2m^2}\\
		&\qquad-\left(\frac{9\lambda_2^2}{\gamma m}+\frac{3\sigma_2^2}{m^2}\right)-\left(\frac{9\lambda_2^2}{\gamma m}+\frac{3\sigma_2^2}{m^2}+\frac{3\lambda_1\gamma}{(2m)^2}\right)\frac{12e^{-\alpha\underline\TE}}{\EE[\exp(-\alpha\TE(\OY_t))]}\bigg)\EE[\absnormal{\OX_t-\OY_t}^2].
	\end{split}
	\end{equation}
\end{lem}

\begin{proof}%[Proof of Lemma~\ref{lem:evolution_H_memory}]
	Let us write $\delta\OX_t:=\OX_t-\EE[\OX_t]$ for short and note that the integration by parts formula gives
	\begin{equation} \label{eq:time_evolution_variance_memory}
		\frac{d}{dt} \EE[\absnormal{\delta\OX_t}^2]
			=2\EE[\left\langle\delta\OX_t, \OV_t\right\rangle].
			%\leq \EE[\absnormal{\delta\OX_t}^2]+\EE[\absnormal{\OV_t}^2]
	\end{equation}
	Observe that the stochastic integrals have vanishing expectations as a consequence of the regularity obtained in Theorem~\ref{thm:wellposedness_memory}. An application of the It\^{o}-Doeblin formula and Young's inequality yields 
	\begin{equation} \label{eq:time_evolution_V2_memory}
	\begin{split}
		\frac{d}{dt} \EE[\absnormal{\OV_t}^2]
		&= -\frac{2\gamma}{m}\EE[\absnormal{\OV_t}^2]
			+\frac{2\lambda_1}{m}\EE[\left\langle\OV_t,\OY_t-\OX_t\right\rangle]
			+\frac{2\lambda_2}{m}\EE[\left\langle\OV_t,y_\alpha(\rho_{Y,t})-\OX_t\right\rangle]
			+\frac{\sigma_1^2}{m^2}\EE[\absnormal{\OY_t-\OX_t}^2]\\
			&\quad\,
			+\frac{\sigma_2^2}{m^2}\EE[\absnormal{y_\alpha(\rho_{Y,t})-\OX_t}^2] \\
		&\leq -\left(\frac{2\gamma}{m}-\frac{\lambda_2}{\varepsilon m}\right)\EE[\absnormal{\OV_t}^2]
			+\frac{\sigma_1^2}{m^2}\EE[\absnormal{\OY_t-\OX_t}^2]
			+\left(\frac{\varepsilon\lambda_2}{m}+\frac{\sigma_2^2}{m^2}\right)\EE[\absnormal{y_\alpha(\rho_{Y,t})-\OX_t}^2]\\
			&\quad\,
			-\frac{2\lambda_1}{m}\EE[\left\langle\OV_t,\OX_t-\OY_t\right\rangle],\quad\forall\, \varepsilon>0.
	\end{split}
	\end{equation}
%	\begin{equation} \label{term2_memory}
%	\begin{split}
%		\frac{3}{2}\frac{d}{dt} \EE[\absnormal{\OV_t}^2]
%		&= -\frac{3\gamma}{m}\EE[\absnormal{\OV_t}^2]
%			+\frac{3\lambda_1}{m}\EE[\left\langle\OV_t,\OY_t-\OX_t\right\rangle]
%			+\frac{3\lambda_2}{m}\EE[\left\langle\OV_t,y_\alpha(\rho_{Y,t})-\OX_t\right\rangle]\\
%			&\quad\,+\frac{3\sigma_1^2}{2m^2}\EE[\absnormal{\OY_t-\OX_t}^2]
%			+\frac{3\sigma_2^2}{2m^2}\EE[\absnormal{y_\alpha(\rho_{Y,t})-\OX_t}^2] \\
%		&\leq -\left(\frac{3\gamma}{m}-\frac{3\lambda_2}{2\varepsilon m}\right)\EE[\absnormal{\OV_t}^2]
%			+\frac{3\sigma_1^2}{2m^2}\EE[\absnormal{\OY_t-\OX_t}^2]
%			+\frac{3}{2}\left(\frac{\varepsilon\lambda_2}{m}+\frac{\sigma_2^2}{m^2}\right)\EE[\absnormal{y_\alpha(\rho_{Y,t})-\OX_t}^2]
%			-\frac{3\lambda_1}{m}\EE[\left\langle\OV_t,\OX_t-\OY_t\right\rangle]
%	\end{split}
%	\end{equation}
%	for any $\varepsilon>0$.
	Again by employing the It\^{o}-Doeblin formula we obtain
%	\begin{equation} \label{eq:time_evolution_scalarproductXV_memory}
%	\begin{split}
%		\frac{d}{dt}\EE[\left\langle\delta\OX_t, \OV_t\right\rangle]
%		&=\EE[\absnormal{\OV_t}^2]-\left(\EE[\OV_t]\right)^2-\frac{\gamma}{m}\EE[\left\langle\delta\OX_t,\OV_t\right\rangle]+\frac{\lambda_1}{m}\EE[\left\langle\delta\OX_t,\OY_t-\OX_t\right\rangle]+\frac{\lambda_2}{m}\EE[\left\langle\delta\OX_t,y_\alpha(\rho_{Y,t})-\OX_t\right\rangle]\\
%		&\leq\EE[\absnormal{\OV_t}^2]-\frac{\gamma}{2m}\frac{d}{dt} \EE[\absnormal{\delta\OX_t}^2]-\frac{\lambda_1+\lambda_2}{m}\EE[\absnormal{\delta\OX_t}^2]+\frac{\lambda_1}{m}\EE[\left\langle\delta\OX_t,\OY_t-\EE[\OX_t]\right\rangle]
%		%+\frac{\lambda_2}{m}\EE[\left\langle\delta\OX_t,y_\alpha(\rho_{Y,t})-\EE[\OX_t]\right\rangle]
%		\\
%		&= \EE[\absnormal{\OV_t}^2]-\frac{\gamma}{2m}\frac{d}{dt} \EE[\absnormal{\delta\OX_t}^2]-\frac{\lambda_1+\lambda_2}{m}\EE[\absnormal{\delta\OX_t}^2]+\frac{\lambda_1}{m}\EE[\left\langle\delta\OX_t,\OY_t-y_\alpha(\rho_{Y,t})\right\rangle],
%	\end{split}
%	\end{equation}
%	where, to obtain the second line, we used the identity~\eqref{eq:time_evolution_variance}, inserted $\pm\EE[\OX_t]$ in the last two terms and noted that $\EE[\left\langle\delta\OX_t,y_\alpha(\rho_{Y,t})-\EE[\OX_t]\right\rangle]=0$.
%	The latter observation is reused in the next-to-last step.
	\begin{equation*} %\label{eq:time_evolution_scalarproductXV_memory}
	\begin{split}
		\frac{d}{dt}\EE[\left\langle\delta\OX_t, \OV_t\right\rangle]
		&=\EE[\absnormal{\OV_t}^2]-\left(\EE[\OV_t]\right)^2-\frac{\gamma}{m}\EE[\left\langle\delta\OX_t,\OV_t\right\rangle]+\frac{\lambda_1}{m}\EE[\left\langle\delta\OX_t,\OY_t-\OX_t\right\rangle]\\
		&\qquad\qquad\;\;\,+\frac{\lambda_2}{m}\EE[\left\langle\delta\OX_t,y_\alpha(\rho_{Y,t})-\OX_t\right\rangle]\\
		&\leq\EE[\absnormal{\OV_t}^2]-\frac{\gamma}{2m}\frac{d}{dt} \EE[\absnormal{\delta\OX_t}^2]+\frac{\lambda_1}{m}\EE[\left\langle\delta\OX_t,\left(\OY_t-y_\alpha(\rho_{Y,t})\right)-\big(\OX_t-\EE[\OX_t]\big)\right\rangle]\\
		&\qquad\qquad\;\;\,+\frac{\lambda_2}{m}\EE[\left\langle\delta\OX_t,\EE[\OX_t]-\OX_t\right\rangle]\\
		&=\EE[\absnormal{\OV_t}^2]-\frac{\gamma}{2m}\frac{d}{dt} \EE[\absnormal{\delta\OX_t}^2]-\frac{\lambda_1+\lambda_2}{m}\EE[\absnormal{\delta\OX_t}^2]+\frac{\lambda_1}{m}\EE[\left\langle\delta\OX_t,\OY_t-y_\alpha(\rho_{Y,t})\right\rangle]\\
		&\leq\EE[\absnormal{\OV_t}^2]-\frac{\gamma}{2m}\frac{d}{dt} \EE[\absnormal{\delta\OX_t}^2]-\frac{\lambda_1+2\lambda_2}{2m}\EE[\absnormal{\delta\OX_t}^2]+\frac{\lambda_1}{2m}\EE[\absnormal{\OY_t-y_\alpha(\rho_{Y,t})}^2],
	\end{split}
	\end{equation*}
	where, for the second line, we used the identity~\eqref{eq:time_evolution_variance_memory} and that $\EE[\left\langle\delta\OX_t,\textbf{C}\right\rangle]=0$, whenever $\textbf{C}\in\RR^d$ is constant, allowing to expand the expression in the way done.	
	We now rearrange the previous inequality to get
	\begin{equation} \label{eq:time_evolution_X2_memory}
		\frac{\gamma}{2m}\frac{d}{dt} \EE[\absnormal{\delta\OX_t}^2]+\frac{d}{dt}\EE[\left\langle\delta\OX_t, \OV_t\right\rangle]
		\leq \EE[\absnormal{\OV_t}^2]-\frac{\lambda_1+2\lambda_2}{2m}\EE[\absnormal{\delta\OX_t}^2]+\frac{\lambda_1}{2m}\EE[\absnormal{\OY_t-y_\alpha(\rho_{Y,t})}^2].
	\end{equation}
	Next, using the It\^{o}-Doeblin formula, we compute
	\begin{equation} \label{eq:time_evolution_XY_memory}
	\begin{split}
		\frac{d}{dt} \EE[\absnormal{\OX_t-\OY_t}^2]
		&= 2\EE[\left\langle\OX_t-\OY_t, \OV_t-\kappa\left(\OX_{t}-\OY_{t}\right) S^{\beta,\theta}(\OX_{t}, \OY_{t})\right\rangle]
		\\
		&\leq 2\EE[\left\langle\OX_t-\OY_t, \OV_t\right\rangle]-2\kappa \theta\EE[\absnormal{\OX_t-\OY_t}^2],
	\end{split}
	\end{equation}
	where the last step follows from the fact that $\theta < S^{\beta,\theta}(\OX_{t}, \OY_{t})<2+\theta<4$.
	And lastly, the It\^{o}-Doeblin formula and Young's inequality allow to bound
	\begin{equation} \label{eq:time_evolution_scalarproductXYV_memory}
	\begin{split}
		&\frac{d}{dt}\EE[\left\langle\OX_t-\OY_t,\OV_t\right\rangle]
		=-\frac{\gamma}{m}\EE[\left\langle\OX_t-\OY_t,\OV_t\right\rangle]
			-\frac{\lambda_1+\lambda_2}{m}\EE[\absnormal{\OX_t-\OY_t}^2]
			+\frac{\lambda_2}{m}\EE[\left\langle\OX_t-\OY_t,y_\alpha(\rho_{Y,t})-\OY_t\right\rangle]\\
		&\qquad\qquad\qquad\qquad\quad\;\;\, +\EE[\left\langle\OV_t-\kappa\left(\OX_{t}-\OY_{t}\right) S^{\beta,\theta}(\OX_{t},\OY_{t}), \OV_t\right\rangle]\\
		&\qquad\leq -\frac{\gamma}{m}\EE[\left\langle\OX_t-\OY_t,\OV_t\right\rangle]
			-\frac{\lambda_1+\lambda_2}{m}\EE[\absnormal{\OX_t-\OY_t}^2]
			+\frac{\lambda_2^2}{2m\lambda_1}\EE[\absnormal{\OX_t-\OY_t}^2]
			+\frac{\lambda_1}{2m}\EE[\absnormal{y_\alpha(\rho_{Y,t})-\OY_t}^2] \\
		&\qquad\qquad\qquad\qquad\quad\;\;\, +\EE[\absnormal{\OV_t}^2]
			+\frac{1}{2}\EE[\absnormal{\OV_t}^2]
			+8\kappa^2 \EE[\absnormal{\OX_t-\OY_t}^2]\\
		&\qquad= -\left(\frac{\lambda_1+\lambda_2}{m}-8\kappa^2-\frac{\lambda_2^2}{2m\lambda_1}\right)\EE[\absnormal{\OX_t-\OY_t}^2]
			+\frac{3}{2}\EE[\absnormal{\OV_t}^2]
			+\frac{\lambda_1}{2m}\EE[\absnormal{y_\alpha(\rho_{Y,t})-\OY_t}^2]\\
			&\qquad\qquad\qquad\qquad\quad\;\;\, -\frac{\gamma}{m}\EE[\left\langle\OX_t-\OY_t,\OV_t\right\rangle].
	\end{split}
	\end{equation}
	We now collect the bounds~\eqref{eq:time_evolution_V2_memory}, \eqref{eq:time_evolution_X2_memory}, \eqref{eq:time_evolution_XY_memory}, and~\eqref{eq:time_evolution_scalarproductXYV_memory} to show
	\begin{equation*} %\label{eq:evolution_H_aux_memory}
	\begin{split}
		\frac{d}{dt}\EE[\CH(t)]
		&\leq -\left(\frac{3\gamma}{m}-\frac{3\lambda_2}{2\varepsilon m}-\frac{\gamma}{2m}-\frac{3\gamma}{2m}\right)\EE[\absnormal{\OV_t}^2]
			-\frac{(\lambda_1+2\lambda_2)\gamma}{(2m)^2}\EE[\absnormal{\delta\OX_t}^2] \\
		&\quad\, -\left(
			\frac{(\lambda_1+\lambda_2)\gamma}{m^2}
			-\frac{8\kappa^2\gamma}{m}-\frac{\lambda_2^2\gamma}{2m^2\lambda_1}
			+\kappa\theta\left(\frac{3\lambda_1}{m}+\frac{\gamma^2}{m^2}\right)
			-\frac{3\sigma_1^2}{2m^2}
			\right)\EE[\absnormal{\OX_t-\OY_t}^2] \\
		&\quad\, +\frac{3}{2}\left(\frac{\varepsilon\lambda_2}{m}+\frac{\sigma_2^2}{m^2}\right)\EE[\absnormal{y_\alpha(\rho_{Y,t})-\OX_t}^2]
			+\frac{3\lambda_1\gamma}{(2m)^2}\EE[\absnormal{y_\alpha(\rho_{Y,t})-\OY_t}^2]\\
	&\leq -\left(\frac{\gamma}{m}-\frac{3\lambda_2}{2\varepsilon m}\right)\EE[\absnormal{\OV_t}^2]
			-\frac{(\lambda_1+2\lambda_2)\gamma}{(2m)^2}\EE[\absnormal{\delta\OX_t}^2] \\
		&\quad\, -\left(\frac{(\lambda_1+\lambda_2)\gamma}{m^2}
			-\frac{8\kappa^2\gamma}{m}-\frac{\lambda_2^2\gamma}{2m^2\lambda_1}
			+\kappa\theta\left(\frac{3\lambda_1}{m}+\frac{\gamma^2}{m^2}\right)
			-\frac{3\sigma_1^2}{2m^2}
			-3\left(\frac{\varepsilon\lambda_2}{m}+\frac{\sigma_2^2}{m^2}\right)
			\right)\EE[\absnormal{\OX_t-\OY_t}^2] \\
		&\quad\, +\left(3\left(\frac{\varepsilon\lambda_2}{ m}+\frac{\sigma_2^2}{m^2}\right)+\frac{3\lambda_1\gamma}{(2m)^2}\right)\EE[\absnormal{y_\alpha(\rho_{Y,t})-\OY_t}^2].
	\end{split}
	\end{equation*}
	Recalling the computation~\eqref{eq:boundX-consensuspoint} yields the bound
	\begin{equation} \label{eq:boundY-consensuspoint}
		\EE[\absnormal{\OY_t-y_\alpha(\rho_{Y,t})}^2]
		\leq 2e^{-\alpha\underline \TE}\frac{\EE[\absnormal{\delta\OY_t}^2]}{\EE[\exp(-\alpha\TE(\OY_t))]}
		\leq 2e^{-\alpha\underline \TE}\frac{6\EE[\absnormal{\OY_t-\OX_t}^2]+3\EE[\absnormal{\delta\OX_t}^2]}{\EE[\exp(-\alpha\TE(\OY_t))]},
	\end{equation}
	where we inserted $\pm \OX_t$ and $\pm \EE[\OX_t]$ in the second step and used that $(a+b+c)^2\leq3(a^2+b^2+c^2)$ as well as Jensen's inequality.
	Combining the last two bounds and choosing $\varepsilon=(3\lambda_2)/\gamma$ we obtain~\eqref{eq:evolution_H_memory} as desired.
\end{proof}

\begin{remark} \label{remark:parameter_choice_memory}
	The exponential decay of $\EE[\CH(t)]$ it obtained by choosing the parameters of PSO in a manner which ensures the negativity of the prefactors of~$\EE[\absnormal{\OX_t-\EE[\OX_t]}^2]$ and~$\EE[\absnormal{\OX_t-\OY_t}^2]$ in Inequality~\eqref{eq:evolution_H_memory}.
	This may be achieved by choosing for any fixed time~$t$, given~$\alpha$ and arbitrary $\theta,\sigma_1,\sigma_2,\gamma>0$,
	\begin{equation*} %\label{eq:remark:parameter_choice_memory}
		\lambda_1 > \frac{3\sigma_1^2}{2\gamma}, \
		\lambda_2 > 6\max\left\{\frac{D_t^Y\lambda_1}{4},\frac{(1+D_t^Y)\sigma_2^2}{\gamma}\right\}, \
		\kappa > \frac{3\lambda_2^2(1+D_t^Y)}{\gamma\theta\lambda_1}, \
		\text{and} \ \ 
		m<\min\left\{\frac{\gamma\theta}{16\kappa},\frac{\lambda_1\gamma^2}{18D_t^Y\lambda_2^2}\right\},
	\end{equation*}
	where we abbreviate $D_t^Y=12e^{-\alpha\underline\TE}/\EE[\exp(-\alpha\TE(\OY_t))]$.
\end{remark}

In our main theorem on convergence of the PSO dynamics with memory mechanisms to the global minimizer~$x^*$ we again ensure that the parameter can be chosen once at initialization time.

\begin{thm} \label{thm:convergence_memory}
	Let $\TE$ satisfy Assumptions~\ref{asm:zero_global}--\ref{asm:Hessian} and let $(\OX_t,\OV_t)_{t\geq 0}$ be a solution to the nonlinear SDE~\eqref{eq:PSO_with_memory_mf}.
	Moreover, let us assume the well-preparation of the initial datum~$\OX_0$ and~$\OV_0$ in the sense that
	\begin{enumerate}[label=P\arabic*,labelsep=10pt,leftmargin=35pt]
		\item\label{asm:well_preparedness_parameters1_memory}
			$\mu_1>0$ with
			\begin{equation*}
				\mu_1:=\frac{(\lambda_1+2\lambda_2)\gamma}{(2m)^2}-\left(\frac{9\lambda_2^2}{\gamma m}+\frac{3\sigma_2^2}{m^2}+\frac{3\lambda_1\gamma}{4m^2}\right)\frac{12e^{-\alpha\underline\TE}}{\EE[\exp(-\alpha\TE(\OY_0))]},
			\end{equation*}
		\item\label{asm:well_preparedness_parameters2_memory}
			$\mu_2>0$ with
			\begin{equation*}
			\begin{split}
				\mu_2:=&\,\frac{(\lambda_1+\lambda_2)\gamma}{m^2}+\kappa\theta\left(\frac{3\lambda_1}{m}+\frac{\gamma^2}{m^2}\right)-\frac{8\kappa^2\gamma}{m}-\frac{\lambda_2^2\gamma}{2m^2\lambda_1}-\frac{3\sigma_1^2}{2m^2}\\
				&\qquad-\left(\frac{9\lambda_2^2}{\gamma m}+\frac{3\sigma_2^2}{m^2}\right)-\left(\frac{9\lambda_2^2}{\gamma m}+\frac{3\sigma_2^2}{m^2}+\frac{3\lambda_1\gamma}{(2m)^2}\right)\frac{24e^{-\alpha\underline\TE}}{\EE[\exp(-\alpha\TE(\OY_0))]},
			\end{split}
			\end{equation*}
		
		\item\label{asm:well_preparedness_initial_memory}
			it holds
%			\begin{equation*}
%				4\alpha\kappa e^{-\alpha\underline\TE}\left(\frac{C_\TE}{\chi}+\frac{2\alpha^2}{\chi}\right)\frac{2m}{3\lambda_1}\EE[\CH(0)]
%				+\frac{4}{\alpha}\kappa e^{-\alpha\underline \TE}\left(
%				2\EE[\absnormal{\nabla\TE(\OX_0)}^2]\frac{2}{\chi}+\frac{16C_{\TE}^2}{\chi^3}\EE[\CH(0)]\right)
%				<\frac{1}{2}\EE[\exp(-\alpha\TE(\OY_0))]
%			\end{equation*}
%			\begin{equation*}
%				%\qquad\quad
%				\alpha\kappa\frac{m}{\lambda_1}\left(\frac{C_\TE}{\chi}+\frac{2\alpha^2}{\chi}\right)\frac{\EE[\CH(0)]}{\EE[\exp(-\alpha\TE(\OY_0))]}
%				+\frac{6}{\alpha}\kappa \left(
%				\frac{1}{\chi}\frac{\EE[\absnormal{\nabla\TE(\OX_0)}^2]}{\EE[\exp(-\alpha\TE(\OY_0))]}+\frac{4C_{\TE}^2}{\chi^3}\frac{\EE[\CH(0)]}{\EE[\exp(-\alpha\TE(\OY_0))]}\right)
%				<\frac{3}{16e^{-\alpha\underline \TE}}
%			\end{equation*}
			\begin{equation*}
				%\qquad\quad
				\left(\frac{\alpha\kappa m}{\lambda_1\chi}\left(C_\TE+2\alpha^2\right)+\frac{24C_{\TE}^2\kappa}{\alpha\chi^3}\right)\frac{\EE[\CH(0)]}{\EE[\exp(-\alpha(\TE(\OY_0)-\underline \TE))]}
				+\frac{6\kappa}{\alpha\chi} \frac{\EE[\absnormal{\nabla\TE(\OX_0)}^2]}{\EE[\exp(-\alpha(\TE(\OY_0)-\underline \TE))]}
				<\frac{3}{32}
			\end{equation*}
			where
			\begin{equation*}
				\chi := \frac{2}{5}\frac{\min\{\gamma/(2m),\mu_1,\mu_2\}}{\big(\!\left(\gamma/(2m)\right)^2+1+3\lambda_1/m+2(\gamma/m)^2\big)}.
			\end{equation*}
	\end{enumerate}
	Then $\EE[\CH(t)]$ with $\CH$ as defined in Equation~\eqref{eq:H_memory} converges exponentially fast with rate~$\chi$ to $0$ as $t\to \infty$.
	Moreover, there exists  some $\tilde x$, which may depend on $\alpha$ and $f_0$, such that $\EE[\OX_t]\to \tilde x$ and $y_\alpha(\rho_{Y,t})\to \tilde x$ exponentially fast with rate~$\chi/2$ as $t\to \infty$.
	Eventually, for any given accuracy~$\varepsilon>0$, there exists $\alpha_0>0$, such that for all $\alpha>\alpha_0$, $\tilde x$ satisfies
	\begin{equation*}
		\TE(\tilde x)-\underline{\TE} \leq \varepsilon.
	\end{equation*}
	If $\TE$ additionally satisfies Assumption~\ref{asm:icp}, we additionally have $\abs{\tilde x-x^*}\leq\varepsilon^\nu/\eta$.
\end{thm}

\begin{remark} \label{remark:parameter_choice_memory2}
	By replacing $D_t^Y$ with $2D_0^Y$ in the parameter choices of Remark~\ref{remark:parameter_choice_memory}, the well-preparation of the parameters as in Conditions~\ref{asm:well_preparedness_parameters1_memory} and~\ref{asm:well_preparedness_parameters2_memory} can be ensured.
	
	In analogy to Remark~\ref{remark:parameter_choice2}, Condition~\ref{asm:well_preparedness_initial_memory} guarantees the well-preparation of the initialization.
\end{remark}

\begin{proof}[Proof of Theorem~\ref{thm:convergence_memory}]
	Let us define the time horizon 
	\begin{equation*}
		T := \inf\left\{t\geq 0:\EE[\exp(-\alpha\TE(\OY_t))] < \frac{1}{2} \EE[\exp(-\alpha\TE(\OY_0))]  \right\} \quad\text{with }\inf \emptyset=\infty.
	\end{equation*}
	Obviously, by continuity, $T>0$.
	We claim that $T=\infty$, which we prove by contradiction in the following.
	Therefore, assume $T<\infty$.
	Then, for $t\in[0,T]$, noting that $\EE[\exp(-\alpha\TE(\OY_t))] \geq \EE[\exp(-\alpha\TE(\OY_0))]/2$, we observe that the prefactors of $\EE[\absnormal{\OX_t-\EE[\OX_t]}^2]$ and $\EE[\absnormal{\OX_t-\OY_t}^2]$ in Lemma~\ref{lem:evolution_H_memory} are upper bounded by $-\mu_1$ and $-\mu_2$, respectively.
	Lemma~\ref{lem:evolution_H_memory} then provides an upper bound for the time derivative of the functional~$\EE[\CH(t)]$, 
	\begin{equation} \label{eq:evolution_H_2}
	\begin{split}
		\frac{d}{dt}\EE[\CH(t)]
		&\leq -\frac{\gamma}{2m}\EE[\absnormal{\OV_t}^2]-\mu_1\EE[\absnormal{\OX_t-\EE[\OX_t]}^2]-\mu_2\EE[\absnormal{\OX_t-\OY_t}^2]\\
		&\leq -\min\left\{\frac{\gamma}{2m},\mu_1,\mu_2\right\} \left(\EE[\absnormal{\OX_t-\EE[\OX_t]}^2] + \EE[\absnormal{\OV_t}^2] + \EE[\absnormal{\OX_t-\OY_t}^2]\right)\\
		&\leq -\frac{2}{5}\frac{\min\{\gamma/(2m),\mu_1,\mu_2\}}{\big(\!\left(\gamma/(2m)\right)^2+1+3\lambda_1/m+2\gamma^2/m^2\big)}\EE[\CH(t)]
		=: -\chi\EE[\CH(t)],
	\end{split}
	\end{equation}
	where we made use of the upper bound of~\eqref{eq:equivalence_H_memory} as in Lemma~\ref{lem:equivalence_H_memory} in the last inequality.
	The rate~$\chi$ is defined implicitly and it is straightforward to check that $0<\chi<\gamma/m$, where the positivity of $\chi$ follows from the well-preparation conditions~\ref{asm:well_preparedness_parameters1_memory} and~\ref{asm:well_preparedness_parameters2_memory} of the initialization.
	Gr\"onwall's inequality implies
	\begin{equation} \label{eq:decay_H_memory}
		\EE[\CH(t)] \leq
		%\EE[\CH(0)]\exp\left(\frac{2\min\{\mu,\gamma/m\}}{3\big(\!\left(\gamma/(2m)\right)^2+1\big)}t\right) = 
		\EE[\CH(0)]\exp(-\chi t).
	\end{equation}
	We now investigate the evolution of the functional~$\mc{Y}(t) := \EE[\exp(-\alpha\TE(\OY_t))]$.
	The It\^o-Doeblin formula yields
	\begin{equation} \label{eq:evolution_Y_memory}
	\begin{split}
		\frac{d}{dt}\mc{Y}(t)
		&= -\alpha\kappa\EE[\exp(-\alpha\TE(\OY_t))\left\langle\nabla\TE(\OY_t),(\OX_t-\OY_t)S^{\beta,\theta}(\OX_t,\OY_t)\right\rangle]\\
		&= -\alpha\kappa\EE[\exp(-\alpha\TE(\OY_t))\left\langle\nabla\TE(\OY_t)-\nabla\TE(\OX_t),(\OX_t-\OY_t)S^{\beta,\theta}(\OX_t,\OY_t)\right\rangle]\\
		&\quad\,-\alpha\kappa\EE[\exp(-\alpha\TE(\OY_t))\left\langle\nabla\TE(\OX_t),(\OX_t-\OY_t)S^{\beta,\theta}(\OX_t,\OY_t)\right\rangle]\\
		&\geq -4\alpha\kappa e^{-\alpha\underline \TE}C_\TE\EE[|\OX_t-\OY_t|^2]
		-4\alpha\kappa  e^{-\alpha\underline \TE}\EE[\absnormal{\nabla\TE(\OX_t)}\absnormal{\OX_t-\OY_t}],
	\end{split}
	\end{equation}
	where the last step follows from Cauchy-Schwarz inequality and uses Assumption~\ref{asm:Hessian} and $S^{\beta,\theta}(\OX_{t}, \OY_{t})<4$.
	Now firstly notice that $\EE[\absnormal{\nabla\TE(\OX_t)}\absnormal{\OX_t-\OY_t}] \leq  e^{(\chi/2)t}\alpha^2\EE[\absnormal{\OX_t-\OY_t}^2]+e^{-(\chi/2)t}/\alpha^2\EE[\absnormal{\nabla\TE(\OX_t)}^2 ]$ by Young's inequality.
	Secondly, using again Assumption~\ref{asm:Hessian} in the first inequality, we have
	\begin{equation*}
	\begin{split}
		\EE[\absnormal{\nabla\TE(\OX_t)}^2]
		&=\EE\left[\abs{\nabla\TE(\OX_0) + \int_0^t \nabla^2\TE(\OX_s) \OV_s\,ds}^2\right]
		\leq 2\EE[\absnormal{\nabla\TE(\OX_0)}^2] + 2C_{\TE}^2t \int_0^t \EE[\absnormal{\OV_s}^2]\,ds\\
		&\leq 2\EE[\absnormal{\nabla\TE(\OX_0)}^2] + 4C_{\TE}^2 t \int_0^t \EE[\CH(s)]\,ds
		\leq 2\EE[\absnormal{\nabla\TE(\OX_0)}^2] + 4C_{\TE}^2 t \EE[\CH(0)] \int_0^t \exp(-\chi s)\,ds\\
		&= 2\EE[\absnormal{\nabla\TE(\OX_0)}^2] + 4C_{\TE}^2 t \EE[\CH(0)] \frac{1}{\chi}\left(1-\exp(-\chi t)\right),
	\end{split}
	\end{equation*}
	where the next-to-last step uses the explicit bound in~\eqref{eq:decay_H_memory}.
	Using the two latter observations together with the fact that $\EE[\absnormal{\OX_t-\OY_t}^2]\leq2m/(3\lambda_1)\EE[\CH(t)]$ we can continue~\eqref{eq:evolution_Y_memory} as follows
	\begin{equation} \label{eq:evolution_Y_memory2}
	\begin{split}
		\frac{d}{dt}\mc{Y}(t)
		&\geq -4\alpha\kappa e^{-\alpha\underline \TE}\left(C_\TE+\exp\left(\frac{\chi}{2}t\right)\alpha^2\right)\frac{2m}{3\lambda_1}\EE[\CH(t)]
		-\frac{4}{\alpha}\kappa e^{-\alpha\underline \TE}\exp\left(-\frac{\chi}{2}t\right)\EE[\absnormal{\nabla\TE(\OX_t)}^2]\\
		&\geq -4\alpha\kappa e^{-\alpha\underline \TE}\left(C_\TE+\exp\left(\frac{\chi}{2}t\right)\alpha^2\right)\frac{2m}{3\lambda_1}\EE[\CH(0)]\exp(-\chi t)\\
		&\quad\, -\frac{4}{\alpha}\kappa e^{-\alpha\underline \TE}\exp\left(-\frac{\chi}{2}t\right)\left(2\EE[\absnormal{\nabla\TE(\OX_0)}^2] + 4C_{\TE}^2 t \EE[\CH(0)] \frac{1}{\chi}\left(1-\exp(-\chi t)\right)\right)\\
		&\geq -4\alpha\kappa e^{-\alpha\underline \TE}\left(C_\TE\exp\left(-\chi t\right)+\exp\left(-\frac{\chi}{2}t\right)\alpha^2\right)\frac{2m}{3\lambda_1}\EE[\CH(0)]\\
		&\quad\, -\frac{4}{\alpha}\kappa e^{-\alpha\underline \TE}\exp\left(-\frac{\chi}{2}t\right)\left(2\EE[\absnormal{\nabla\TE(\OX_0)}^2] + \frac{4C_{\TE}^2 t}{\chi} \EE[\CH(0)]\right).
	\end{split}
	\end{equation}
	By integrating~\eqref{eq:evolution_Y_memory2} we obtain for all $t\in[0,T]$
	\begin{equation*}
	\begin{split}
		\mc{Y}(t) \geq \mc{Y}(0)
		&-4\alpha\kappa e^{-\alpha\underline\TE}\left(\frac{C_\TE}{\chi}+\frac{2\alpha^2}{\chi}\right)\frac{2m}{3\lambda_1}\EE[\CH(0)]
		-\frac{4}{\alpha}\kappa e^{-\alpha\underline \TE}\left(
		2\EE[\absnormal{\nabla\TE(\OX_0)}^2]\frac{2}{\chi}+\frac{16C_{\TE}^2}{\chi^3}\EE[\CH(0)]\right).
	\end{split}
	\end{equation*}
	Recalling the definition of~$\mc{Y}$ and employing Condition~\ref{asm:well_preparedness_initial_memory}, we can deduce that for all $t\in[0,T]$ it holds
	\begin{equation*}
	\begin{split}
		\EE[\exp(-\alpha\TE(\OY_t))]
		&\geq \EE[\exp(-\alpha\TE(\OY_0))]
		-4\alpha\kappa e^{-\alpha\underline\TE}\left(\frac{C_\TE}{\chi}+\frac{2\alpha^2}{\chi}\right)\frac{2m}{3\lambda_1}\EE[\CH(0)]\\
		&\quad\,-\frac{4}{\alpha}\kappa e^{-\alpha\underline \TE}\left(
		2\EE[\absnormal{\nabla\TE(\OX_0)}^2]\frac{2}{\chi}+\frac{16C_{\TE}^2}{\chi^3}\EE[\CH(0)]\right)\\
		&> \frac{3}{4} \EE[\exp(-\alpha\TE(\OY_0))],
	\end{split}
	\end{equation*}
	which entails that there exists $\delta>0$ such that $\EE[\exp(-\alpha\TE(\OY_t))]\geq \EE[\exp(-\alpha\TE(\OY_0))]/2$ in $[T,T+\delta]$ as well, contradicting the definition of $T$ and therefore showing the claim $T=\infty$.

	\noindent
	As a consequence of~\eqref{eq:decay_H_memory} we have
	\begin{equation} \label{eq:decay_H_memory2}
		\EE[\CH(t)] \leq \EE[\CH(0)]\exp(-\chi t)
		\quad\text{and}\quad
		\EE[\exp(-\alpha\TE(\OY_t))]\geq\frac{1}{2}\EE[\exp(-\alpha\TE(\OY_0))]
	\end{equation}
	for all $t\geq0$.
	In particular, by means of Lemma~\ref{lem:equivalence_H_memory}, for a suitable generic constant $C>0$, we infer
	\begin{equation} \label{eq:decay_rest_memory}
		\EE[\absnormal{\OX_t-\EE[\OX_t]}^2] \leq C\exp(-\chi t),
		\quad
		\EE[\absnormal{\OV_t}^2] \leq C\exp(-\chi t),
		\quad\text{and}\quad
		\EE[\absnormal{\OX_t-\OY_t}^2] \leq C\exp(-\chi t).
	\end{equation}
	%$C_1 = 2(2m/\gamma)^2\EE[\CH(0)]$, $C_2 = 2\EE[\CH(0)]$, and $C_3 = 2m/(3\lambda_1)\EE[\CH(0)]$
	Moreover, with Jensen's inequality,
	\begin{equation*}
		\abs{\frac{d}{dt}\EE[\OX_t]}
		\leq \EE[\absnormal{\OV_t}]
		\leq C\exp\left(-\chi t/2\right) \to 0 \quad\text{as } t\to\infty,
	\end{equation*}
	showing that $\EE[\OX_t]\to \tilde x$ for some $\tilde x\in\RR^d$, which may depend on $\alpha$ and $f_0$.
	According to \eqref{eq:decay_rest_memory}, $\OX_t\to \tilde x$ as well as $\OY_t\to \tilde x$ in mean-square.
	Moreover, by reusing the inequality~\eqref{eq:boundY-consensuspoint} we get
	\begin{equation*}
	\begin{split}
		\EE[\absnormal{\OY_t-y_\alpha(\rho_{Y,t})}^2]
		&
		%\leq 2e^{-\alpha\underline \TE}\frac{6\EE[\absnormal{\OY_t-\OX_t}^2]+3\EE[\absnormal{\OX_t-\EE\OX_t}^2]}{\EE[\exp(-\alpha\TE(\OY_t))]}
		\leq 4e^{-\alpha\underline \TE}\frac{6\EE[\absnormal{\OY_t-\OX_t}^2]+3\EE[\absnormal{\OX_t-\EE\OX_t}^2]}{\EE[\exp(-\alpha\TE(\OY_0))]}%\\
		%&
		\leq C\exp(-\chi t)
	\end{split}
	\end{equation*}
	showing $y_\alpha(\rho_{Y,t})\to \tilde x$, since
	\begin{equation*}
		\absnormal{y_\alpha(\rho_{Y,t})-\tilde x}^2
		\leq 4\EE[\absnormal{y_\alpha(\rho_{Y,t})-\OY_t}^2]+4\EE[\absnormal{\OY_t-\OX_t}^2]+4\EE[\absnormal{\OX_t-\EE\OX_t}^2]+4\absnormal{\EE\OX_t-\tilde x}^2 \to 0 \quad\text{as } t\to\infty.
	\end{equation*}
	The remainder of the proof follows the lines of the proof of Theorem~\ref{thm:convergence}, replacing merely $\OX_t$ with $\OY_t$.
\end{proof}

%%%%%%%%%%%%%%%%%%%%%%%%%%%%%%%%%%%%%%%%
%%%%%%%%%%%%%%%%%%%%%%%%%%%%%%%%%%%%%%%%
%%%%%%%%%%%%%%%%%%%%%%%%%%%%%%%%%%%%%%%%
\section{A Wholistic Convergence Statement of PSO without Memory Effects} \label{sec:wholisticconvergence}

\noindent
In Sections~\ref{sec:PSO_without_memory} and~\ref{sec:PSO_with_memory} we analyzed the macroscopic behavior of PSO without and with memory effects in the mean-field regime.
For this purpose we introduced the with~\eqref{eq:PSO_with_memory} and~\eqref{eq:PSO_without_memory} associated self-consistent mono-particle processes~\eqref{eq:PSO_with_memory_mf} and~\eqref{eq:PSO_without_memory_mf}, for which we then established convergence guarantees under the in Theorems~\ref{thm:convergence} and~\ref{thm:convergence_memory} specified assumptions.
However, in order to be able to infer therefrom the optimization capabilities of the numerically implemented PSO method, a quantitative estimate on the approximation quality of the interacting particle system by the corresponding mean-field dynamics is necessary.

\subsection{On the Mean-Field Approximation of PSO without Memory Effects}
The following theorem provides a probabilistic quantitative estimate on the mean-field approximation for PSO without memory effects.
Notably, the result does not suffer from the curse of dimensionality.

\begin{thm} \label{thm:MFA}
	Let $T>0$, $f_0\in\CP_4(\RR^{2d})$ and let $N\in\NN$ be fixed.	
	Moreover, let $\TE$ obey Assumptions~\ref{asm:zero_global}--\ref{asm:Hessian}.
	We denote by $\big((X_{t}^i,V_{t}^i)_{t\geq0}\big)_{i=1,\dots,N}$ the solution to system~\eqref{eq:PSO_without_memory} and let $\big((\OX_{t}^i,\OV_{t}^i)_{t\geq0}\big)_{i=1,\dots,N}$ be $N$ independent copies of the solution to the mean-field dynamics~\eqref{eq:PSO_without_memory_mf}.
	Then it holds
	\begin{equation} \label{eq:thm:MFA:OmegaM}
		\PP\left(\Omega_M\right) 
		= \PP\left(\sup_{t\in[0,T]}\left[\frac{1}{N}\sum_{i=1}^N \,\max\Big\{\absnormal{X_{t}^i}^4 + \absnormal{V_{t}^i}^4, \absnormal{\OX_{t}^i}^4 + \absnormal{\OV_{t}^i}^4\Big\}\right]\leq M\right)
		\geq 1-\frac{2K}{M},
	\end{equation}
	where $K=K(\gamma/m, \lambda/m, \sigma/m, T, \TE)$ is a constant, which is in particular independent of $N$ and $d$.
	
	Furthermore, if the processes share the initial data as well as the Brownian motion paths~$(B^{i}_t)_{t\geq0}$ for all~$i=1,\dots,N$,
	then we have a probabilistic mean-field approximation of the form
	\begin{equation}
		\max_{i=1,\dots,N}\sup_{t\in[0,T]} \,\EE \left[\absnormal{X_{t}^i-\OX_{t}^i}^2 + \absnormal{V_{t}^i-\OV_{t}^i}^2 \,\Big|\, \Omega_M\right]
		\leq C_{\mathrm{MFA}}N^{-1}
	\end{equation}
	with a constant~$C_{\mathrm{MFA}}=C_{\mathrm{MFA}}(\alpha, \gamma/m, \lambda/m, \sigma/m, T, \TE,K,M)$, which is in particular independent of $N$ and $d$.
\end{thm}

%\begin{remark}
%	From 
%	\begin{equation}
%		\absnormal{X_{t}^i-\OX_{t}^i}^2 + \absnormal{V_{t}^i-\OV_{t}^i}^2
%		\leq \epsilon
%		\qquad \text{with probability of at least } 1-\big(\delta + \epsilon^{-1}C_{\mathrm{MFA}}N^{-1}\big).
%	\end{equation}
%\end{remark}

\begin{proof}%[Proof of Theorem~\ref{thm:MFA}]
	The proof is based on the arguments of~\cite[Section~3.3]{fornasier2021consensus} about the mean-field approximation of CBO.
	First we compute a bound for $\EE[\sup_{t\in[0,T]}\frac{1}{N}\sum_{i=1}^N\max\{\absnormal{X_{t}^i}^4 + \absnormal{V_{t}^i}^4,\absnormal{\OX_{t}^i}^4 + \absnormal{\OV_{t}^i}^4\}]$, which is then used to derive a mean-field approximation for PSO conditioned on the set $\Omega_M$ of uniformly bounded processes.
	
	\noindent\textit{Step 1:} 
	Using standard inequalities and Jensen's inequality allows to derive the bound
	\begin{equation} \label{eq:proof:MFA:part1_X}
	\begin{split}
		\EE\left[\sup_{t\in[0,T]}\absnormal{X_{t}^i}^4\right]
		&\lesssim
			\EE[\absnormal{X_{0}^i}^4]
			+ \EE\left[\sup_{t\in[0,T]}\abs{\int_0^t V_{s}^i \,ds}^4\right]
		\leq C\left(\EE[\absnormal{X_{0}^i}^4] + \EE\left[\int_0^T \absnormal{V_{s}^i}^4 \,ds\right]\right)
	\end{split}
	\end{equation}
	with $C=C(T)$.
	For the velocities $V_{t}^i$ we first note that
	\begin{equation} \label{eq:proof:MFA:part1_V}
	\begin{split}
		\EE\left[\sup_{t\in[0,T]}\absnormal{V_{t}^i}^4\right]
		\lesssim
			\EE[\absnormal{V_{0}^i}^4]
			&+ \left(\frac{\gamma}{m}\right)^{\!4}\EE\left[\sup_{t\in[0,T]}\abs{\int_0^t V_{s}^i \,ds}^4\right]
			+ \left(\frac{\lambda}{m}\right)^{\!4}\EE\left[\sup_{t\in[0,T]}\abs{\int_0^t \left(x_{\alpha}(\widehat{\rho}_{X,s}^N)-X_{s}^i\right) ds}^4\right]\\
			&+ \left(\frac{\sigma}{m}\right)^{\!4}\EE\left[\sup_{t\in[0,T]}\abs{\int_0^t D\left(x_{\alpha}(\widehat{\rho}_{X,s}^N)-X_{s}^i\right) dB_s^{i}}^4\right].
	\end{split}
	\end{equation}
	While the two middle terms on the right-hand side of~\eqref{eq:proof:MFA:part1_V} can be controlled as before by applying Jensen's inequality, the last term is treated as follows.
	Since $\int_0^t D\left(x_{\alpha}(\widehat{\rho}_{X,s}^N)-X_{s}^i\right) dB_s^{i}$ is a martingale we can apply the Burkholder-Davis-Gundy inequality~\cite[Chapter IV, Theorem 4.1]{RevuzYor1999martingales}, which gives
	\begin{equation}
	\begin{split}
		\EE\left[\sup_{t\in[0,T]}\abs{\int_0^t D\left(x_{\alpha}(\widehat{\rho}_{X,s}^N)-X_{s}^i\right) dB_s^{i}}^4\right]
		&\lesssim 
		\sup_{t\in[0,T]}\EE\left[\left(\int_0^t \abs{x_{\alpha}(\widehat{\rho}_{X,s}^N)-X_{s}^i}^2 ds\right)^2\right]\\
		&\leq C\EE\left[\int_0^T \abs{x_{\alpha}(\widehat{\rho}_{X,s}^N)-X_{s}^i}^4 ds\right],
	\end{split}
	\end{equation}
	where the latter step is again due to Jensen's inequality and with a constant $C=C(T)$.
	Utilizing these bounds allows to continue the inequality in~\eqref{eq:proof:MFA:part1_V} and to obtain
	\begin{equation} \label{eq:proof:MFA:part2_V}
	\begin{split}
		\EE\left[\sup_{t\in[0,T]}\absnormal{V_{t}^i}^4\right]
		&\leq C\left(
			\EE[\absnormal{V_{0}^i}^4]
			+ \EE\left[\int_0^T \absnormal{X_{s}^i}^4 + \absnormal{x_{\alpha}(\widehat{\rho}_{X,s}^N)}^4 + \absnormal{V_{s}^i}^4 \,ds\right]
		\right)
	\end{split}
	\end{equation}
	with $C=C(\gamma/m, \lambda/m, \sigma/m, T)$.
	Since according to \cite[Lemma~3.3]{carrillo2018analytical} it holds
	\begin{equation*}
		\absnormal{x_{\alpha}(\widehat{\rho}_{X,s}^N)}^2
		\leq \int \absnormal{x}^2 \frac{\omega_\alpha^\TE(x)}{\norm{\omega_\alpha^\TE}_{L_1(\widehat{\rho}_{X,s}^N)}}\,d\widehat{\rho}_{X,s}^N(x)
		\leq b_1 + b_2 \int \absnormal{x}^2 \,d\widehat{\rho}_{X,s}^N(y)
		= b_1 + b_2 \frac{1}{N} \sum_{i=1}^N \,\absnormal{X_{s}^i}^2
	\end{equation*}
	with $b_1=0$ and $b_2=e^{\alpha(\overline\TE-\underline\TE)}$ in the case that $\TE$ is bounded, and
	\begin{equation*}
		b_1 = R^2+b_2^2 \ \text{ and } \ b_2 = \frac{2L_\TE \max\{1,\absnormal{x^*}^2\}}{c_\TE}\left(1+\frac{1}{\alpha c_\TE R^2}\right)
	\end{equation*}
	in the case that $\TE$ satisfies the coercivity assumption~\ref{asm:growth}, we eventually obtain the upper bound
	\begin{equation} \label{eq:proof:MFA:part3_V}
	\begin{split}
		\EE\left[\sup_{t\in[0,T]}\absnormal{V_{t}^i}^4\right]
		&\leq C\left(
			1
			+ \EE[\absnormal{V_{0}^i}^4]
			+ \EE\left[\int_0^T \absnormal{X_{s}^i}^4 + \frac{1}{N} \sum_{j=1}^N \,\absnormal{X_{s}^j}^4 + \absnormal{V_{s}^i}^4 \,ds\right]
		\right)
	\end{split}
	\end{equation}
	with $C=C(\gamma/m, \lambda/m, \sigma/m, T, b_1, b_2)$.
	Adding up \eqref{eq:proof:MFA:part1_X} and \eqref{eq:proof:MFA:part3_V} yields
	\begin{equation} \label{eq:proof:MFA:XYV}
	\begin{split}
		&\EE\left[\sup_{t\in[0,T]}\absnormal{X_{t}^i}^4 + \absnormal{V_{t}^i}^4\right] \leq C\left(
			1
			+ \EE[\absnormal{X_{0}^i}^4 + \absnormal{V_{0}^i}^4]
			+ \EE\left[\int_0^T \absnormal{X_{s}^i}^4 + \frac{1}{N} \sum_{j=1}^N \,\absnormal{X_{s}^j}^2 + \absnormal{V_{s}^i}^4 \,ds\right]
		\right),
	\end{split}
	\end{equation}
	which, averaged over $i$, allows to derive the bound
	\begin{equation} \label{eq:proof:MFA:XYV_averaged}
	\begin{split}
		&\EE\left[\sup_{t\in[0,T]}\frac{1}{N}\sum_{i=1}^N \left( \absnormal{X_{t}^i}^4 + \absnormal{V_{t}^i}^4 \right)\right] \\
		&\qquad\qquad\leq C\left(
			1
			+ \EE\left[\frac{1}{N}\sum_{i=1}^N \left(\absnormal{X_{0}^i}^4 + \absnormal{V_{0}^i}^4\right)\right]
			+ \int_0^T \EE\left[\frac{1}{N}\sum_{i=1}^N \left( \absnormal{X_{s}^i}^4 + \absnormal{V_{s}^i}^4 \right)\right] ds
		\right).
	\end{split}
	\end{equation}
	An application of Gr\"onwall's inequality ensures that $\EE\sup_{t\in[0,T]}\big[\frac{1}{N}\sum_{i=1}^N \left( \absnormal{X_{t}^i}^4 + \absnormal{V_{t}^i}^4 \right)\big]$ is bounded independently of $N$ by some constant $K=K(\gamma/m, \lambda/m, \sigma/m, T, b_1, b_2)$.
	Note, that the constant~$K$ does in particular not depend on $N$ or $d$.
	With identical arguments for the processes $(\OX_{t}^i,\OV_{t}^i)$ an analogous bound can be obtained for $\EE\big[\sup_{t\in[0,T]}\frac{1}{N}\sum_{i=1}^N \big( \absnormal{\OX_{t}^i}^4 + \absnormal{\OV_{t}^i}^4 \big)\big]$.
	The first claim of the statement now follows from Markov's inequality.
	
	\noindent\textit{Step 2:} 
	We define the cutoff function
	\begin{equation}
		I_M(t) = \begin{cases}
			1, & \text{if } \frac{1}{N}\sum_{i=1}^N\max\big\{\absnormal{X_{s}^i}^4+\absnormal{V_{s}^i}^4,\absnormal{\OX_{s}^i}^4+\absnormal{\OV_{s}^i}^4\big\} \leq M \text{ for all } s\in[0,t],\\
			0, & \text{else,}
		\end{cases}
	\end{equation}
	which is a random variable adapted to the natural filtration and satisfying $\mathbbm{1}_{\Omega_M}\leq I_M(t)$ pointwise for all $t\in[0,T]$ as well as $I_M(t) = I_M(t)I_M(s)$ for all $s\in[0,t]$.
	Firstly, for the positions, by using standard inequalities and Jensen's inequality, we obtain the bound
	\begin{equation} \label{eq:proof:MFA2:X}
	\begin{split}
		\EE[\absnormal{X_{t}^i-\OX_{t}^i}^2I_M(t)]
		&\lesssim 
		\EE[\absnormal{X_{0}^i-\OX_{0}^i}^2] + \EE\left[\abs{\int_0^t \big(V_{s}^i-\OV_{s}^i\big)I_M(s)\,ds}^2\right]\\
		&\leq
		C\left(\EE[\absnormal{X_{0}^i-\OX_{0}^i}^2] + \int_0^t \EE\left[\absnormal{V_{s}^i-\OV_{s}^i}^2I_M(s)\right]ds\right)
	\end{split}
	\end{equation}
	with $C=C(T)$.
	Secondly, for the velocities we have
	\begin{equation} \label{eq:proof:MFA2:V}
	\begin{split}
		\EE[\absnormal{V_{t}^i-\OV_{t}^i}^2I_M(t)]
		&\lesssim 
		\EE[\absnormal{V_{0}^i-\OV_{0}^i}^2]
			+ \left(\frac{\gamma}{m}\right)^{\!2} \EE\left[\abs{\int_0^t\big(V_{s}^i-\OV_{s}^i\big)I_M(s)\,ds}^2\right]\\
		&\qquad\qquad	+ \left(\frac{\lambda}{m}\right)^{\!2} \EE\left[\abs{\int_0^t\Big(\big(x_{\alpha}(\widehat{\rho}_{X,s}^N)-X_{s}^i\big)-\big(x_{\alpha}(\rho_{X,s})-\OX_{s}^i\big)\Big)I_M(s)\,ds}^2\right]\\
		&\qquad\qquad	+ \left(\frac{\sigma}{m}\right)^{\!2} \EE\left[\abs{\int_0^t\Big(\absnormal{x_{\alpha}(\widehat{\rho}_{X,s}^N)-X_{s}^i}-\absnormal{x_{\alpha}(\rho_{X,s})-\OX_{s}^i}\Big)I_M(s)\,dB_s^{i}}^2\right] \\
		&\leq 
		C\left(\EE[\absnormal{V_{0}^i-\OV_{0}^i}^2]
			+ \int_0^t\EE\left[\absnormal{V_{s}^i-\OV_{s}^i}^2I_M(s)\right]ds\right.\\
		&\qquad\qquad+ \int_0^t\EE\left[\big(\absnormal{x_{\alpha}(\widehat{\rho}_{X,s}^N)-x_{\alpha}(\rho_{X,s})}^2+\absnormal{X_{s}^i-\OX_{s}^i}^2\big)I_M(s)\right]ds
	\end{split}
	\end{equation}	
	with $C=C(\gamma/m, \lambda/m, \sigma/m, T)$.
	In the first step of~\eqref{eq:proof:MFA2:V} we used that the processes $(X_{t}^i,V_{t}^i)$ and $(\OX_{t}^i,\OV_{t}^i)$ share the Brownian motion paths, and in the second both It\^o isometry and Jensen's inequality.
	In order to conclude, it remains to control the term $\EE\left[\absnormal{x_{\alpha}(\widehat{\rho}_{X,s}^N)-x_{\alpha}(\rho_{X,s})}^2I_M(s)\right]$.
	To do so, in analogy to the definition of $\widehat{\rho}_{X,s}^N$, let us denote by $\overline{\rho}_{X,s}^N$ the empirical measure associated with the processes~$\OX_{s}^i$, i.e., $\overline{\rho}_{X,s}^N:=\frac{1}{N}\sum_{i=1}^{N}\delta_{\OX_s^{i}}$.
	Then, by following the proofs of \cite[Lemma~3.2]{carrillo2018analytical} and \cite[Lemma~3.1]{FHPS1}, and exploiting the boundedness ensured by the multiplication with the random variable~$I_M(s)$, we obtain
	\begin{equation*}
	\begin{split}
		\EE\left[\absnormal{x_{\alpha}(\widehat{\rho}_{X,s}^N)-x_{\alpha}(\rho_{X,s})}^2I_M(s)\right]
		&\lesssim
		\EE\left[\absnormal{x_{\alpha}(\widehat{\rho}_{X,s}^N)-x_{\alpha}(\overline{\rho}_{X,s}^N)}^2I_M(s)\right] + \EE\left[\absnormal{x_{\alpha}(\overline{\rho}_{X,s}^N)-x_{\alpha}(\rho_{X,s})}^2I_M(s)\right]\\
		&\leq
		C\left(\frac{1}{N}\sum_{i=1}^N \,\EE[\absnormal{X_{s}^i-\OX_{s}^i}^2I_M(s)] + N^{-1}\right)\\
		&\leq
		C\left(\max_{i=1,\dots,N} \EE[\absnormal{X_{s}^i-\OX_{s}^i}^2I_M(s)] + N^{-1}\right)
	\end{split}
	\end{equation*}
	with $C=C(\alpha, L_\TE, c_\TE, \abs{x^*}, M, b_1, b_2)$.
	Inserting the latter into~\eqref{eq:proof:MFA2:V}, and adding up~\eqref{eq:proof:MFA2:X} and \eqref{eq:proof:MFA2:V} yields
	\begin{equation}
	\begin{split}
		&\EE\big[\big(\absnormal{X_{t}^i-\OX_{t}^i}^2 +  \absnormal{V_{t}^i-\OV_{t}^i}^2\big)I_M(t)\big]\\
		&\qquad\qquad \leq C\int_0^t\EE\left[\big(\absnormal{X_{s}^i-\OX_{s}^i}^2 + \absnormal{V_{s}^i-\OV_{s}^i}^2\big)I_M(s)\right] + \max_{j=1,\dots,N} \EE\left[\absnormal{X_{s}^j-\OX_{s}^j}^2I_M(s)\right] + N^{-1} ds
	\end{split}
	\end{equation}
	with $C=C(\alpha, \gamma/m, \lambda/m, \sigma/m, T, L_\TE, c_\TE, \abs{x^*}, M, b_1, b_2)$ and where we used that the processes $(X_{t}^i,V_{t}^i)$ and $(\OX_{t}^i,\OV_{t}^i)$ share the initial conditions.
	Lastly, by taking the maximum over $i$ on both sides we get
	\begin{equation}
	\begin{split}
		&\max_{i=1,\dots,N}\EE\big[\big(\absnormal{X_{t}^i-\OX_{t}^i}^2 +  \absnormal{V_{t}^i-\OV_{t}^i}^2\big)I_M(t)\big]\\
		&\qquad\qquad\leq C\int_0^t\EE\left[\max_{j=1,\dots,N} \EE\left[\big(\absnormal{X_{s}^j-\OX_{s}^j}^2+\absnormal{V_{s}^j-\OV_{s}^j}^2\big)I_M(s)\right] + N^{-1}\right]ds
	\end{split}
	\end{equation}
	with the $C$ from before.
	After recalling the definition of the conditional expectation, an application of Gr\"onwall's inequality concludes the proof.
\end{proof}

\begin{remark}
	While the first part of Theorem~\ref{thm:MFA} about the uniform in time boundedness of the empirical measures holds mutatis mutandis for the PSO dynamics with memory effects~\eqref{eq:PSO_with_memory} and~\eqref{eq:PSO_with_memory_mf}, it seems not straightforward to obtain the second part in this setting due to the way the memory effects are implemented in~\eqref{eq:PSO_with_memory_Y}.
	We leave the investigation of this extension to future research, in particular in regard to the question whether other implementations of memory effects might resolve this issue.
\end{remark}

%%%%%%%%%%%%%%%%%%%%%%%%%%%%%%%%%%%%%%%%
%%%%%%%%%%%%%%%%%%%%%%%%%%%%%%%%%%%%%%%%
\subsection{Convergence of PSO without Memory Effects in Probability}

Combining Theorem~\ref{thm:MFA} with the convergence analysis of the mean-field dynamics~\eqref{eq:PSO_without_memory} as described in Theorem~\ref{thm:convergence}, as well as a classical result about the numerical approximation of SDEs allows to obtain convergence guarantees with provable polynomial complexity for the numerical PSO method as stated in Theorem~\ref{thm:convergence_probability} below.
Let us, for the reader's convenience, recall from \cite[Section~6]{grassi2021mean} that a possible discretized version of the interacting particle system~\eqref{eq:PSO_without_memory} is given by 
\begin{subequations} \label{eq:PSO_without_memory_discrete}
\begin{align}
	X_{(k+1)\Delta t}^i   &= X_{k\Delta t}^i + {\Delta t}V_{(k+1)\Delta t}^i ,\\
	V_{(k+1)\Delta t}^i &= \left(\frac{m}{m+\Delta t \gamma}\right) V_{k\Delta t}^i
		    +\left(\frac{\Delta t\lambda}{m+\Delta t \gamma}\right)\left(x_{\alpha}(\widehat{\rho}_{X,k\Delta t}^N)-X_{k\Delta t}^i\right)\\
		  &  \qquad\qquad\qquad\qquad\ \ \; \, +\left(\frac{\sqrt{\Delta t}\sigma}{m+\Delta t \gamma}\right) D\!\left(x_{\alpha}(\widehat{\rho}_{X,k\Delta t}^N)-X_{k \Delta t}^i\right) B_{k\Delta t}^{i} \nonumber
\end{align}
\end{subequations}
for $k=0,\dots,K$ and where $\big((B_{k\Delta t}^{i})_{k=1,\dots,K-1}\big)_{i=1,\dots,N}$ are independent, identically distributed standard Gaussian random vectors in $\RR^d$.

\begin{thm} \label{thm:convergence_probability}
	Let $\epsilon_\mathrm{total}>0$ and $\delta\in(0,1/2)$.
	Then, under the assumptions of Theorems~\ref{thm:convergence} and~\ref{thm:MFA}, it holds for the discretized PSO dynamics~\eqref{eq:PSO_without_memory_discrete} that
	\begin{equation} \label{eq:thm:convergence_probability}
		\abs{\frac{1}{N}\sum_{i=1}^N X^i_{K\Delta t} - x^*}^2 \leq \epsilon_\mathrm{total}
	\end{equation}
	with probability larger than $1-\big(\delta + \epsilon_\mathrm{total}^{-1}(C_\mathrm{NA}(\Delta t)^m + C_\mathrm{MFA}N^{-1} + C_\mathrm{LLN}N^{-1} + \tilde\varepsilon + \varepsilon^{2\nu}/\eta^2)\big)$.
	Here, $m$ denotes the order of accuracy of the used discretization scheme.
	Moreover, besides problem dependent factors and the parameters of the method, the dependence of the constants is as follows.
	$C_\mathrm{NA}$ depends linearly on $d$ and $N$, and exponentially on $T^*$.
	$C_\mathrm{MFA}$ depends on exponentially on $\alpha$, $T^*$ and $\delta^{-1}$.
	$C_\mathrm{LLN}$ depends on the moment bound from Theorem~\ref{thm:wellposedness}.
	Lastly, $\tilde\varepsilon$ and $\varepsilon$ are chosen according to Theorem~\ref{thm:convergence}.
\end{thm}

\begin{proof}
	The overall error can be decomposed as
	\begin{equation} \label{eq:proof:convergence_probability:overall_error}
	\begin{split}
		&\EE\left[\abs{\frac{1}{N}\sum_{i=1}^N X^i_{K\Delta t} - x^*}^2\Bigg|\;\Omega_M\right]
		\lesssim
			\EE\left[\abs{\frac{1}{N}\sum_{i=1}^N \big(X^i_{K\Delta t} - X^i_{T^*}\big)}^2\right]
			+\EE\left[\abs{\frac{1}{N}\sum_{i=1}^N \big(X^i_{T^*} - \OX^i_{T^*}\big)}^2\Bigg|\;\Omega_M\right] \\
		&\qquad\qquad\qquad\qquad\qquad\qquad\qquad\;\;
			+\EE\left[\abs{\frac{1}{N}\sum_{i=1}^N \OX^i_{T^*} - \EE\big[\,\OX_{T^*}\!\big]}^2\right]
			+\absbig{\EE\big[\,\OX_{T^*}\!\big] - \tilde{x}}^2 + \abs{\tilde{x}-x^*}^2,
	\end{split}
	\end{equation}
	where we used that $\PP(\Omega_M) \geq (1-\delta) \geq 1/2$.
	By means of a classical result about the convergence of numerical schemes for SDEs~\cite{platen1999introduction}, the first term in~\eqref{eq:proof:convergence_probability:overall_error} can be bounded by $C_\mathrm{NA}(\Delta t)^m$.
	For the second term, Theorem~\ref{thm:MFA} gives the estimate $C_\mathrm{MFA}N^{-1}$.
	The third term can be bounded by $C_\mathrm{LLN}N^{-1}$ as a consequence of the law of large numbers.
	Eventually, Theorem~\ref{thm:convergence} allows us to choose $T^*=\mathcal{O}\big(\log(\tilde\varepsilon^{-1})/\chi\big)$ sufficiently large to reach any prescribed accuracy~$\tilde\varepsilon$ for the next-to-last term as well as $\varepsilon^{2\nu}/\eta^2$ for the last term by a suitable choice of $\alpha$.
	With these individual bounds we obtain 
	\begin{equation} \label{eq:proof:convergence_probability:overall_error_2}
	\begin{split}
		&\EE\left[\abs{\frac{1}{N}\sum_{i=1}^N X^i_{K\Delta t} - x^*}^2\Bigg|\;\Omega_M\right]
		\leq C_\mathrm{NA}(\Delta t)^m + C_\mathrm{MFA}N^{-1} + C_\mathrm{LLN}N^{-1} + \tilde\varepsilon + \varepsilon^{2\nu}/\eta^2.
	\end{split}
	\end{equation}
	It now remains to estimate the probability of the set~$K^N_{\epsilon_\mathrm{total}}\subset\Omega$, where Inequality~\eqref{eq:thm:convergence_probability} does not hold.
	By utilizing the conditional version of Markov's inequality together with the formerly established bound~\eqref{eq:proof:convergence_probability:overall_error_2}, we have
	\begin{equation}
	\begin{split}
		\PP\big(K^N_{\epsilon_\mathrm{total}}\big) 
		&= \PP\big(K^N_{\epsilon_\mathrm{total}} \cap \Omega_M\big) + \PP\big(K^N_{\epsilon_\mathrm{total}} \cap \Omega_M^c\big)\\
		&\leq \PP\big(K^N_{\epsilon_\mathrm{total}} \big|\, \Omega_M\big)\,\PP(\Omega_M) + \PP\big(\Omega_M^c\big)\\
		&\leq \frac{C_\mathrm{NA}(\Delta t)^m + C_\mathrm{MFA}N^{-1} + C_\mathrm{LLN}N^{-1} + \tilde\varepsilon + \varepsilon^{2\nu}/\eta^2}{\epsilon_\mathrm{total}} + \delta
	\end{split}
	\end{equation}
	for a sufficiently large choice of $M$ in~\eqref{eq:thm:MFA:OmegaM}.
\end{proof}

A result in this spirit was first presented for CBO in \cite[Theorem~14]{fornasier2021consensus} and is hereby extended to PSO.

%%%%%%%%%%%%%%%%%%%%%%%%%%%%%%%%%%%%%%%%
%%%%%%%%%%%%%%%%%%%%%%%%%%%%%%%%%%%%%%%%
%%%%%%%%%%%%%%%%%%%%%%%%%%%%%%%%%%%%%%%%
\section{Implementation of PSO and Numerical Results} \label{sec:numerics}

\noindent
The purpose of this section is twofold.
At first, an efficient implementation of PSO is provided, which is particularly suited for high-dimensional optimization problems arising in machine learning.
Its performance is then evaluated on a standard benchmark problem, where we investigate the influence of the parameters, and the training of a neural network classifier for handwritten digits.
Furthermore, several relevant implementational aspects are discussed, including the computational complexity and scalability, modifications inspired from simulated annealing and evolutionary algorithms, and the numerical stability of the method.

%%%%%%%%%%%%%%%%%%%%%%%%%%%%%%%%%%%%%%%%
%%%%%%%%%%%%%%%%%%%%%%%%%%%%%%%%%%%%%%%%
\subsection{An Efficient Implementation of PSO} \label{sec:implementation}

Let us stress that PSO is an extremely versatile, flexible and customizable optimization method, which can be regarded as a black-box optimizer.
As a zero-order method it is not reliant on the gradient information and can be even applied to discontinuous objectives, making it inevitably superior to first-order optimization methods in cases where derivatives are computationally infeasible.
However, also in machine learning applications, where gradient-based optimizers are considered the state of the art, PSO may be of particular interest in view of vanishing or exploding gradient phenomena.

Typical objective functions appearing in machine learning are of the form
\begin{equation} \label{eq:objective_function_ML}
	\TE(x) = \frac{1}{M} \sum_{j=1}^M \TE_j(x),
\end{equation}
where $\TE_j$ is usually the loss of the $j$th training sample.
In order to run the scheme~\eqref{eq:PSO_with_memory}, frequent evaluations of $\TE$ are necessary, which may be computationally intense or even prohibitive in some applications.

\noindent
\textbf{Computational complexity:}
Inspired by mini-batch gradient descent, the authors of~\cite{jin2020random} developed a random batch method for interacting particle systems, which was employed for CBO in~\cite{carrillo2019consensus}.
In the same spirit, we present with Algorithm~\ref{algorithm:PSOadvanced} a computationally efficient implementation of PSO.
\begin{algorithm}[!ht]
\caption{Particle swarm optimization~(PSO)}
\begin{algorithmic}[1]
		\floatname{algorithm}{Procedure}
		\renewcommand{\algorithmicrequire}{\textbf{Input:}}
		\renewcommand{\algorithmicensure}{\textbf{Output:}}
		\renewcommand{\algorithmicloop}{\textbf{while }}
		\renewcommand\algorithmicdo{}
		\renewcommand\algorithmicthen{}
		
	\Require{Objective~$\TE$ as in~\eqref{eq:objective_function_ML}, time horizon~$T$ or number of epochs~$\#epochs$, discrete time step size~$\Delta t$, batch sizes~$n_N$ and~$n_{\TE}$, parameters~$m,\gamma,\lambda_1,\lambda_2,\sigma_1,\sigma_2,\alpha,\beta,\theta$ and $\kappa$, number of particles~$N$, initialization~$f_0$}
	\Ensure{Approximation~$y_{\alpha}(\widehat{\rho}_{Y,T}^N)$ of the global minimizer~$x^*$ of $\TE$}
	\State
		Generate the particles' initial positions and velocities $(X^i_0,V^i_0)_{i=1,\dots,N}$ according to a common initial law~$f_0$.
		Initialize the local best positions $Y^i_0=X^i_0$.
	\State
		Ensure that $n_\TE$ divides $M$ and $n_N$ divides $N$.
	\State
		Convert $T$ into $\#epochs$ or vice versa via $T=\#epochs \, (M/n_\TE) (N/n_N) \Delta t$. Set $k=0$ and $epoch=1$.
	\Loop{$epoch\leq \#epochs$ and stopping criterion not fulfilled}
	\State
		Partition $\{1,\dots,M\}$ into batches $\CB_k^1,\dots,\CB_k^{M/n_\TE}$ of batch size $n_\TE$.
	\For $b=1,\dots,M/n_\TE$
		\State Define the objective function on this batch as\label{algorithm:objectivebatch}
		\begin{equation}
			\TE_{\mathit{batch}}(x) = \frac{1}{n_\TE} \sum_{j\in \CB_k^{b}}\TE_j(x).
		\end{equation}
		\vspace{-0.8em}
		\State \parbox[t]{\dimexpr\linewidth-2.9em}{Partition the particles, i.e., the set $\{1,\dots,N\}$, into disjoint sets $\CP_k^1,\dots,\CP_k^{\,N/n_N}$ of size $n_N$.\strut}
		\For $n=1,\dots,N/n_N$ \label{algorithm:partitionN}
			\State \parbox[t]{\dimexpr\linewidth-4.4em}{Compute the consensus point $y_{\alpha}(\widehat \rho_{Y,k\Delta t}^{N/n_N})$ according to Equation~\eqref{eq:momentaneous_consensus} with objective $\TE_{\mathit{batch}}$ from the particles in~$\CP_k^n$, i.e., with the empirical measure $\widehat\rho_{Y,k\Delta t}^{N/n_N}=\frac{1}{n_N}\sum_{i\in \CP_k^{n}}\delta_{Y_{k\Delta t}^{i}}$. \label{algorithm:consensus_point_comp}\strut}
				\vspace{-0.6em}
			\State \parbox[t]{\dimexpr\linewidth-4.4em}{Update either all particles (full update) or only the particles in the current batch~$\CP_k^n$ (partial update) according to a discretized version of the PSO dynamics~\eqref{eq:PSO_with_memory}.\label{algorithm:update}\strut}
			\If {$k>0$ and $\big|y_{\alpha}(\rho_{Y,k\Delta t}^N)-y_{\alpha}(\rho_{Y,(k-1)\Delta t}^N)\big|$ is too small despite stopping criterion not fulfilled}
				\State \parbox[t]{\dimexpr\linewidth-5.9em}{Perform an independent Brownian motion for the positions or velocities of all particles.\label{algorithm:random_step}\strut}
			\EndIf
			\State Set $k=k+1$.
		\EndFor 
	\EndFor
	\State \parbox[t]{\dimexpr\linewidth-1.4em}{Check the stopping criterion and \textbf{break} if fulfilled. If not, employ the optional strategies described at the end of Section~\ref{sec:implementation}, set $epoch=epoch+1$ and continue.\label{algorithm:optionalstrategies}\strut}
	\EndLoop
	\State Compute the consensus point $y_{\alpha}(\widehat{\rho}_{Y,T}^N)$ according to Equation~\eqref{eq:momentaneous_consensus} with objective $\TE$ from all particles, i.e., with $\widehat{\rho}_{Y,T}^N=\frac{1}{N}\sum_{i=1}^N\delta_{Y_{T}^{i}}$. \label{algorithm:consensus_point_comp_final}
\end{algorithmic}
\label{algorithm:PSOadvanced}
\end{algorithm}
The mini-batch idea is present on two different levels.
In line~\ref{algorithm:objectivebatch}, the objective is defined with respect to a batch of the training data of size $n_\TE$, meaning that only a subsample of the data is considered.
One epoch is completed after each data sample was seen exactly once, i.e., after $M/n_\TE$ steps.
At each step the consensus point~$y_\alpha$ has to be computed, for which~$\TE_{\mathit{batch}}$ needs to be evaluated for $N$ particles.
This still constitutes the most significant computational effort.
However, the mini-batch idea can be leveraged for a second time.
In the \textbf{for} loop in line~\ref{algorithm:partitionN} we partition the particles into sets of size~$n_N$ and perform the updates of line~\ref{algorithm:update} only for the $n_N$~particles in the respective subset.
Since this is embarrassingly parallel, a parallel machine can be deployed to reduce the runtime by up to a factor~$p$ (the number of available processors).
While this is referred to as partial update, alternatively, on a sequential architecture, a full update can be made at every iteration, requiring all $N$~particles to be updated in line~\ref{algorithm:update}.
Apart from lowering the required computing resources tremendously, these mini-batch ideas actually improve the stability of the method and the capability of finding good optima by introducing more stochasticity into the algorithm.

Concerning additional computational complexity due to the usage of memory effects, let us point out that, except for the required storage of the local (historical) best positions and their objective values, the update rule~\eqref{eq:local_best update} in combination with the partial update allows to include such mechanisms with no additional cost by keeping track of the objective values of the local best positions.
In such case, only one function call of each~$\TE_{\mathit{batch}}$ per epoch and per particle is necessary, which is optimal and coincides with PSO without memory effects or CBO.
A different realization of~\eqref{eq:PSO_with_memory_Y} might result in a higher cost.

\noindent
\textbf{Implementational aspects:}
A discretization of the SDE~\eqref{eq:PSO_with_memory} in line~\ref{algorithm:update} can be obtained for instance from a simple Euler-Maruyama or semi-implicit scheme~\cite{platen1999introduction,higham2001algorithmic}, see, e.g., \cite[Equation~(6.3)]{grassi2021mean}.
In our numerical experiments below Equation~\eqref{eq:local_best update} is used for updating the local best position, which corresponds to $\kappa=1/(2\Delta t)$, $\theta=0$, and $\beta=\infty$.
Furthermore, the friction parameter is set according to $\gamma=1-m$, which is a typical choice in the literature.
Let us also remark that a numerically stable computation of the consensus point in lines~\ref{algorithm:consensus_point_comp} and~\ref{algorithm:consensus_point_comp_final} for~$\alpha\gg1$ can be obtained by replacing $\TE_{\mathit{batch}}$ with $\TE_{\mathit{batch}}-\widetilde{\underline{\TE}}$, where $\widetilde{\underline{\TE}}:=\min_{i\in \CP_k^n} \TE_{\mathit{batch}}(Y^i_{k\Delta t})$.

\noindent
\textbf{Cooling and evolutionary strategies:}
The PSO algorithm can be divided into two phases, an exploration phase, where the energy landscape is searched coarsely, and a determination phase, where the final output is identified.
While the former benefits from small $\alpha$ and large diffusion parameters, in the latter, $\alpha\gg1$ guarantees the selection of the best solution.
A cooling strategy inspired from simulated annealing allows to start with moderate $\alpha$ and relatively large diffusion parameters $\sigma_1,\sigma_2$.
After each epoch, $\alpha$ is multiplied by $2$, while the diffusion parameters follow the schedule $\sigma = \sigma / \log(epoch+2)$ for $\sigma\in\{\sigma_1,\sigma_2\}$.
Such strategy was proposed in \cite[Section~4]{carrillo2019consensus} for CBO.
In order to further reduce computational complexity, the provable decay of the variance suggests to decrease the number of agents by discarding particles in accordance with the empirical variance.
A possible schedule for the number of agents proposed in \cite[Section~2.2]{fornasier2020consensus} is to set
$N_{epoch+1} = \big\lceil N_{epoch}\big((1-\mu)+\mu\widetilde\Sigma_{epoch}/\Sigma_{epoch}\big)\big\rceil$
%\begin{equation}
%	N_{epoch+1}
%	%= \left\lceil N_{epoch}\left(1+\mu\frac{\widetilde\Sigma_{epoch}-\Sigma_{epoch}}{\Sigma_{epoch}}\right)\right\rceil
%	= \left\lceil N_{epoch}\left((1-\mu)+\mu\frac{\widetilde\Sigma_{epoch}}{\Sigma_{epoch}}\right)\right\rceil
%\end{equation}
for $\mu\in[0,1]$ and where $\Sigma_{epoch}$ and $\widetilde\Sigma_{epoch}$ denote the empirical variances of the $N_{epoch}$ particles at the beginning and at the end of the current epoch.

%%%%%%%%%%%%%%%%%%%%%%%%%%%%%%%%%%%%%%%%
%%%%%%%%%%%%%%%%%%%%%%%%%%%%%%%%%%%%%%%%
\subsection{Numerical Experiments for the Rastrigin Function} \label{sec:numericalexperimentsRastrigin}
Before turning to high-dimensional optimization problems, let us discuss the parameter choices of PSO in moderate dimensions ($d=20$) at the example of the well-known Rastrigin benchmark function~$\TE(v)=\sum_{k=1}^d v_k^2 + \frac{5}{2}(1-\cos(2\pi v_k))$, which meets the requirements of Assumption~\ref{ass:assumptions} despite being highly non-convex with many spurious local optima.
To narrow down the number of tunable parameters, we let $\gamma=1-m$, choose $\alpha=100$, $N=100$, and update the local best position (if present) according to Equation~\eqref{eq:local_best update}, i.e., $\kappa=1/(2\Delta t)$, $\theta=0$, and $\beta=\infty$.
We moreover let $\lambda_2=1$ (or $\lambda=1$ for PSO without memory) and $\Delta t=0.01$, which are such that the algorithm either finds consensus or explodes within the time horizon~$T=100$ in all instances.
For simplicity we assume that $\sigma_1=\lambda_1\sigma_2$.
The algorithm is initialized with positions distributed according to $\mathcal{N}\big((2,\dots,2),4\Id\big)$ and velocities according to  $\mathcal{N}\big((0,\dots,0),\Id\big)$.
\begin{figure}[!ht]
\centering
\subcaptionbox{PSO without memory \label{fig:phasediagram_memory:0_drift1:0}}{\includegraphics[width=0.29\textwidth, trim=54 226 70 238,clip]{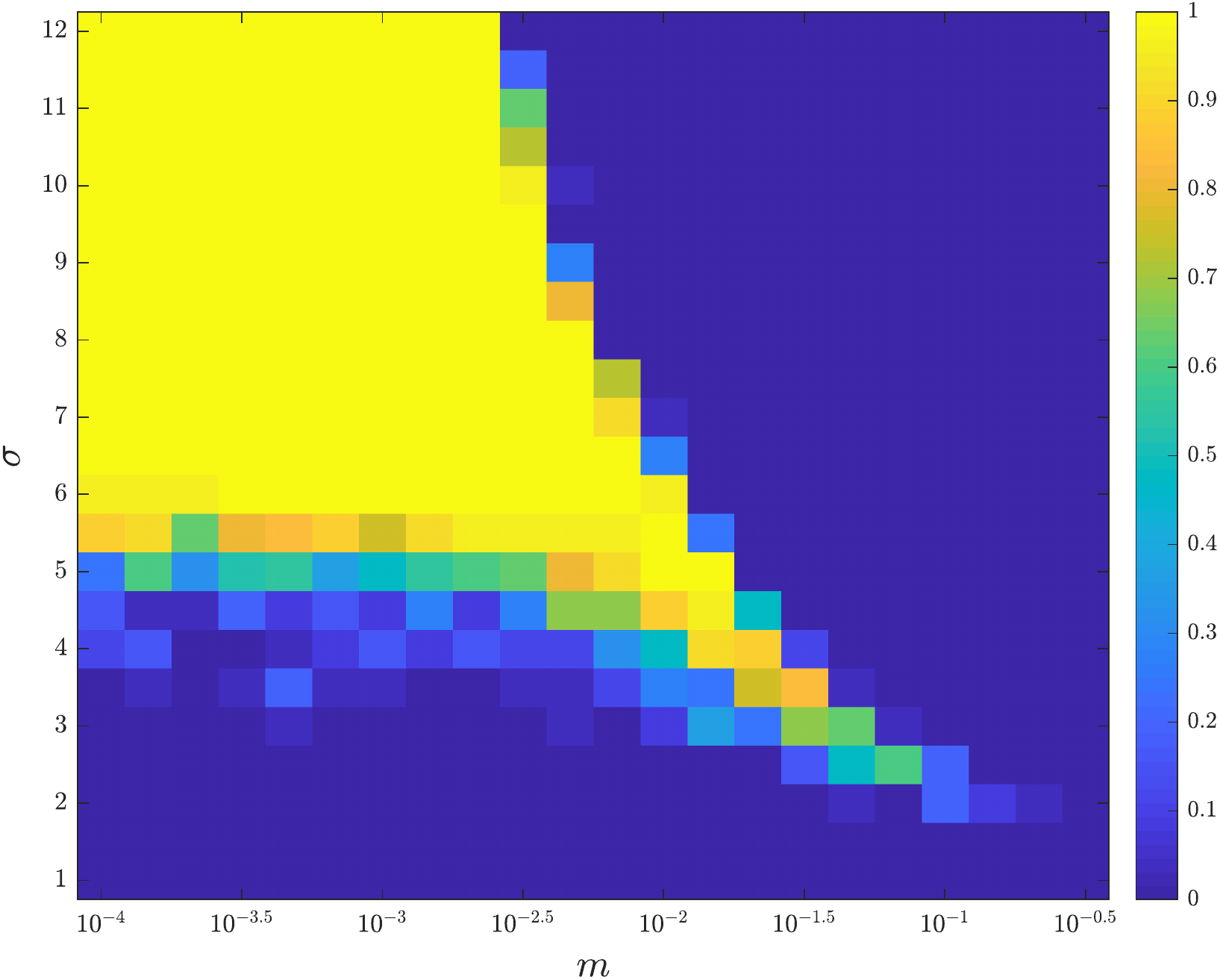}}
\hspace{1.2em}
\subcaptionbox{PSO with memory but no local best drift ($\lambda_1=0$, $\sigma_1=0$)\label{fig:phasediagram_memory:1_drift1:0}}{\includegraphics[width=0.29\textwidth, trim=54 226 70 238,clip]{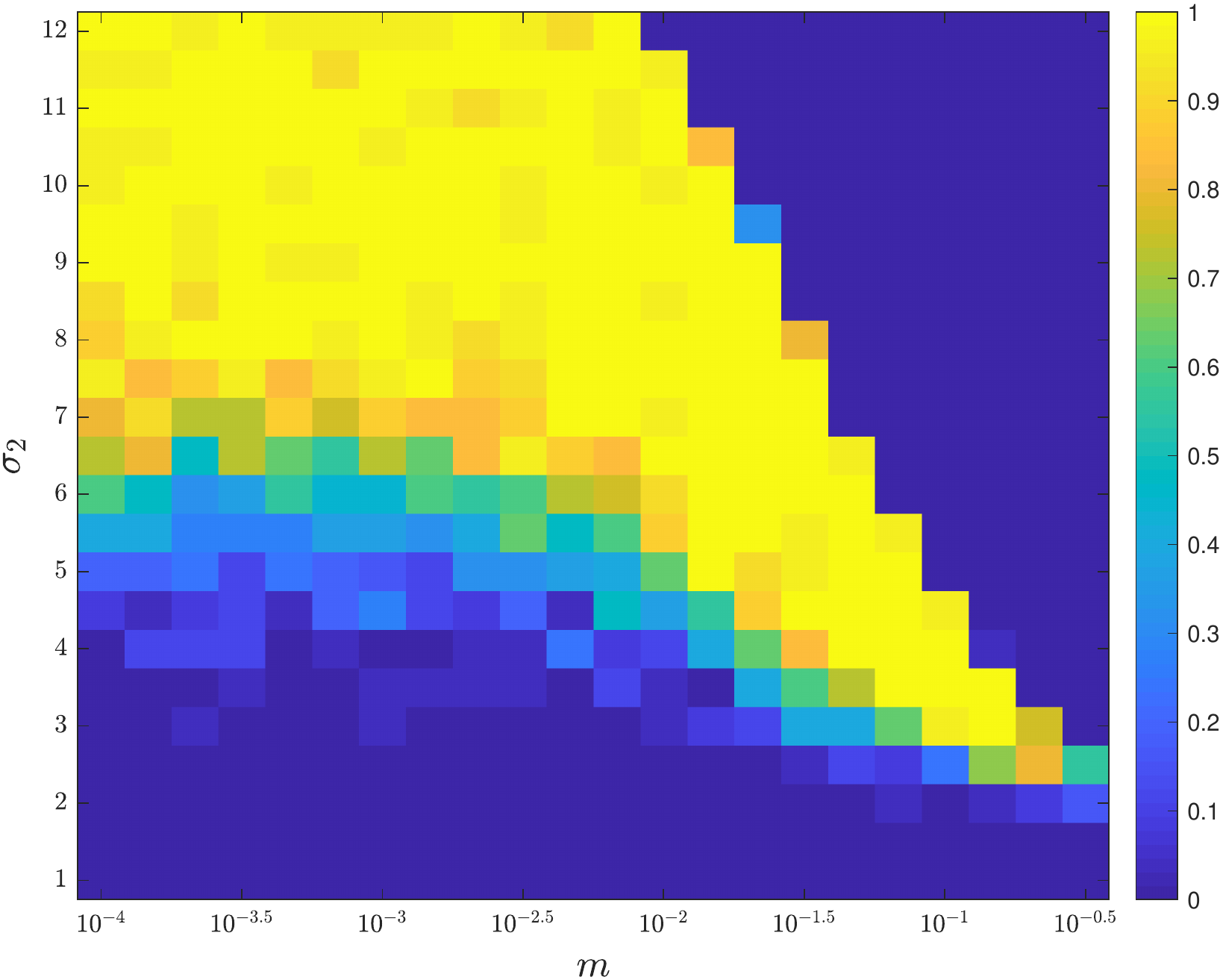}}
\hspace{1.2em}
\subcaptionbox{PSO with memory and local best drift ($\lambda_1=0.4$, $\sigma_1=0.4\sigma_2$)\label{fig:phasediagram_memory:1_drift1:0.25}}{\includegraphics[width=0.29\textwidth, trim=54 226 70 238,clip]{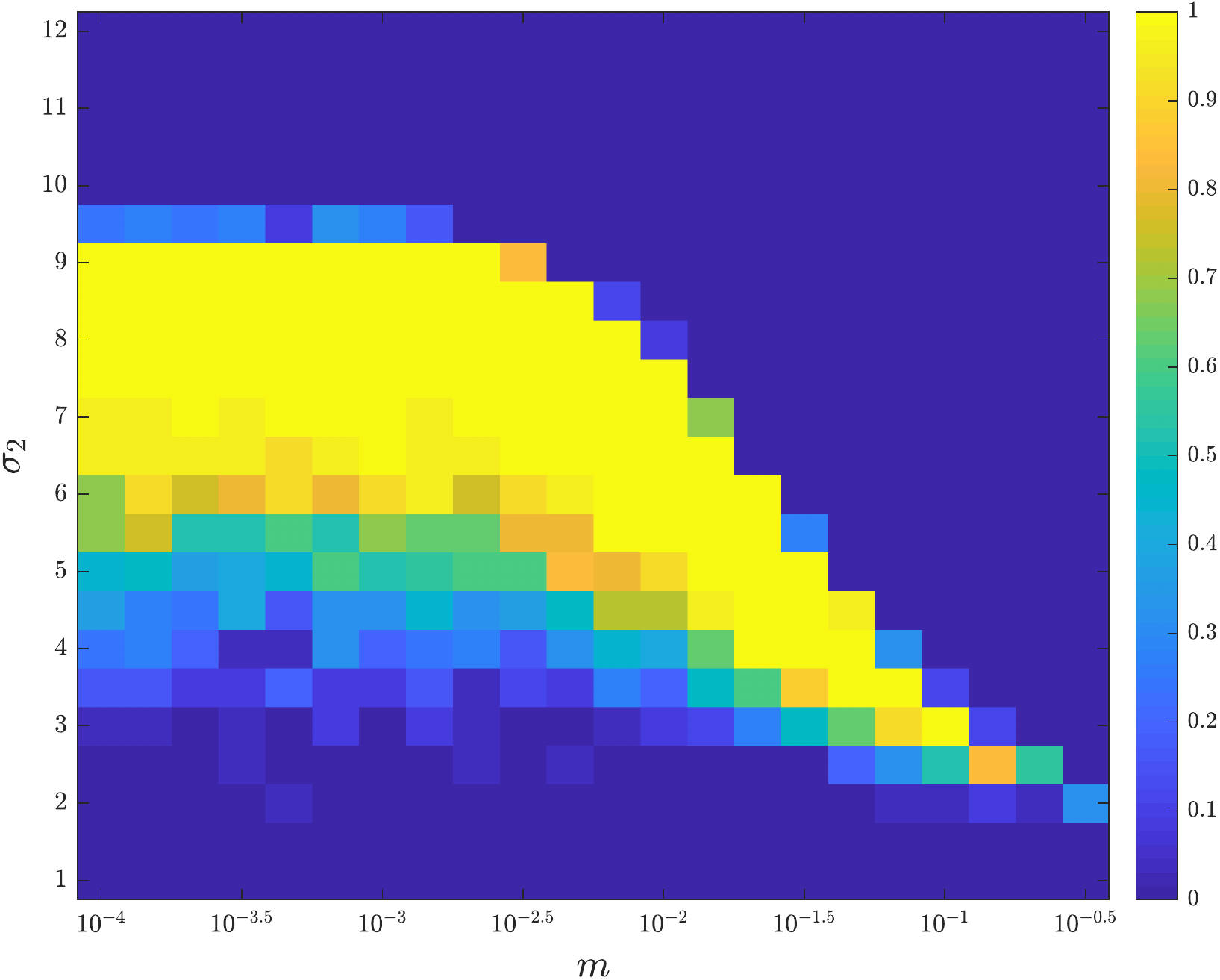}}
\caption{Phase transition diagrams comparing PSO without and with memory effects for different inertia parameters~$m$ and noise coefficients~$\sigma$ (PSO without memory) and~$\sigma_2$ (PSO with memory).
The empirical success probability is computed from $25$ runs and depicted by color from zero (blue) to one (yellow).}
%PSO is performed with $N=100$ particles, parameters $T=100$, $\Delta t=0.01$, $\gamma=1-m$, $\lambda_2=1$ and $\alpha=100$, and with a local best update as in Equation~\eqref{eq:local_best update}.
\label{fig:phasediagramsN100}
\end{figure}
In Figure~\ref{fig:phasediagramsN100} we depict the phase diagram comparing the success probability of PSO for different parameter choices of the inertia parameter~$m$ and the diffusion parameter~$\sigma$ or $\sigma_2$, respectively.
We observe that for $m$ fixed there is a noise threshold above which the dynamics explodes. 
While smaller~$m$ permit a larger flexibility in the used noise, they require an individual minimal noise level.
Further numerical experiments suggest however that increasing the number of particles~$N$ allows for a lower minimal noise level.
There are subtle differences between PSO without and with memory, but they are not decisive as in applications also confirmed by the numerical experiments in Section~\ref{sec:MLapplication}, \cite[Section~5.3]{grassi2020particle} as well as the survey paper~\cite[Section~6.3]{grassi2021mean}.
%\begin{figure}[!ht]
%\centering
%\subcaptionbox{PSO without memory \label{fig:phasediagram_memory:0_drift1:0}}{\includegraphics[width=0.29\textwidth, trim=54 226 70 238,clip]{img/memory0_drift1_0_N1000}}
%\hspace{1em}
%\subcaptionbox{PSO with memory but no local best drift ($\lambda_1=0$, $\sigma_1=0$)\label{fig:phasediagram_memory:1_drift1:0}}{\includegraphics[width=0.29\textwidth, trim=54 226 70 238,clip]{img/memory1_drift1_0_N1000}}
%\hspace{1em}
%\subcaptionbox{PSO with memory and local best drift ($\lambda_1=0.4$, $\sigma_1=0.4\sigma_2$)\label{fig:phasediagram_memory:1_drift1:0.25}}{\includegraphics[width=0.29\textwidth, trim=54 226 70 238,clip]{img/memory1_drift1_04_N1000}}
%\caption{Phase transition diagrams comparing PSO without and with memory effects for different inertia parameters~$m$ and noise coefficients in the setting of Figure~\ref{fig:phasediagramsN100} but with $N=1000$ particles.}
%\label{fig:phasediagramsN1000}
%\end{figure}

%%%%%%%%%%%%%%%%%%%%%%%%%%%%%%%%%%%%%%%%
%%%%%%%%%%%%%%%%%%%%%%%%%%%%%%%%%%%%%%%%
\subsection{A Machine Learning Application} \label{sec:MLapplication}
We now showcase the practicability of PSO as implemented in Algorithm~\ref{algorithm:PSOadvanced} at the example of a very competitive high-dimensional benchmark problem in machine learning, the classification of handwritten digits.
In what follows we train a shallow and a convolutional~NN~(CNN) classifier for the MNIST dataset~\cite{MNIST}.
Let us point out, that it is not our objective to challenge the state of the art by employing the most sophisticated model (deep CNNs achieve near-human performance of more than $99.5\%$ accuracy).
Instead, we want to demonstrate that PSO reaches results comparable to SGD with backpropagation, while at the same time relying exclusively on the evaluation of~$\TE$.

In our experiment we use NNs with architectures as depicted in Figure~\ref{fig:architectures}.
\begin{figure}[!ht]
\centering
\subcaptionbox{\label{fig:shallowNN} Dense shallow NN.}{\vspace{1.6em}\includegraphics[width=0.252\textwidth, trim=0 0 0 0,clip]{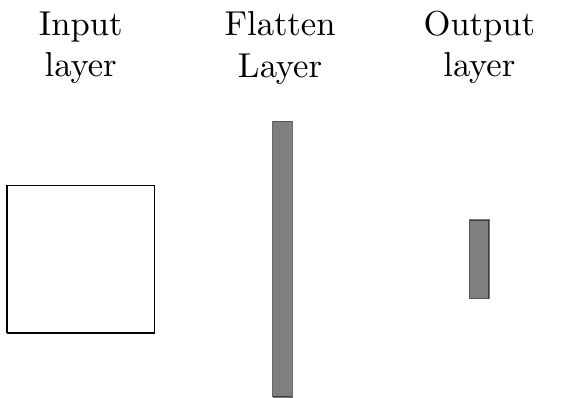}}
\hspace{2em}
\subcaptionbox{\label{fig:CNN} CNN with two convolutional and two pooling layers, and one dense layer.}{\vspace{0.04em}\includegraphics[width=0.66\textwidth, trim=0 0 0 0,clip]{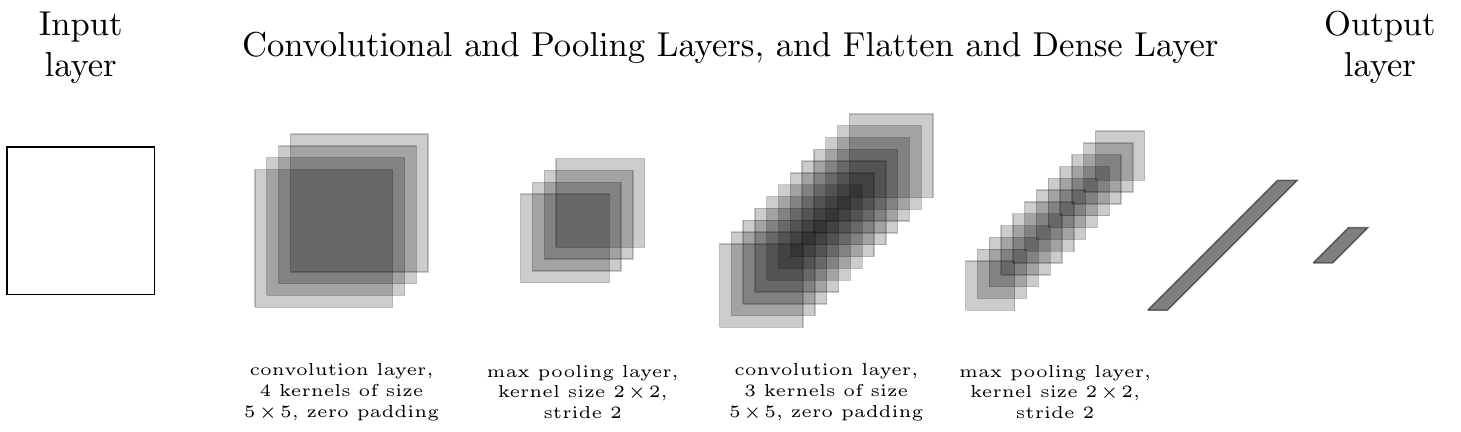}}
\caption{Architectures of the NNs used in the experiments of Section~\ref{sec:MLapplication}, cf.\@~\cite[Section~4]{fornasier2021convergence}.}
\label{fig:architectures}
\end{figure}
The input is a black-and-white image represented by a $(28\times28)$-dimensional matrix with entries between $0$ and $1$.
For the shallow NN (see Figure~\ref{fig:shallowNN}), the flattened image is passed through a dense layer~$\relu(W \cdot+b)$ with trainable weights~$W\in\RR^{10\times728}$ and bias $b\in\RR^{10}$.
Our CNN (see Figure~\ref{fig:CNN}) is similar to LeNet-1, cf.\@~\cite[Section~III.C.7]{lecun1998gradient}.
Each dense or convolution layer has a $\relu$ activation and is followed by a batch normalization layer to speed up the training process.
Eventually, the final layers of both NNs apply a softmax activation function allowing to interpret the $10$-dimensional output vector as a probability distribution over the digits.

We denote by~$\theta$ the trainable parameters of the NNs, which are $7850$ for the shallow NN and $2112$ for the CNN.
They are learned by minimizing~$\TE(\theta) = \frac{1}{M} \sum_{j=1}^M \ell(f_\theta(x^j),y^j)$,
where $f_\theta$ denotes the forward pass of the NN, $(x^j,y^j)$ the $j$th image-label tuple and $\ell$ the categorical crossentropy loss $\ell(\widehat{y},y)=-\sum_{k=0}^9 y_k \log \left(\widehat{y}_k\right)$.
The performance is measured by counting the number of successful predictions on a test set.
We use a train-test split of $60000$ training and $10000$ test images.
For our experiments we choose $\lambda_2=1$, $(\sigma_{2})_{initial}=\sqrt{0.4}$, $\alpha_{initial}=50$, $\Delta t=0.1$ and update the local best position according to Equation~\eqref{eq:local_best update}.
We use $N=100$ agents, which are initialized according to $\mathcal{N}\big((0,\dots,0)^T,\Id\big)$ in position and velocity.
The mini-batch sizes are $n_\TE=60$ and $n_N=100$ (consequently a full update is performed in line~\ref{algorithm:update}) and a cooling strategy is used in line~\ref{algorithm:optionalstrategies}.

Figure~\ref{fig:MNISTresults_differentmemory} reports the performances for different memory settings and fixed $m=0.2$, whereas Figure~\ref{fig:MNISTresults_differentm} depicts the results for different inertia parameters~$m$ in the case of PSO with memory but no memory drift.
\begin{figure}[!ht]
\centering
\subcaptionbox{PSO for three different memory settings: without memory (lightest lines); with memory but without local best drift, i.e., $\lambda_1=0$, $\sigma_1=0$ (line with intermediate opacity); with memory with local best drift $\lambda_1=0.4$, $\sigma_1=\lambda_1\sigma_2$ (darkest lines). \label{fig:MNISTresults_differentmemory}}{\includegraphics[width=0.47\textwidth, trim=28 248 8 266,clip]{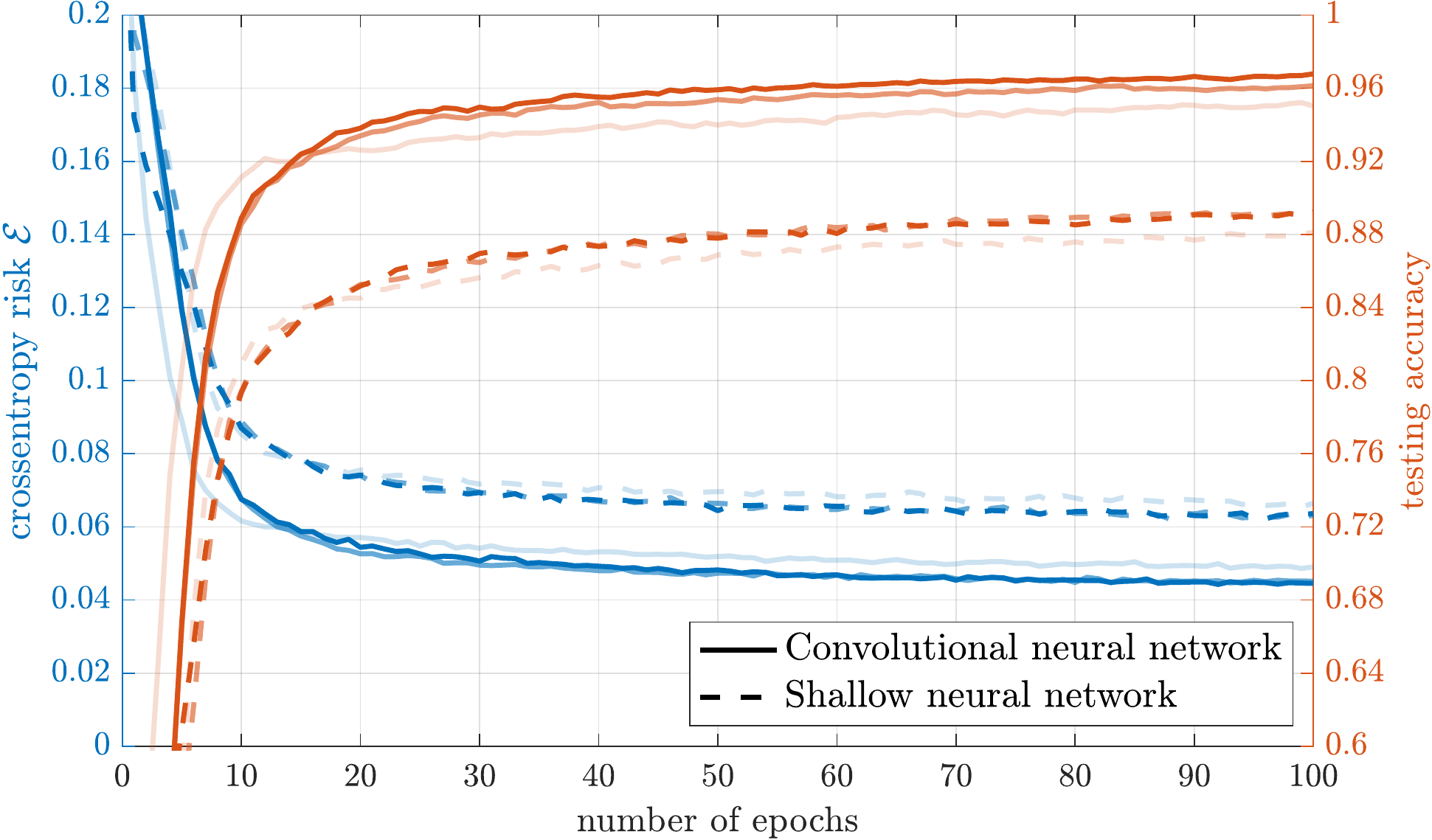}}
\hspace{2em}
\subcaptionbox{PSO with memory but without local best drift for three different inertia parameters $m\in\{0.1,0.2,0.4\}$ (increasing opacity for larger $m$).\\
Note that, for reference, the lines with intermediate opacity coincide with the ones of Figure~\ref{fig:MNISTresults_differentmemory}.\label{fig:MNISTresults_differentm}}{\includegraphics[width=0.47\textwidth, trim=28 248 8 266,clip]{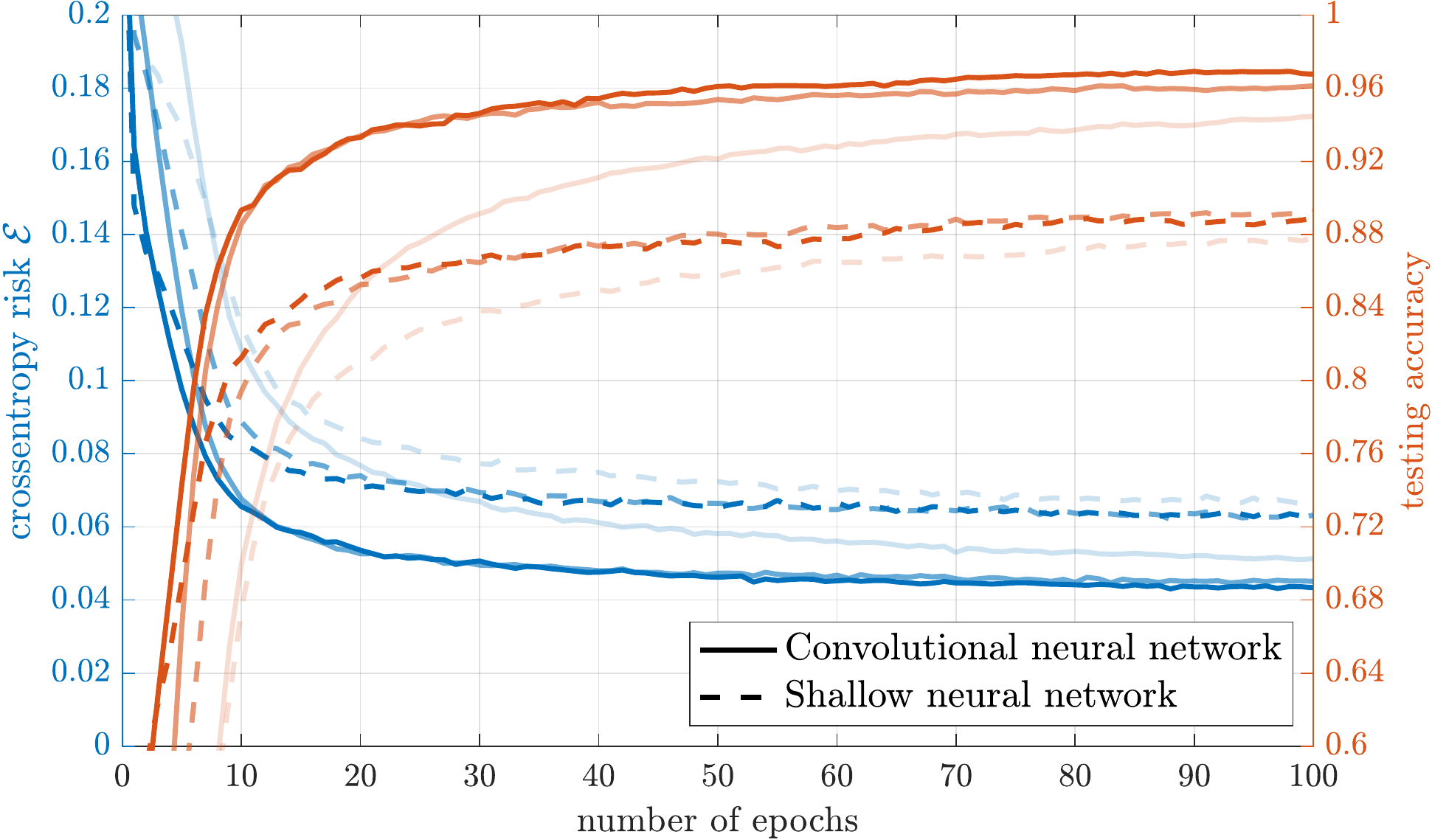}}
\caption{Comparison of the performances of a shallow (dashed lines) and convolutional (solid lines) NN with architectures as described in Figure~\ref{fig:architectures}, when trained with PSO as in Algorithm~\ref{algorithm:PSOadvanced}.
Depicted are the accuracies on a test dataset~(orange lines) and the values of the objective function~$\TE$~(blue lines) evaluated on a random sample of the training set of size $10000$.}
\label{fig:MNISTresults}
\end{figure}
For the shallow NN, we obtain a test accuracy of above $89\%$, whereas the CNN achieves almost $97\%$.
To put those numbers into perspective, when trained with SGD, a similar performance for the shallow NN, see \cite[Figure~7]{carrillo2019consensus}, and a benchmark accuracy of $98.3\%$ for a comparable CNN, cf.\@~\cite[Figure~9]{lecun1998gradient}, are reached.
As can be seen from Figure~\ref{fig:MNISTresults_differentmemory}, the usage of the local best positions when computing the consensus point significantly improves the performance.
The advantage of having an additional drift towards the local best position is less pronounced.
Regarding the inertia parameter~$m$ in Figure~\ref{fig:MNISTresults_differentm}, our numerical results suggest that larger $m$ yield faster convergence.

%In consideration of these results, we perceive PSO as a versatile and competitive black-box optimizer for machine learning tasks, which admits provable convergence guarantees.

%%%%%%%%%%%%%%%%%%%%%%%%%%%%%%%%%%%%%%%%
%%%%%%%%%%%%%%%%%%%%%%%%%%%%%%%%%%%%%%%%
%%%%%%%%%%%%%%%%%%%%%%%%%%%%%%%%%%%%%%%%
\section{Conclusions} \label{sec:conclusion}

\noindent
In this paper we prove the convergence of PSO without and with memory effects to a global minimizer of a possibly nonconvex and nonsmooth objective function in the mean-field sense. %Particle Swarm Optimization~(PSO)
Our analysis holds under a suitable well-preparation condition about the initialization and comprises a rich class of objectives which in particular includes functions with multiple global minimizers.
For PSO without memory effects we furthermore quantify how well the mean-field dynamics approximates the interacting finite particle dynamics, which is implemented for numerical experiments.
Since in particular the latter results does not suffer from the curse of dimensionality, we thereby prove that the numerical PSO method has polynomial complexity.
%We conjecture that the well-preparation condition is an artifact of our proof technique and can be removed by the argument of~\cite{fornasier2021convergence,fornasier2021convergence}.
With this we contribute to the completion of a mathematically rigorous understanding of PSO.
Furthermore, we propose a computationally efficient and parallelizable implementation and showcase its practicability by training a CNN reaching a performance comparable to stochastic gradient descent.

It remains an open problem to extend the mean-field approximation result to the variant of PSO with memory effects or, alternatively, to devise an implementation of such effects compatible with the used proof technique.
Moreover, we also leave a more thorough understanding of the influence of the parameters as well as the influence of memory effects for future, more experimental research.

Finally, we believe that the analysis framework of this and prior works on CBO~\cite{pinnau2017consensus,carrillo2018analytical,fornasier2021consensus} motivates to investigate also other renowned metaheuristic algorithms through the lens of a mean-field limit.

%%%%%%%%%%%%%%%%%%%%%%%%%%%%%%%%%%%%%%%%
%%%%%%%%%%%%%%%%%%%%%%%%%%%%%%%%%%%%%%%%
%%%%%%%%%%%%%%%%%%%%%%%%%%%%%%%%%%%%%%%%
\section*{Acknowledgements}

\noindent
H.H.\@~is partially supported by the Pacific Institute for the Mathematical Sciences~(PIMS) postdoctoral fellowship. J.Q.\@~is partially supported  by the National Science and Engineering Research Council of Canada~(NSERC) and by the start-up funds from the University of Calgary.
K.R.\@~acknowledges the financial support from the Technical University of Munich -- Institute for Ethics in Artificial Intelligence~(IEAI).
The authors gratefully acknowledge the compute and data resources provided by the Leibniz Supercomputing Centre~(LRZ).

%%%%%%%%%%%%%%%%%%%%%%%%%%%%%%%%%%%%%%
\bibliographystyle{abbrv}
\bibliography{conPSO}
%%%%%%%%%%%%%%%%%%%%%%%%%%%%%%%%%%%%%%

%\bibliography{bounded}
%\bibliographystyle{abbrv}

\end{document}